\documentclass[12pt]{amsart}

\usepackage{amsmath}
\usepackage{tabu}
\usepackage{tikz}
\usepackage[margin=1in]{geometry}
\usepackage[matrix,arrow,curve,frame]{xy}
\usepackage{hyperref}
\usepackage{amsthm}
\usepackage{amsfonts}
\usepackage{amssymb}
\usepackage{wasysym}
\usepackage{mathrsfs}
\usepackage{mathtools}

\definecolor{my-linkcolor}{rgb}{0.75,0,0}
\definecolor{my-citecolor}{rgb}{0.1,0.57,0}
\definecolor{my-urlcolor}{rgb}{0,0,0.75}
\hypersetup{
	colorlinks,
	linkcolor={my-linkcolor},
	citecolor={my-citecolor},
	urlcolor={my-urlcolor}
}

\title[Direct limit completions]{Direct limit completions of vertex tensor categories}
 \author{Thomas Creutzig, Robert McRae and Jinwei Yang}
\date{}

\address{(T. C. and J. Y.) Department of Mathematical and Statistical Sciences, University of Alberta, Edmonton, Alberta T6G 2R3, Canada}

\email{creutzig@ualberta.ca, jinwei2@ualberta.ca}

\address{(R. M.) Yau Mathematical Sciences Center, Tsinghua University, Beijing 100084, China}
 \email{rhmcrae@tsinghua.edu.cn}
 \subjclass[2010]{Primary 17B69, 18D10, 81R10}

\newtheorem{thm}{Theorem}[section]
\newtheorem{cor}[thm]{Corollary}
\newtheorem{lem}[thm]{Lemma}
\newtheorem{prop}[thm]{Proposition}
\theoremstyle{definition}\newtheorem{defi}[thm]{Definition}
\theoremstyle{definition}\newtheorem{rem}[thm]{Remark}
\theoremstyle{definition}\newtheorem{exam}[thm]{Example}
\theoremstyle{definition}\newtheorem{assum}[thm]{Assumption}

\newcommand{\cY}{\mathcal{Y}}

\newcommand{\cV}{\mathcal{V}}
\newcommand{\cA}{\mathcal{A}}
\newcommand{\cR}{\mathcal{R}}
\newcommand{\cM}{\mathcal{M}}

\newcommand{\cF}{\mathcal{F}}

\newcommand{\cC}{\mathcal{C}}
\newcommand{\cG}{\mathcal{G}}
\newcommand{\cW}{\mathcal{W}}
\newcommand{\til}{\widetilde}

\newcommand{\CC}{\mathbb{C}}
\newcommand{\ZZ}{\mathbb{Z}}
\newcommand{\NN}{\mathbb{N}}
\newcommand{\RR}{\mathbb{R}}
\newcommand{\fA}{\mathfrak{A}}

\newcommand{\Id}{\mathrm{Id}}

\newcommand{\tens}{\boxtimes}
\newcommand{\vac}{\mathbf{1}}

\newcommand{\ind}{\mathrm{Ind}}
\DeclareMathOperator{\im}{Im}

\DeclareMathOperator{\rep}{Rep}
\let\ker\relax
\let\hom\relax
\DeclareMathOperator{\ker}{Ker}
\DeclareMathOperator{\hom}{Hom}
\DeclareMathOperator{\Endo}{End}

\begin{document}
\bibliographystyle{alpha}

\numberwithin{equation}{section}

 \begin{abstract}
  We show that direct limit completions of vertex tensor categories inherit vertex and braided tensor category structures, under conditions that hold for example for all known Virasoro and  affine Lie algebra tensor categories. A consequence is that the theory of vertex operator (super)algebra extensions also applies to infinite-order extensions. As an application, we relate  rigid and non-degenerate vertex tensor categories of certain modules for both the affine vertex superalgebra of $\mathfrak{osp}(1|2)$ and the $N=1$ super Virasoro algebra to categories of Virasoro algebra modules via certain cosets.
 \end{abstract}

\maketitle

\tableofcontents

\section{Introduction}

A central problem in the representation theory of vertex operator algebras is existence of vertex tensor category structure \cite{HL-vrtx-tens} on module categories; this in particular implies existence of braided tensor category structure. Unfortunately, the assumptions on the module category that are needed to apply the (logarithmic) vertex tensor category theory of Huang-Lepowsky-Zhang \cite{HLZ1}-\cite{HLZ8} are rather extensive and can be difficult to verify in practice. For example, only recently has vertex tensor category structure been established on the category of $C_1$-cofinite modules for the universal Virasoro vertex operator algebra at arbitrary central charge \cite{CJORY}.

One way to avoid direct verification of the vertex tensor category theory assumptions is to use extension theory. Suppose a vertex operator algebra $A$ is an object in a category $\cC$ of modules for a subalgebra $V$ that is already known to have vertex tensor category structure. Then $A$ is a commutative algebra in the braided tensor category $\cC$ \cite{HKL} and tensor-categorical methods can be used in showing that the category of $A$-modules which are objects of $\cC$ (when viewed as $V$-modules) also has vertex and braided tensor category structures \cite{CKM-ext}. This method works well when $A$ is a finite-order extension of a nice vertex operator subalgebra $V$.

However, it often occurs that $A$ is not actually an object of $\cC$ but is rather, for example, an infinite direct sum of modules in $\cC$. For example, the singlet $W$-algebras \cite{Ad} are infinite direct sums of irreducible $C_1$-cofinite modules for their Virasoro subalgebras. The usual tensor-categorical method for handling such examples is to realize $A$ as an object in the ind-completion of $\cC$ (see for example \cite[Remark 7.8.2]{EGNO}); if $\cC$ is semisimple, this agrees with the direct sum completion $\cC_\oplus$ of \cite{AR}. The ind-completion $\ind(\cC)$, as constructed in references such as \cite{KS}, is the full subcategory of direct limits in the category of contravariant functors from $\cC$ to sets, into which $\cC$ embeds by the Yoneda Lemma. Or, when $\cC$ is a $\CC$-linear abelian category, we can use the category of contravariant functors from $\cC$ to $\CC$-vector spaces. But unfortunately, this ind-completion is not suitable for studying vertex operator algebra extensions. The problem is that if we want the category of $A$-modules in $\ind(\cC)$ to have the correct vertex algebraic tensor category structure, then the underlying tensor category structure on $\ind(\cC)$ must also be vertex algebraically natural. In particular, objects of $\ind(\cC)$ must be genuine $V$-modules of a sort, and the tensor product on $\ind(\cC)$ must satisfy the intertwining operator universal property of \cite[Definition 4.15]{HLZ3}.

In this paper, we construct an alternative direct limit completion (which we also denote $\ind(\cC)$) of a vertex tensor category $\cC$ that is suitable for studying representations of an extension algebra $A$. Especially, we show that under suitable conditions, $\ind(\cC)$ has vertex and braided tensor category structures as described in the tensor category theory of \cite{HLZ1}-\cite{HLZ8}. Instead of realizing $\ind(\cC)$ within the category of contravariant functors from $\cC$ to vector spaces, we use the ambient category of weak $V$-modules: we define $\ind(\cC)$ to be the full subcategory of weak modules whose objects are isomorphic to direct limits of inductive systems into $\cC$. Equivalently, $\ind(\cC)$ is the category of generalized $V$-modules (in the sense of \cite[Definition 2.12]{HLZ1}) which are the unions of submodules that are objects of $\cC$. Now here is our main theorem:

\begin{thm}\label{thm:main_thm}
 Let $V$ be a vertex operator algebra, $\cC$ a category of grading-restricted generalized $V$-modules, and $\ind(\cC)$ the category of generalized $V$-modules that are the unions of their $\cC$-submodules. Assume also that:
 \begin{enumerate}
  \item The vertex operator algebra $V$ is an object of $\cC$.
  \item The category $\cC$ is closed under submodules, quotients, and finite direct sums.
  \item Every module in $\cC$ is finitely generated.
  \item The category $\cC$ admits $P(z)$-vertex and braided tensor category structures as described in Section \ref{sec:vertex_tensor}.
  \item For any intertwining operator $\cY$ of type $\binom{X}{W_1\,W_2}$ where $W_1$, $W_2$ are modules in $\cC$ and $X$ is a generalized $V$-module in $\ind(\cC)$, the submodule $\im\cY\subseteq X$ is an object of $\cC$.
 \end{enumerate}
Then $\ind(\cC)$ also admits $P(z)$-vertex and braided tensor category structures, as described in Section \ref{sec:vertex_tensor}, that extend those on $\cC$.
\end{thm}

Much of the construction of the tensor category $\ind(\cC)$ in the proof of this theorem goes quite naturally. For example, we realize the tensor product $X_1\widehat{\tens}X_2$ of two direct limits in $\ind(\cC)$ as a direct limit of tensor products $W_1\tens W_2$, where $W_1\subseteq X_1$ and $W_2\subseteq X_2$ are $\cC$-submodules. The most difficult part of the construction is the associativity isomorphisms: since $X_1\widehat{\tens}(X_2\widehat{\tens} X_3)$ and $(X_1\widehat{\tens}X_2)\widehat{\tens}X_3$ are two different iterated direct limits, we need a kind of ``Fubini's Theorem'' (Proposition \ref{prop:Fubini} below) that shows both are isomorphic to a suitable multiple direct limit.

Although the conditions of Theorem \ref{thm:main_thm} may appear extensive, they hold in many examples. In particular, recent results in \cite{CJORY, CY} show that if the category of $C_1$-cofinite modules is closed under contragredient modules, then it has the vertex and braided tensor category structures of \cite{HLZ1}-\cite{HLZ8}. We verify the remaining conditions of Theorem \ref{thm:main_thm} in Theorem \ref{thm:ind_C_1_V} below to conclude:
\begin{thm}\label{thm:intro_ind_C_1_V}
 Suppose the category $\cC^1_V$ of $C_1$-cofinite grading-restricted generalized modules for a vertex operator algebra $V$ is closed under contragredient modules. Then $\cC^1_V$ satisfies the conditions of Theorem \ref{thm:main_thm}, so that $\ind(\cC^1_V)$ admits vertex and braided tensor category structures extending those on $\cC^1_V$.
\end{thm}

Once Theorems \ref{thm:main_thm} and \ref{thm:intro_ind_C_1_V} are established, we have applications to infinite-order extensions of $V$ in $\ind(\cC)$. These are vertex operator algebras $A$ which contain $V$ as a vertex operator subalgebra, that is, $V$ conformally embeds into $A$ since they share the same conformal vector. Although the extension results in \cite{HKL, CKM-ext} are stated for categories of generalized modules  with grading restrictions that modules in $\ind(\cC)$ almost never satisfy, these grading-restriction conditions are not used in essential ways. Thus we get the following two results almost immediately (see Theorems \ref{thm:ext_alg} and \ref{thm:Rep0A} below):

\begin{thm}\label{introthm:ext_alg}
 Let $(V,Y,\vac,\omega)$ be a vertex operator algebra and $\cC$ a category of grading-restricted generalized $V$-modules that satisfies the conditions of Theorem \ref{thm:main_thm}. Then the following two categories are isomorphic:
 \begin{enumerate}
  \item Vertex operator algebras $(A,Y,\vac_A,\omega_A)$ such that:
  \begin{itemize}
   \item $A$ is a $V$-module in $\ind(\cC)$ with vertex operator $Y_A: V\otimes A\rightarrow A((x))$,
   \item $Y_A(v,x)=Y(v_{-1}\vac_A,x)$ for $v\in V$, and
   \item $\omega_A=L(-2)\vac_A \quad (\,=\omega_{-1}\vac_A\,)$.
  \end{itemize}
\item Commutative associative algebras $(A,\mu_A,\iota_A)$ in the braided tensor category $\ind(\cC)$ such that $A$ is $\ZZ$-graded by $L(0)$-eigenvalues and satisfies the grading restriction conditions.
 \end{enumerate}
\end{thm}

\begin{thm}\label{introthm:Rep0A}
 Let $V$ be a vertex operator algebra, $\cC$ a category of grading-restricted generalized $V$-modules that satisfies the conditions of Theorem \ref{thm:main_thm}, $A$ a vertex operator algebra that satisfies the conditions of Theorem \ref{introthm:ext_alg}(1), and $\rep^0 A$ the category of generalized $A$-modules $X$ which are objects of $\ind(\cC)$ as $V$-modules with respect to the vertex operator $Y_X(\iota_A(\cdot),x)$. Then $\rep^0 A$ has vertex and braided tensor category structures as described in Section \ref{sec:vertex_tensor}.
\end{thm}

 We aim to apply these results to two classes of vertex operator algebras whose representation categories are neither finite nor semisimple. The first class is affine vertex algebras and $W$-algebras at specific levels. Already the simple affine vertex algebra of $\mathfrak{sl}_2$ at non-integral admissible level admits uncountably many inequivalent simple modules, most of which fail to satisfy grading restrictions such as finite-dimensional conformal weight spaces and lower bounds on conformal weight. Such properties have prevented the construction of vertex tensor categories for these vertex operator algebras so far, except for the $\beta\gamma$-ghost vertex algebra \cite{AW} and the well-behaved subcategory of grading-restricted modules for many affine vertex algebras \cite{CHY, C, CY}. The second class of vertex operator algebras is those that still have uncountably many inequivalent simple modules, but these modules have finite-dimensional weight spaces and lower-bounded conformal weights. Prototypical examples are the singlet algebras \cite{Ad} and cosets of affine vertex algebras and $W$-algebras, such as the Heisenberg coset of the simple affine vertex algebra of $\mathfrak{sl}_2$ at admissible level \cite{ACR}.

We can now use the direct limit completion to study these non-finite, non-semisimple module categories. Concretely, the singlet algebra $\cM(p)$ for $p\geq 2$ is an infinite direct sum of $C_1$-cofinite modules for the Virasoro algebra at central charge $1-6(p-1)^2/p$, and the same is true for the triplet algebra $\cW(p)$ which at the same time is also an infinite direct sum of irreducible $\cM(p)$-modules. It turns out that all triplet modules lie in the direct limit completion of the category of $C_1$-cofinite Virasoro modules. Moreover, certain subregular $W$-algebras of type $A$, called $\mathcal B_p$ algebras, are infinite-order extensions of the tensor product of $\cM(p)$ with a rank-one Heisenberg  algebra.
Our main results allow us to study categories of $\mathcal B_p$- and $\cM(p)$-modules that lie in the direct limit completion of $C_1$-cofinite Virasoro modules (times Fock modules for the Heisenberg algebra in the $\mathcal B_p$ case).
In the singlet case, we will characterize this category thoroughly, that is, we will compute its fusion ring and prove rigidity, in \cite{CMY}.

In this paper, we present one detailed application in Example \ref{ex:svir}. Using coset constructions from  \cite{CGL}, we show that certain categories of modules for the $N=1$ super Virasoro algebra at generic central charge and for the affine vertex superalgebra of $\mathfrak{osp}(1|2)$ at generic level are equivalent to certain categories of $C_1$-cofinite modules for Virasoro vertex algebras.

The remainder of this paper is organized as follows. In Section \ref{sec:vertex_tensor}, we recall the definitions of various classes of module for a vertex operator algebra, the description of braided tensor categories of modules for a vertex operator algebra from \cite{HLZ1}-\cite{HLZ8}, and properties of the category of $C_1$-cofinite modules for a vertex operator algebra. In Section \ref{sec:ind_lim}, we recall the definition of direct limit in a category and discuss the existence and properties of direct limits in the category of weak modules for a vertex operator algebra. In Section \ref{sec:ind_comp}, we present the definition of direct limit completion $\ind(\cC)$ of a category of modules $\cC$ for a vertex operator algebra, and we show that $\ind(\cC)$ is a $\CC$-linear abelian category if $\cC$ is. In Section \ref{sec:ind_tens_cat}, we show that if $\cC$ is a braided tensor category, then under suitable conditions, $\ind(\cC)$ is also a braided tensor category with twist that contains $\cC$ as a braided tensor subcategory. In Section \ref{sec:ind_lim_vrtx_tens}, we show that if $\cC$ is a vertex tensor category, then under suitable conditions, $\ind(\cC)$ is also a vertex tensor category, with structure isomorphisms described as in Section \ref{sec:vertex_tensor}. Finally, in Section \ref{sec:exam}, we demonstrate sufficient conditions for the category of $C_1$-cofinite modules for a vertex operator algebra to satisfy the conditions of Theorem \ref{thm:main_thm}, and we apply these sufficient conditions to Virasoro and affine vertex operator algebras. We also explain the application of Theorem \ref{thm:main_thm} to vertex operator (super)algebra extensions and present several examples of extensions that can be studied using direct limit completions.

\section{Vertex tensor categories}\label{sec:vertex_tensor}

We use the definition of vertex operator algebra from \cite{FLM,LL}. We will work with the following categories of modules for a vertex operator algebra:
\begin{defi}
Let $V$ be a vertex operator algebra.
\begin{itemize}
\item A \textit{weak $V$-module} is a vector space $W$ equipped with a vertex operator
 \begin{align*}
  Y_W: V & \rightarrow (\Endo W)[[x,x^{-1}]]\nonumber\\
   v & \mapsto Y_W(v,x) = \sum_{n\in\ZZ} v_n\,x^{-n-1}
 \end{align*}
satisfying the following properties:
\begin{enumerate}
 \item \textit{Lower truncation}: for all $v\in V$ and $w\in W$, $v_n w=0$ for $n\in\NN$ sufficiently large.

 \item The \textit{vacuum property}: $Y_W(\vac,x) =\Id_W$.

 \item The \textit{Jacobi identity}: For $v_1,v_2\in V$,
 \begin{align*}
  x_0^{-1}\delta\left(\frac{x_1-x_2}{x_0}\right) Y_W(v_1,x_1)Y_W(v_2,x_2) & - x_0^{-1}\delta\left(\frac{x_2-x_1}{-x_0}\right) Y_W(v_2,x_2)Y_W(v_1,x_1)\nonumber\\
  & =x_2^{-1}\delta\left(\frac{x_1-x_0}{x_2}\right) Y_W(Y(v_1,x_0)v_2,x_2).
 \end{align*}

 \item The \textit{$L(-1)$-derivative property}: For $v\in V$,
 \begin{equation*}
  Y_W(L(-1)v,x)=\frac{d}{dx}Y_W(v,x).
 \end{equation*}
\end{enumerate}

\item A weak $V$-module $W$ is \textit{$\NN$-gradable} if there is an $\NN$-grading
\begin{equation*}
 W=\bigoplus_{n\in\NN} W(n)
\end{equation*}
such that for homogeneous $v\in V$, with conformal weight $\mathrm{wt}\,v$,
\begin{equation}\label{eqn:N-grading_cond}
 v_m W(n)\subseteq W(\mathrm{wt}\,v+n-m-1)
\end{equation}
for $m\in\ZZ$.

\item A \textit{generalized $V$-module} is a weak $V$-module that is graded by generalized $L(0)$-eigenvalues: $W=\bigoplus_{h\in\CC} W_{[h]}$ where
\begin{equation*}
 W_{[h]} =\lbrace w\in W\,\,\vert\,\,(L(0)-h)^N\cdot w = 0\,\,\text{for}\,\,N\in\NN\,\,\text{sufficiently large}\rbrace.
\end{equation*}

\item  A generalized $V$-module $W=\bigoplus_{h\in\CC} W_{[h]}$ is \textit{grading restricted} if:
 \begin{enumerate}
  \item For any $h\in\CC$, $W_{[h]}$ is finite dimensional.
  \item For any $h\in\CC$, $W_{[h+n]}=0$ for $n\in\ZZ$ sufficiently negative.
 \end{enumerate}
\end{itemize}
\end{defi}

 \begin{rem}
The $\NN$-grading on an $\NN$-gradable weak $V$-module need not be unique, and we do not treat such a grading as part of the data of a weak module. That is, $V$-homomorphisms between $\NN$-gradable weak $V$-modules need not preserve $\NN$-gradings.
 \end{rem}
\begin{rem}
We will sometimes use the term ``$V$-module'' as an abbreviation for ``grading-restricted generalized $V$-module.''
\end{rem}
\begin{rem}\label{rem:N-gradable}
 Any $V$-module is $\NN$-gradable. To see why, note that the second restriction on generalized $L(0)$-eigenspaces for a $V$-module $W$ implies that for any coset $\mu\in\CC/\ZZ$ such that $W_{[h]}\neq 0$ for some $h\in\mu$, there is some $h_\mu\in\mu$ with minimal real part such that $W_{[h_\mu]}\neq 0$. For a coset $\mu$ such that $W_{[h]}=0$ for all $h\in\mu$, we can pick $h_\mu\in\mu$ arbitrarily. Then $W=\bigoplus_{n\in\NN} W(n)$ where
 \begin{equation*}
  W(n)=\bigoplus_{\mu\in\CC/\ZZ} W_{[h_\mu+n]}
 \end{equation*}
for $n\in\NN$.
\end{rem}

We will consider categories of generalized $V$-modules that admit braided tensor category structure induced from vertex tensor category structure as in \cite{HLZ1}-\cite{HLZ8}. The tensor product bifunctor of such a braided tensor category is defined in terms of (logarithmic) intertwining operators among $V$-modules, whose definition we now recall:
\begin{defi}
 Let $W_1$, $W_2$, and $W_3$ be a triple of (weak) modules for a vertex operator algebra $V$. An \textit{intertwining operator} of type $\binom{W_3}{W_1\,W_2}$ is a linear map
 \begin{align*}
  \cY : W_1\otimes W_2 & \rightarrow W_3[\log x]\lbrace x\rbrace\nonumber\\
  w_1\otimes w_2 & \mapsto \cY(w_1,x)w_2 =\sum_{h\in\CC}\sum_{k\in\NN} (w_1)_{h,k}w_2\,x^{-h-1}(\log x)^k
 \end{align*}
that satisfies the following properties:
\begin{enumerate}
 \item Lower truncation: for any $h\in\CC$, $w_1\in W_1$, and $w_2\in W_2$,
 \begin{equation*}
  (w_1)_{h+n,k} w_2 = 0
 \end{equation*}
for $n\in\NN$ sufficiently large, independent of $k$.

\item The Jacobi identity: for $v\in V$ and $w_1\in W_1$,
\begin{align*}
 x_0^{-1}\delta\left(\frac{x_1-x_2}{x_0}\right) Y_{W_3}(v,x_1)\cY(w_1,x_2) & - x_0^{-1}\delta\left(\frac{x_2-x_1}{-x_0}\right) \cY(w_1,x_2)Y_{W_2}(v,x_1)\nonumber\\
  & =x_2^{-1}\delta\left(\frac{x_1-x_0}{x_2}\right) \cY(Y_{W_1}(v,x_0)w_1,x_2).
\end{align*}

\item The $L(-1)$-derivative property: For $v\in V$,
\begin{equation*}
 \cY(L(-1)v,x)=\dfrac{d}{dx}\cY(v,x).
\end{equation*}
\end{enumerate}
\end{defi}

For (weak) $V$-modules $W_1$, $W_2$, and $W_3$, we use $\cV^{W_3}_{W_1, W_2}$ to denote the vector space of intertwining operators of type $\binom{W_3}{W_1\,W_2}$. We say that an intertwining operator $\cY$ of type $\binom{W_3}{W_1\,W_2}$ is \textit{surjective} if
\begin{equation*}
 W_3=\text{span}\lbrace (w_1)_{h,k} w_2\,\vert\,w_1\in W_1, w_2\in W_2, h\in\CC,k\in\NN\rbrace.
\end{equation*}
We will use $\im\cY$ to denote the above span of the $(w_1)_{h,k} w_2$, so that $\cY$ is surjective if $W_3=\im\cY$.

We now discuss the braided tensor category structure on a category of generalized $V$-modules induced from vertex tensor category structure as in \cite{HLZ8}; see also the expositions in \cite{HL-review} and \cite{CKM-ext} for more details. The data of a vertex tensor category, as defined in \cite{HL-vrtx-tens}, includes a family of tensor product bifunctors parametrized by elements of the moduli space of Riemann spheres with two positively-oriented punctures, one negatively-oriented puncture, and local coordinates at the punctures. However, to describe the induced braided tensor category structure, we only need to use the $P(z)$-tensor products, where $P(z)$ is the sphere with positively-oriented punctures at $0$ and $z\in\CC^\times$, negatively-oriented puncture at $\infty$, and standard coordinates at each puncture. For any $z_1,z_2\in\CC^\times$, the $P(z_1)$- and $P(z_2)$-tensor products are isomorphic via natural parallel transport isomorphisms associated to homotopy classes of continuous paths in $\CC^\times$ from $z_1$ to $z_2$. For this reason, it is possible to choose, say, a $P(1)$-tensor product bifunctor $\tens$ and then use it to realize all other $P(z)$-tensor products. We shall implicitly do so here to simplify the discussion.

Suppose $\cC$ is a category of (not-necessarily-grading-restricted) generalized modules for a vertex operator algebra $V$ that admits the braided tensor category structure of \cite{HLZ8}.
\begin{itemize}
 \item If $W_1$ and $W_2$ are objects of $\cC$, the tensor product module $W_1\tens W_2$ is an object of $\cC$ characterized by a universal property: There is an intertwining operator $\cY_{W_1,W_2}$ of type $\binom{W_1\tens W_2}{W_1\,W_2}$ such that if $\cY$ is any intertwining operator of type $\binom{W_3}{W_1\,W_2}$ where $W_3$ is an object of $\cC$, then there is a unique $V$-module homomorphism
 \begin{equation*}
  f: W_1\tens W_2\rightarrow W_3
 \end{equation*}
such that $\cY=f\circ\cY_{W_1,W_2}$. In other words, the linear map
\begin{align*}
 \hom_V(W_1\tens W_2, W_3) & \rightarrow\cV_{W_1,W_2}^{W_3}\\
  f & \mapsto f\circ\cY_{W_1,W_2}
\end{align*}
is an isomorphism. It is not hard to see from the universal property of $(W_1\tens W_2,\cY_{W_1,W_2})$ that the intertwining operator $\cY_{W_1,W_2}$ must be surjective (see for instance \cite[Proposition 4.23]{HLZ3}).

\item For morphisms $f_1: W_1\rightarrow\til{W}_1$ and $f_2: W_2\rightarrow\til{W}_2$, the tensor product morphism
\begin{equation*}
 f_1\tens f_2: W_1\tens W_2\rightarrow\til{W}_1\tens\til{W}_2
\end{equation*}
is the unique $V$-module homomorphism guaranteed by the universal property of $(W_1\tens W_2,\cY_{W_1,W_2})$ such that
\begin{equation*}
\cY_{\til{W}_1,\til{W}_2}\circ(f_1\otimes f_2) =(f_1\tens f_2)\circ\cY_{W_1,W_2}.
\end{equation*}

\item The unit object of the braided tensor category $\cC$ is $V$, and for a generalized module $W$ in $\cC$, the left and right unit isomorphisms $l_W: V\tens W\rightarrow W$ and $r_W: W\tens V\rightarrow W$ satisfy
\begin{equation*}
 l_W(\cY_{V,W}(v,x)w)=Y_W(v,x)w
 \end{equation*}
 and
 \begin{equation*}
 r_W(\cY_{W,V}(w,x)v)=e^{xL(-1)} Y_W(v,-x)w
\end{equation*}
for $v\in V$, $w\in W$. Because $\cY_{V,W}$ and $\cY_{W,V}$ are surjective, $l_W$ and $r_W$ are completely determined by these relations.

\item For generalized modules $W_1$ and $W_2$ in $\cC$, the braiding isomorphism $\cR_{W_1,W_2}: W_1\tens W_2\rightarrow W_2\tens W_1$ is characterized by
\begin{equation*}
 \cR_{W_1,W_2}\left(\cY_{W_1,W_2}(w_1,x)w_2\right) =e^{xL(-1)}\cY_{W_2,W_1}(w_2,e^{\pi i}x)w_1
\end{equation*}
for $w_1\in W_1$, $w_2\in W_2$.

\item There is also a twist $\theta$ on $\cC$, a natural automorphism of the identity functor, given by $\theta_W=e^{2\pi iL(0)}$. The twist satisfies $\theta_V=\Id_V$ and the balancing equation
\begin{equation*}
 \theta_{W_1\tens W_2}=\cR_{W_2,W_1}\circ\cR_{W_1,W_2}\circ(\theta_{W_1}\tens\theta_{W_2})
\end{equation*}
for modules $W_1$, $W_2$ in $\cC$.
\end{itemize}

The description of the associativity isomorphisms in $\cC$ requires some preparation. First of all, the existence of associativity isomorphisms depends on, among other conditions, the convergence of products and iterates of intertwining operators among generalized modules in $\cC$. To explain what this means, we first define for any generalized $V$-module $W$ the \textit{graded dual} $W'=\bigoplus_{h\in\CC} W_{[h]}^*$ and the \textit{algebraic completion} $\overline{W}=\prod_{h\in\CC} W_{[h]}$. Note that there is an obvious embedding $\overline{W}\subseteq(W')^*$ (and this embedding is a linear isomorphism if and only if each $W_{[h]}$ is finite dimensional). For any $h\in\CC$, let $\pi_h: \overline{W}\rightarrow W_{[h]}$ denote the projection.

\begin{rem}
 If a generalized module $W$ satisfies the grading restriction condition that for any $h\in\CC$, $W_{[h+n]}=0$ for $n\in\ZZ$ sufficiently negative, then the graded dual $W'$ has the structure of a generalized $V$-module, called the \textit{contragredient} of $W$ (see \cite[Section 5]{FHL}).
\end{rem}

Now take intertwining operators $\cY_1$ of type $\binom{W_4}{W_1\,M_1}$ and $\cY_2$ of type $\binom{M_1}{W_2\,W_3}$ where $W_1$, $W_2$, $W_3$, $W_4$, and $M_1$ are generalized modules in $\cC$. We say that the product of $\cY_1$ and $\cY_2$ is convergent if for $w_1\in W_1$, $w_2\in W_2$, $w_3\in W_3$, and $z_1,z_2\in\CC^\times$ such that $\vert z_1\vert>\vert z_2\vert>0$, the series of linear functionals
\begin{equation*}
 \sum_{h\in\CC} \left\langle\cdot, \cY_1(w_1,z_1)\pi_h\left(\cY_2(w_2,z_2)w_3\right)\right\rangle
\end{equation*}
on $W_4'$ converges absolutely to an element of $\overline{W_4}$ (via the embedding of $\overline{W_4}$ into $(W_4')^*$). The substitution of non-zero complex numbers for formal variables in intertwining operators is accomplished using (any) choice of branches of logarithm. If the product of intertwining operators converges, we denote the limit of the series by
\begin{equation*}
 \cY_1(w_1,z_1)\cY_2(w_2,z_2)w_3\in\overline{W_4};
\end{equation*}
if we wish to emphasize that we are using the choices $\ell_1(z_1)$ and $\ell_2(z_2)$ of branch of logarithm, we denote the product by
\begin{equation*}
 \cY_1(w_1,e^{\ell_1(z_1)})\cY_2(w_2,e^{\ell_2(z_2)})w_3.
\end{equation*}
Similarly, the iterate of intertwining operators $\cY^1$ of type $\binom{W_4}{M_2\,W_3}$ and $\cY^2$ of type $\binom{M_2}{W_2\,W_3}$ converges if for $w_1\in W_1$, $w_2\in W_2$, $w_3\in W_3$, and $z_0,z_2\in\CC^\times$ such that $\vert z_2\vert>\vert z_0\vert>0$, the series of linear functionals
\begin{equation*}
\sum_{h\in\CC} \left\langle\cdot, \cY^1(\pi_h\left(\cY^2(w_1,z_0)w_2\right),z_2)w_3\right\rangle
\end{equation*}
on $W_4'$ converges absolutely to an element of $\overline{W_4}$.

The convergence of products and iterates of intertwining operators is essential for the existence of suitable associativity isomorphisms in $\cC$. In fact, when associativity isomorphisms exist, they are described as follows:
\begin{itemize}
 \item For generalized modules $W_1$, $W_2$, and $W_3$ in $\cC$, the associativity isomorphism
 \begin{equation*}
  \cA_{W_1,W_2,W_3}: W_1\tens(W_2\tens W_3)\rightarrow(W_1\tens W_2)\tens W_3
 \end{equation*}
is characterized by the equality
\begin{align*}
& \left\langle w', \overline{\cA_{W_1,W_2,W_3}}\left(\cY_{W_1,W_2\tens W_3}(w_1,e^{\ln r_1})\cY_{W_2,W_3}(w_2,e^{\ln r_2})w_3\right)\right\rangle\nonumber\\
&\hspace{5em} = \left\langle w',\cY_{W_1\tens W_2,W_3}\left(\cY_{W_1,W_2}(w_1,e^{\ln(r_1-r_2)})w_2),e^{\ln r_2}\right)w_3\right\rangle
\end{align*}
for any $w_1\in W_1$, $w_2\in W_2$, $w_3\in W_3$, $w'\in((W_1\tens W_2)\tens W_3)'$, and $r_1,r_2\in\RR$ such that $r_1>r_2>r_1-r_2>0$. Here, $\overline{\cA_{W_1,W_2,W_3}}$ is the obvious extension of $\cA_{W_1,W_2,W_3}$ to algebraic completions, and $\ln r$ for $r\in\RR_+$ is the real-valued branch of logarithm.
\end{itemize}

\begin{rem}
 As explained in \cite[Proposition 3.32]{CKM-ext}, the above relation for the associativity isomorphism $\cA_{W_1,W_2,W_3}$ holds for every $r_1,r_2\in\RR$ which satisfy $r_1>r_2>r_1-r_2>0$. Indeed, this relation shows that the multivalued analytic functions on $\vert z_1\vert>\vert z_2\vert>0$ and $\vert z_2\vert>\vert z_1-z_2\vert>0$ determined, respectively, by the product of $\cA_{W_1,W_2,W_3}\circ\cY_{W_1,W_2\tens W_3}$, $\cY_{W_2,W_3}$ and by the iterate of $\cY_{W_1\tens W_2,W_3}$, $\cY_{W_1,W_2}$ have equal restrictions to their common domain.
\end{rem}

\begin{rem}
For the full vertex tensor category structure on $\cC$, one can take the $P(z)$-tensor product $W_1\tens_{P(z)} W_2$ for each $z\in\CC^\times$ to be the module $W_1\tens W_2$, which satisfies a universal property with respect to the \textit{$P(z)$-intertwining map}
 \begin{equation*}
  \cY_{W_1,W_2}(\cdot,z)\cdot : W_1\otimes W_2\rightarrow\overline{W_1\tens W_2}.
 \end{equation*}
As before, the substitution $x\mapsto z$ is accomplished using a choice of branch for $\log z$.
\end{rem}

The question of whether a category $\cC$ of modules for a vertex operator admits braided tensor category structure as described above is usually rather difficult: the most difficult part is showing the existence of associativity isomorphisms. Perhaps the most natural category of grading-restricted generalized $V$-modules to consider for vertex tensor category structure is the category of $C_1$-cofinite modules:
\begin{defi}
 Let $V$ be a vertex operator algebra and $W$ an $\NN$-gradable weak $V$-module. Define the vector subspace
 \begin{equation*}
  C_1(W)=\text{span}\bigg\lbrace v_{-1} w\,\bigg\vert\,v\in\bigoplus_{n\geq 1} V_{(n)}, w\in W\bigg\rbrace.
 \end{equation*}
Then we say that $W$ is \textit{$C_1$-cofinite} if $\dim W/C_1(W)<\infty$. We use $\cC^1_V$ to denote the category of $C_1$-cofinite grading-restricted generalized $V$-modules.
\end{defi}

The category $\cC^1_V$ has good vertex algebraic properties: Huang showed in \cite{Hu_diff_eqs} that compositions of intertwining operators among $C_1$-cofinite modules satisfy regular-singular-point differential equations and hence are convergent, and in \cite{Mi}, Miyamoto showed (also using differential equations) that $C_1$-cofinite modules are closed under tensor products. However, it is not clear in general that $\cC^1_V$ has good algebraic properties: especially, it is not clear in general that $\cC^1_V$ is closed under submodules and contragredients. Nevertheless, recent results in \cite[Section 4]{CJORY} (see especially the proof of Theorem 4.2.5) and \cite[Section 3]{CY} (see especially the proof of Theorem 3.3.4) show that if $\cC^1_V$ is in fact closed under contragredients, then it admits braided tensor category structure as described in this section:
\begin{thm}\label{thm:C1_vrtx_tens}
 Let $V$ be a vertex operator algebra and assume that the category $\cC^1_V$ of $C_1$-cofinite grading-restricted generalized $V$-modules is closed under contragredients. Then $\cC^1_V$ admits the vertex and braided tensor category structures of \cite{HLZ1}-\cite{HLZ8}.
\end{thm}

We end this section with some basic results on $C_1$-cofinite modules that we will use later:
\begin{prop}\label{prop:C1-fin-gen}
Any $C_1$-cofinite $\NN$-gradable weak $V$-module is finitely generated.
\end{prop}
\begin{proof}
 Suppose $W=\bigoplus_{n\in\NN} W(n)$ is a $C_1$-cofinite $\NN$-gradable weak $V$-module and let $T$ be a finite-dimensional subspace of $W$ such $W=T+C_1(W)$. We may assume that $T$ is graded, and moreover since $W(0)$ is finite dimensional by \cite[Lemma 6]{Mi}, we may assume that $T$ contains $W(0)$. We will show that $W$ is generated by $T$, that is, $W=\langle T\rangle$.

 We prove that $W(n)\subseteq\langle T\rangle$ by induction on $n$ with the base case $n=0$ handled by the assumption $W(0)\subseteq T$. So take $n>0$ and assume $W(m)\subseteq\langle T\rangle$ for $m<n$. For any $w\in W(n)$, we have
 \begin{equation}\label{eqn:C1-fin-gen}
  w=w'+\sum_i v^{(i)}_{-1} w^{(i)}
 \end{equation}
for $w'\in T$, $v^{(i)}\in V$, and $w^{(i)}\in W$. Since $w$ is homogeneous and $T$ is graded, we may assume that $w'\in T\cap W(n)$ and also that each $v^{(i)}$ is homogeneous (of positive weight). Then by \eqref{eqn:N-grading_cond}, the $w^{(i)}$ are homogeneous in the $\mathbb{N}$-grading of $W$ with degree
\begin{equation*}
 \deg w^{(i)} =n-(\mathrm{wt}\,v^{(i)}-(-1)-1)< n,
\end{equation*}
so each $w^{(i)}\in\langle T\rangle$ by induction. Then \eqref{eqn:C1-fin-gen} shows $w\in\langle T\rangle$ as well.
\end{proof}

\begin{prop}\label{prop:C1N=C1mod}
 The category of $C_1$-cofinite $\NN$-gradable weak $V$-modules equals $\cC^1_V$.
\end{prop}
\begin{proof}
 Remark \ref{rem:N-gradable} shows that every module in $\cC^1_V$ is $\NN$-gradable. Conversely, if $W=\bigoplus_{n\in\NN} W(n)$ is a $C_1$-cofinite $\NN$-gradable weak $V$-module, \cite[Lemma 6]{Mi} shows that every $W(n)$ is finite dimensional. Since $W(n)$ is also $L(0)$-invariant, each $W(n)$ decomposes as the direct sum of $L(0)$-generalized eigenspaces, so $W=\bigoplus_{h\in\CC} W_{[h]}$ where $W_{[h]}$ is the $L(0)$-generalized eigenspace with generalized eigenvalue $h$.

 Now Proposition \ref{prop:C1-fin-gen} shows that $W$ is generated by the finite-dimensional subspace $T=\bigoplus_{n=0}^N W(n)$ for some sufficiently large $N$. Let $\lbrace h_i\rbrace_{i=1}^I$ be the finitely many generalized eigenvalues for $L(0)$ acting on $T$. Since $T$ generates $W$, $W(n)$ for $n>N$ is spanned by vectors of the form $v_m w'$ for $v\in V$ and $w'\in T$ (see \cite[Proposition 4.5.6]{LL}). We may assume that $v$ is $L(0)$-homogeneous and that $w'$ is both $L(0)$-homogeneous (with generalized $L(0)$-eigenvalue $h_i$) and homogeneous in the $\mathbb{N}$-gradation of $W$ (with degree $\deg w'$). Such a vector $v_m w'$ has generalized $L(0)$-eigenvalue
 \begin{equation}\label{eqn:vmw'_conf_wt}
  \mathrm{wt}\,v_m w' =h_i+\mathrm{wt}\,v-m-1
 \end{equation}
and $\mathbb{N}$-degree
\begin{equation*}
 \deg v_m w' =n =\deg w'+\mathrm{wt}\,v-m-1,
\end{equation*}
so that in particular $\mathrm{wt}\,v-m-1=n-\deg w'$. Thus
\begin{equation*}
0< \mathrm{wt}\,v-m-1\leq n,
\end{equation*}
and combined with \eqref{eqn:vmw'_conf_wt}, this shows that the generalized eigenvalues of $L(0)$ acting on $W(n)$ for $n>N$ have the form $h_i+k$ for $0<k\leq n$. In particular, all generalized $L(0)$-eigenvalues on $W$ are contained in
$\bigcup_{i=1}^I \lbrace h_i+\NN\rbrace$, and moreover, any $W_{[h_i+m]}$ is contained in the finite-dimensional subspace $\bigoplus_{n=0}^{N+m} W(n)$. This proves the grading-restriction conditions showing that $W$ is a grading-restricted generalized $V$-module.
\end{proof}

Finally, the following results are key for applying our main results on direct limit completions to $\cC^1_V$:
\begin{lem}
 Suppose $W_1$ is a $V$-module in $C^1_V$, $W_2$ is a finitely-generated grading-restricted generalized $V$-module, $X$ is a generalized $V$-module, and $\cY$ is an intertwining operator of type $\binom{X}{W_1\,W_2}$. Then $\im\cY$ is $\NN$-gradable.
\end{lem}
\begin{proof}
 Since $W_1$ and $W_2$ are finitely generated, their conformal weights lie in finitely many cosets of $\CC/\ZZ$. Thus they are both $\NN$-gradable as in Remark \ref{rem:N-gradable} with finite-dimensional homogeneous subspaces. Let $T_1$ be a finite-dimensional graded (with respect to the $\NN$-grading) subspace of $W_1$ such that $W_1=T_1+C_1(W_1)$; then $W_1(0)\subseteq T_1$. We first consider the case that $W_2$ is generated by $W_2(0)$.

 From the finite-dimensionality of $T_1$ and $W_2(0)$, the lower truncation of the intertwining operator $\cY$, and the $L(0)$-conjugation formula
 \begin{equation*}
  y^{L(0)}\cY(u_1,x)u_2 =\cY(y^{L(0)} u_1,xy)y^{L(0)} u_2
 \end{equation*}
for $u_1\in T_1$, $u_2\in W_2(0)$ (see \cite[Proposition 3.36(b)]{HLZ2}), we find that
\begin{equation*}
 \im\cY\vert_{T_1\otimes W_2(0)}\subseteq\bigoplus_{\mu\in\CC/\ZZ}\bigoplus_{n\in\NN} X_{[h_\mu+n]}
\end{equation*}
for suitable $h_\mu\in\mu$. Denoting this graded subspace of $X$ by $\widetilde{X}$, the argument of Remark \ref{rem:N-gradable} implies that it is enough to show that $\im\cY\subseteq\widetilde{X}$.

We first show that $\im\cY\vert_{W_1(n)\otimes W_2(0)}\subseteq\widetilde{X}$ by induction on $n$. The base case holds because $W_1(0)\subseteq T_1$. Then for $w_1\in W_1(n)$ with $n>0$, we may assume either $w_1\in T_1$ or $w_1=v_{-1} u_1$ for homogeneous $v\in V$ of weight $\mathrm{wt}\,v>0$ and homogeneous $u_1\in W_1$ of degree $\deg u_1$ such that
\begin{equation*}
 \mathrm{wt}\, v+\deg u_1=n.
\end{equation*}
In the first case, there is nothing to prove, and in the second weak associativity for intertwining operators (which follows from the Jacobi identity) implies
\begin{equation*}
 (x_0+x_2)^{\mathrm{wt}\, v} Y_X(v,x_0+x_2)\cY(u_1,x_2)u_2 = (x_0+x_2)^{\mathrm{wt}\, v} \cY(Y_{W_1}(v,x_0)u_1,x_2)u_2
\end{equation*}
for $u_2\in W_2(0)$. Extracting the constant term in $x_0$ leads to
\begin{align*}
 \cY(v_{-1} u_1,x_2)u_2 = \sum_{i\geq 0} v_{\mathrm{wt}\, v-1-i}\,x_2^{i-\mathrm{wt}\, v}\cY(u_1,x_2)u_2 -\sum_{i=1}^{\mathrm{wt}\, v}\binom{\mathrm{wt}\, v}{i} x_2^{-i}\cY(v_{-1+i} u_1,x_2)u_2.
\end{align*}
Since $v_{\mathrm{wt}\,v-1-i}$ is an operator of degree $i$ for all $i\geq 0$ and since
\begin{equation*}
 \deg v_{-1+i} u_1=\mathrm{wt}\, v+\deg u_1-i<n
\end{equation*}
for $1\leq i\leq\mathrm{wt}\, v$, the induction hypothesis implies that the coefficients of $\cY(v_{-1} u,x_2)$ lie in $\widetilde{X}$. This completes the proof that $\im\cY\vert_{W_1\otimes W_2(0)}\subseteq\widetilde{X}$.

Now under the assumption that $W_2(0)$ generates $W_2$, $W_2$ is the linear span of vectors $v_n u_2$ for $u_2\in W_2(0)$, homogeneous $v\in V$, and $n\in\ZZ$ such that $\mathrm{wt}\, v-n-1\geq 0$ (see \cite[Proposition 4.5.6]{LL}). The commutator formula for intertwining operators then implies
\begin{align*}
 \cY(w_1,x)v_n u_2 =v_n\cY(w_1,x)u_2-\sum_{i\geq 0}\binom{n}{i}x^{n-i}\cY(v_i w_1,x)u_2\in\widetilde{X}[\log x]\lbrace x\rbrace
\end{align*}
since $v_n$ is an operator with non-negative degree. This proves that $\im\cY$ is a generalized submodule of $X$ with lower-bounded conformal weights, and is thus $\NN$-gradable, when $W(0)$ generates $W_2$.

Now for general finitely-generated $W_2$, there is a finite filtration
\begin{equation*}
 0=W_2^{(0)}\subseteq W_2^{(1)}\subseteq\cdots\subseteq W_2^{(n)}=W_2
\end{equation*}
of $\NN$-graded submodules such that the successive quotients $W_2^{(i)}/W_2^{(i-1)}$ are $\NN$-gradable $V$-modules generated by their finite-dimensional degree-$0$ subspaces. Thus we can prove the lemma by induction on the length $n$ of the filtration, having already handled the case $n=1$.

For finitely-generated $W_2$ with $\NN$-grading, let $\widetilde{W}_2$ denote the submodule generated by $W_2(0)$. We have shown that $\im\cY\vert_{W_1\otimes\widetilde{W}_2}$ is a generalized submodule $\widetilde{X}\subseteq X$ with lower-bounded conformal weights. Also, $\cY$ induces an intertwining operator
\begin{equation*}
 \widetilde{\cY}: W_1\otimes(W_2/\widetilde{W}_2)\rightarrow(X/\widetilde{X})[\log x]\lbrace x\rbrace.
\end{equation*}
By induction, $\im\widetilde{\cY}$ is a generalized submodule of $X/\widetilde{X}$ with lower-bounded conformal weights. Then because we have an exact sequence
\begin{equation*}
 0\longrightarrow \widetilde{X}\longrightarrow\im\cY\longrightarrow\im\widetilde{\cY}\longrightarrow 0,
\end{equation*}
we see that $\im\cY$ has lower-bounded conformal weights and is thus $\NN$-gradable by the argument of Remark \ref{rem:N-gradable}.
\end{proof}

Now we conclude the following variant of the Key Theorem of \cite{Mi}:
\begin{cor}\label{cor:ImY_C1}
 Suppose $W_1$, $W_2$ are $V$-modules in $\cC^1_V$, $X$ is a generalized $V$-module, and $\cY$ is an intertwining operator of type $\binom{X}{W_1\,W_2}$. Then $\im\cY$ is a module in $\cC^1_V$.
\end{cor}
\begin{proof}
 By Proposition \ref{prop:C1-fin-gen}, $W_1$ and $W_2$ satisfy the hypotheses of the preceding lemma, so $\im\cY$ is $\NN$-gradable. Then $\im\cY$ is also $C_1$-cofinite by \cite[Key Theorem]{Mi}. Finally, Proposition \ref{prop:C1N=C1mod} shows $\im\cY$ is a grading-restricted generalized module in $\cC^1_V$.
\end{proof}

\section{Direct limits of weak modules}\label{sec:ind_lim}

In this section, we recall the definition of direct limit (in any category) and discuss the basic properties of direct limits in the category of weak modules for a vertex operator algebra.
\begin{defi}
A \textit{directed set} is a non-empty set $I$ with a reflexive and transitive binary relation $\leq$ such that for any $i,j\in I$, there exists $k\in I$ such that $i\leq k$ and $j\leq k$.
\end{defi}
\begin{rem}\label{rem:dir_set_cat}
To any directed set $(I,\leq)$, we can associate a small category with object set $I$ and $\hom(i,j)$ non-empty (and a singleton) if and only if $i\leq j$.
 \end{rem}

 \begin{rem}\label{rem:filt_two}
 If $(I,\leq)$ is a directed set, then for any finite subset $\lbrace i_n\rbrace_{n=1}^N\subseteq I$, there exists $k\in I$ such that $i_n\leq k$ for all $1\leq n\leq N$.
 \end{rem}

Now we define direct limits in any category $\cC$:
\begin{defi}
 Let $\cC$ be a category.
 \begin{itemize}
  \item A \textit{direct} (or \textit{inductive}) \textit{system} is a functor $\alpha: I\rightarrow\cC$, where $I$ is the category associated to a directed set $(I,\leq)$ as in Remark \ref{rem:dir_set_cat}. If $i,j\in I$ such that $i\leq j$, then we use $f_{i}^{j}:\alpha(i)\rightarrow\alpha(j)$ to denote the morphism $\alpha(i\to j)$ in $\cC$. Thus $$f_{i}^{i}=\Id_{\alpha(i)}$$
  for any $i\in I$ and
  $$f_{j}^{k}\circ f_{i}^{j}=f_{i}^{k}$$
  if $i\leq j\leq k$ in $I$.

  \item A \textit{target} of a direct system $\alpha: I\rightarrow\cC$ is an object $X$ in $\cC$ together with morphisms $\psi_i: \alpha(i)\rightarrow X$ such that
  \begin{equation*}
   \psi_i=\psi_j\circ f_i^j
  \end{equation*}
for any $i\leq j$ in $I$.

\item A \textit{direct} (or \textit{inductive}) \textit{limit} of a direct system $\alpha: I\rightarrow\cC$ is a target $(\varinjlim\alpha,\lbrace\phi_i\rbrace_{i\in I})$ satisfying the following universal property: for any target $(X,\lbrace\psi_i\rbrace_{i\in I})$, there is a unique morphism $F:\varinjlim\alpha\rightarrow X$ such that the diagram
\begin{equation*}
  \xymatrixcolsep{4pc}
  \xymatrix{
  \alpha(i) \ar[d]_{\phi_i} \ar[rd]^{\psi_i} & \\
  \varinjlim\alpha \ar[r]_{F} & X \\
  }
 \end{equation*}
commutes for all $i\in I$.
 \end{itemize}
\end{defi}

We now fix a vertex operator algebra $V$ and let $\cW$ denote the full category of weak $V$-modules. Note that $\cW$ is a $\CC$-linear abelian category; recall that this means:
\begin{enumerate}
 \item Morphism sets are $\CC$-vector spaces such that composition is bilinear.
 \item There is a zero object $0$.
 \item Any finite set of objects $\lbrace W_n\rbrace$ has a biproduct. Recall that this is an object $\bigoplus W_n$ together with projection morphisms  $p_n: \bigoplus W_n\rightarrow W_n$ and inclusion morphisms $q_n: W_n\rightarrow\bigoplus W_n$ such that $p_m\circ q_n=\delta_{m,n}\mathrm{Id}_{W_n}$ and $\sum q_n\circ p_n=\mathrm{Id}_{\bigoplus W_n}$. A biproduct is both a product and a coproduct of the $W_n$.
%

 \item Every morphism has a kernel and cokernel.
 \item Every monomorphism is a kernel and every epimorphism is a cokernel.
\end{enumerate}
In addition, $\cW$ is closed under arbitrary coproducts since any direct sum of weak $V$-modules is a weak $V$-module in the obvious way. As above, for any direct sum $\bigoplus_{i\in I} W(i)$ in $\cW$, we will use $q_i$ to denote the inclusion of $W(i)$ into the direct sum and $p_i$ to denote the projection from the direct sum to $W(i)$.
\begin{prop}\label{prop:direct_lim_exist}
 The category $\cW$ of weak $V$-modules contains a direct limit of any direct system $\alpha: I\rightarrow\cW$.
\end{prop}
\begin{proof}
 We can use the construction of direct limits in the category of vector spaces. For a direct system $\alpha: I\rightarrow\cW$, we set $\alpha(i)= W(i)$ for $i\in I$ and define
\begin{equation*}
 \varinjlim\alpha =\bigoplus_{i\in I} W(i)\bigg/ K_\alpha
\end{equation*}
where
\begin{equation*}
 K_\alpha=\sum_{i\in I}\sum_{j\geq i}\text{span}\lbrace q_i(w_i)-q_j(f_{i}^{j}(w_i))\,\vert\,w_i\in W(i)\rbrace.
\end{equation*}
We define $\phi_i: W(i)\rightarrow\varinjlim\alpha$ for $i\in I$ to be $q_i$ followed by projection onto the quotient. Then it is immediate from the definition of $K_\alpha$ that $\phi_j\circ f_{i}^{j}=\phi_i$ for any $i\leq j$ in $I$.

Since the $W(i)$ are weak $V$-modules, the direct sum $\bigoplus_{i\in I} W(i)$ has the structure of a weak $V$-module such that the $q_i$ are $V$-homomorphisms. Also, $K_\alpha$ is a weak $V$-submodule, because if $q_i(w_i)-q_j(f_{i}^{j}(w_i))$ is a spanning vector of $K_\alpha$, then 
\begin{align*}
 Y_{\bigoplus_{i\in I} W(i)}(v,x)\left(q_i(w_i)-q_j(f_{i}^{j}(w_i))\right) = q_i\left(Y_{W(i)}(v,x)w_i\right)-q_j\left(f_{i}^{j}(Y_{W(i)}(v,x)w_i)\right)
\end{align*}
for $v\in V$, so that the coefficients of powers of $x$ are still spanning vectors of $K_\alpha$. Thus the quotient $\varinjlim\alpha$ is a weak $V$-module and $\phi_i$ for each $i\in I$ is a $V$-homomorphism.

Now we show that $(\varinjlim\alpha,\lbrace\phi_i\rbrace_{i\in I})$ is actually a direct limit of $\alpha$. Thus suppose $(X,\lbrace\psi_i\rbrace_{i\in I})$ is a target of $\alpha$ and define
\begin{equation*}
 F=\sum_{i\in I} \psi_i\circ p_i: \bigoplus_{i\in I} W(i)\rightarrow X.
\end{equation*}
If $q_i(w_i)-q_j(f_{i}^{j}(w_i))$ is a spanning vector for $K_\alpha$, then
\begin{equation*}
 F\left(q_i(w_i)-q_j(f_{i}^{j}(w_i))\right) =\psi_i(w_i)-\psi_j(f_{i}^{j}(w_i)) =0,
\end{equation*}
so $F$ descends to a well-defined homomorphism $\overline{F}:\varinjlim\alpha\rightarrow X$ such that
\begin{equation*}
 \overline{F}\circ\phi_i = F\circ q_i=\psi_i
\end{equation*}
for $\in I$. That $\overline{F}$ is the unique homomorphism such that $\overline{F}\circ\phi_i=\psi_i$ follows because $\varinjlim\alpha$ is spanned by the images of the $\phi_i$.
\end{proof}

In the realization of $(\varinjlim\alpha,\lbrace\phi_i\rbrace_{i\in I})$ given in the preceding proof, $\varinjlim\alpha$ is the sum of the images of the $\phi_i$. In fact, from the universal property, any realization of the direct limit will be the sum of the images of the $\phi_i$. From properties of directed sets, we can say more:
\begin{prop}\label{prop:dir_lim_union}
For any direct limit $(\varinjlim\alpha,\lbrace\phi_i\rbrace_{i\in I})$ of a direct system $\alpha: I\rightarrow\cW$,
\begin{equation*}
 \varinjlim\alpha=\bigcup_{i\in I} \im \phi_i.
\end{equation*}
\end{prop}
\begin{proof}
 We have seen that any vector $b\in\varinjlim\alpha$ can be written $\phi_{i_1}(w_1)+\ldots+\phi_{i_N}(w_N)$ for finitely many $i_1,\ldots,i_N\in I$ and $w_n\in W(i_n)$ for $1\leq n\leq N$. By Remark \ref{rem:filt_two}, there is some $k\in I$ such that $i_n\leq k$ for each $n$. Then
 \begin{equation*}
  b=\sum_{n=1}^N \phi_{i_n}(w_n) =\sum_{n=1}^N \phi_k(f_{i_n}^{k}(w_n)) \in\im \phi_k.
 \end{equation*}
This shows that $\varinjlim\alpha\subseteq\bigcup_{i\in I}\im\phi_i$, and the reverse inclusion is obvious.
\end{proof}

Later, we will need a characterization of the kernels of the $\phi_i$ associated to a direct limit:
\begin{lem}\label{lem:ker_phi_i}
 For any direct limit $(\varinjlim\alpha,\lbrace\phi_i\rbrace_{i\in I})$ of a direct system $\alpha: I\rightarrow\cW$,
 \begin{equation*}
  \ker\phi_i=\bigcup_{j\geq i} \ker f_{i}^{j},
 \end{equation*}
recalling that the $f_i^j$ are the morphisms $\alpha(i \rightarrow j)$ in $\cW$.
\end{lem}
\begin{proof}
 If $i\leq j$ for some $j\in I$, then $\ker f_{i}^{j}\subseteq\ker\phi_i$ because $\phi_j\circ f_{i}^{j}=\phi_i$. Conversely, we need to show that if $w_i\in\ker\phi_i$, then $w_i\in\ker f_{i}^{k}$ for some $k\geq i$ in $I$. We may use the realization of the direct limit $\varinjlim\alpha$ from the proof of Proposition \ref{prop:direct_lim_exist}, so that $w_i\in\ker\phi_i$ if and only if $q_i(w_i)\in K_\alpha$, that is, if and only if
\begin{equation*}
 q_i(w_i)=\sum_{n=1}^N \left(q_{i_n}(w_{n})-q_{j_n}(f_{i_n}^{j_n}(w_{n}))\right)
\end{equation*}
for certain $i_n,j_n\in I$ such that $i_n\leq j_n$ and $w_{n}\in W(i_n)$.

By using the notation $w_{N+n}=-f_{i_n}^{j_n}(w_n)$, $i_{N+n}=j_n$ for $n\in\lbrace 1,\ldots,N\rbrace$, we may rewrite
\begin{equation}\label{eqn:qi_wi}
 q_i(w_i)=\sum_{n=1}^{2N} q_{i_n}(w_n).
\end{equation}
Moreover, by Remark \ref{rem:filt_two}, there is some $k\in I$ such that $i\leq k$ and $j_n\leq k$ for $1\leq n\leq N$. Then we get
\begin{align}\label{eqn:gn_wn}
 \sum_{n=1}^{2N} f_{i_n}^{k}(w_n) & =\sum_{n=1}^N \left(f_{i_n}^{k}(w_n)- f_{j_n}^{k}(f_{i_n}^{j_n}(w_n))\right) = 0.
\end{align}
By applying the projection $p_j$ to \eqref{eqn:qi_wi} for $j\in I$, we see that
\begin{equation*}
 \sum_{i_n=j} w_n =\delta_{i,j}w_i.
\end{equation*}
Thus using \eqref{eqn:gn_wn}, we conclude
\begin{align*}
 0 =\sum_{n=1}^{2N} f_{i_n}^{k}(w_n) =\sum_{j\leq k}\sum_{i_n=j} f_{j}^k(w_n) =\sum_{j\leq k} f_{j}^{k}(\delta_{i,j} w_i) =f_{i}^{k}(w_i).
\end{align*}
That is, $w_i\in\ker f_{i}^{k}$.
\end{proof}

\section{Direct limit completion in the weak module category}\label{sec:ind_comp}

In this section, we will define the direct limit completion $\ind(\cC)$ of a category $\cC$ of grading-restricted generalized modules for a vertex operator algebra in the category $\cW$ of weak modules. The work here is motivated by the ind-completion constructions of, for example \cite[Chapter 6]{KS}, but it is important to note that here we are using the ambient category $\cW$ of weak modules rather than the Yoneda category $\cC^\wedge$ of contravariant functors from $\cC$ to vector spaces. By \cite[Proposition 2.7.1]{KS}, there will be an essentially surjective functor from the ind-completion of \cite[Chapter 6]{KS} to the category $\ind(\cC)$ defined here, but it is not clear whether this functor will be fully faithful.

We now fix a category of $V$-modules that satisfies the following conditions:
\begin{assum}\label{assum:Sec4}
Assume that $\cC$ is a full (sub)category of grading-restricted generalized $V$-modules such that:
 \begin{enumerate}
  \item The zero module $0$ is an object of $\cC$.
  \item The category $\cC$ is closed under submodules, quotients, and finite direct sums.
  \item Every module in $\cC$ is finitely generated.
 \end{enumerate}
\end{assum}

The first two assumptions on $\cC$ guarantee that $\cC$ is an abelian category, while the third is used in the following two lemmas that we will need in the next section:
\begin{lem}\label{lem:im_f_in_im_phi_i}
 Suppose $(\varinjlim\alpha,\lbrace\phi_i\rbrace_{i\in I})$ is the direct limit in $\cW$ of a direct system $\alpha: I\rightarrow\cC$. If $W\subseteq\varinjlim\alpha$ is any $V$-submodule that is an object of $\cC$, then $W\subseteq\im\phi_i$ for some $i\in I$.
\end{lem}
\begin{proof}
 Since $\varinjlim\alpha=\bigcup_{i\in I}\im\phi_i$ by Proposition \ref{prop:dir_lim_union}, the finitely many generators of $W$ are contained in $\sum_{n=1}^N\im\phi_{j_n}$ for finitely may $j_n\in I$. Thus
 \begin{equation*}
  W\subseteq\sum_{n=1}^N\im\phi_{j_n}\subseteq\im\phi_i
 \end{equation*}
where $i\in I$ is such that $j_n\leq i$ for each $n$.
\end{proof}

\begin{lem}\label{lem:ker_fin_gen}
 If $(\varinjlim\alpha,\lbrace\phi_i\rbrace_{i\in I})$ is the direct limit in $\cW$ of a direct system $\alpha: I\rightarrow\cC$, then for each $i\in I$,
 \begin{equation*}
  \ker\phi_i=\ker f_{i}^{j}
 \end{equation*}
for some $j\in I$.
\end{lem}
\begin{proof}
 By Lemma \ref{lem:ker_phi_i}, $\ker\phi_i=\bigcup_{j\geq i} \ker f_{i}^{j}$. Since $\ker\phi_i\subseteq W(i)$ is a submodule of a module in $\cC$, it is itself a module in $\cC$ and is finitely generated. Thus there are finitely many $j_1,\ldots,j_N\in I$ such that
 \begin{equation*}
  \ker\phi_i\subseteq\sum_{n=1}^N\ker f_{i}^{j_n}\subseteq\ker f_{i}^{j},
 \end{equation*}
where $j\in I$ is such that $j_n\leq j$ for each $n$. Since obviously $\ker f_{i}^{j}\subseteq\ker\phi_i$, we get $\ker\phi_i=\ker f_{i}^{j}$.
\end{proof}

We now introduce the direct limit completion of $\cC$ in $\cW$:
\begin{defi}
 The \textit{direct limit completion} of $\cC$ in $\cW$ is the full subcategory $\ind(\cC)$ of objects in $\cW$ which are isomorphic to direct limits of direct systems $\alpha: I\rightarrow\cC$.
\end{defi}

We would like to realize objects of $\ind(\cC)$ as direct limits in a canonical way. For any weak $V$-module $X$, let $I_X$ denote the set of $V$-submodules which are objects of $\cC$; $I_X$ is non-empty because by assumption $0$ is an object of $\cC$. Then $(I_X,\subseteq)$ is a directed set because if $W_1, W_2$ are two $\cC$-submodules of $X$, then so is $W_1+W_2$: it is a quotient of $W_1\oplus W_2$, and $\cC$ is closed under finite direct sums and quotients by assumption. Thus for any weak $V$-module $X$, we can define an inductive system $\alpha_X: I_X\rightarrow\cC$ such that $\alpha_X(W)=W$ for $\cC$-submodules $W\subseteq X$, and such that $f_{W}^{\til{W}}: W\rightarrow\til{W}$ is the inclusion map for $\cC$-submodules $W\subseteq\til{W}\subseteq X$.

We would like to show that if $X$ is an object of $\ind(\cC)$, then $X$ is isomorphic to $\varinjlim\alpha_X$. Indeed, defining $i_W: W\rightarrow X$ for any $W\in I_X$ to be the inclusion, $(X,\lbrace i_W\rbrace_{W\in I_X}\rbrace)$ is a target of $\alpha_X$. Thus by the universal property of $\varinjlim\alpha_X$, we have a unique $V$-homomorphism $Q_X: \varinjlim\alpha_X\rightarrow X$ such that
\begin{equation*}
 \xymatrix{
 W \ar[d]_{\phi_W} \ar[rd]^{i_W} & \\
 \varinjlim\alpha_X \ar[r]_(.6){Q_X} & X \\
 }
\end{equation*}
commutes for $W\in I_X$.

\begin{prop}\label{prop:alpha_X}
The $V$-homomorphism $Q_X$ is injective for every weak $V$-module $X$ and is surjective if and only if $X$ is an object of $\ind(\cC)$.
\end{prop}
\begin{proof}
 We use Proposition \ref{prop:dir_lim_union}. For $b\in\ker Q_X$, we have $b=\phi_W(w)$ for some $\cC$-submodule $W\in I_X$ and $w\in W$, so that
 \begin{equation*}
  0=Q_X(b)=Q_X(\phi_W(w))=i_W(w).
 \end{equation*}
 Since $i_W$ is injective, $w=0$ and $b=\phi_W(w)=0$ as well. Thus $\ker Q_X=0$.

Now if $X$ is not an object of $\ind(\cC)$, then $Q_X$ is not surjective because $Q_X$ cannot be an isomorphism. On the other hand, if $X$ is an object of $\ind(\cC)$, then Proposition \ref{prop:dir_lim_union} implies that $X$ is the union of submodules which are objects of $\cC$ (because $\cC$ is closed under quotients). In other words, for any $b\in X$, $b=i_W(w)$ for some $\cC$-submodule $W\in I_X$ and $w\in W$. Thus
\begin{equation*}
 b=i_W(w)=Q_X(\phi_W(w))
\end{equation*}
and we conclude $\im Q_X=X$.
\end{proof}

The second paragraph of the above proof shows that $Q_X$ is surjective (and therefore an isomorphism) precisely when $X$ is the union of submodules which are objects of $\cC$. Every vector in such a weak module $X$ is an element of a (grading-restricted) generalized submodule, and hence is the sum of generalized $L(0)$-eigenvectors. So $X$ is actually a generalized module (see \cite[Remark 2.13]{HLZ1}). We can rephrase these observations as an alternative characterization of $\ind(\cC)$:
\begin{prop}\label{prop:indC_as_unions}
 The category $\ind(\cC)$ is the full subcategory of generalized $V$-modules that are unions of submodules which are objects of $\cC$.
\end{prop}

This proposition shows in particular that $\cC$ itself is a subcategory of $\ind(\cC)$ (as it should be). In fact, when $W$ is an object of $\cC$, it is easy to show that the inverse of $Q_W$ is $\phi_{W}: W\rightarrow\varinjlim\alpha_W$. We can also use the characterization of $\ind(\cC)$ in the proposition to show that $\ind(\cC)$ is an abelian category:
\begin{prop}\label{prop:IndC_abelian}
 The direct limit completion $\ind(\cC)$ is a $\CC$-linear abelian category and is closed under arbitrary coproducts.
\end{prop}
\begin{proof}
 Using Proposition \ref{prop:indC_as_unions}, any direct sum of generalized modules in $\ind(\cC)$ is still a union of $\cC$-submodules and hence is still an object of $\ind(\cC)$. Then to show that $\ind(\cC)$ is abelian, in particular closed under kernels and cokernels, we just need to show that $\ind(\cC)$ is closed under submodules and quotients.

 Suppose $\til{X}$ is any submodule of a generalized module $X$ in $\ind(\cC)$. By Proposition \ref{prop:indC_as_unions}, any $b\in\til{X}$ is contained in a $\cC$-submodule $W\subseteq X$. Since $\cC$ is closed under submodules, $W\cap\til{X}$ is a $\cC$-submodule of $\til{X}$ that contains $b$. It follows that $\til{X}$ is a union of submodules from $\cC$ and thus is an object of $\ind(\cC)$.

 Now consider the quotient $X/\til{X}$ where $X$ is a generalized module in $\ind(\cC)$. For any $b+\til{X}\in X/\til{X}$, there is a $\cC$-submodule $W\subseteq X$ such that $b\in W$, so that $b+\til{X}\in (W+\til{X})/\til{X}\cong W/W\cap\til{X}$. Since $\cC$ is closed under quotients, $b+\til{X}$ is contained in a $\cC$-submodule of $X/\til{X}$ and we conclude that $X/\til{X}$ is the union of its $\cC$-submodules. Consequently, $X/\til{X}$ is an object of $\ind(\cC)$, so that $\ind(\cC)$ is closed under quotients.
\end{proof}

We have now seen that abelian category structure on $\cC$ extends to abelian category structure on $\ind(\cC)$. In the
next section, we will show how to extend braided tensor category structure on $\cC$ to $\ind(\cC)$. In preparation for this, we conclude this section by describing how to extend any functor of the form
\begin{equation*}
 \cF: \cC^n\rightarrow\cC,
\end{equation*}
for $n\in\ZZ_+$, to a functor $\cW^n\rightarrow\ind(\cC)$; for example, in the next section we will take $n=2$ and $\cF$ a tensor product bifunctor on $\cC$. For such a functor $\cF$ and for weak modules $X_1,\ldots, X_n$, we define the direct system
\begin{equation*}
 \alpha_{\cF; X_1,\ldots, X_n}: I_{X_1}\times\cdots\times I_{X_n}\rightarrow\cC
\end{equation*}
on objects by
\begin{equation*}
 \alpha_{\cF;X_1,\ldots, X_n}(W_1,\ldots,W_n)=\cF(W_1,\ldots, W_n)
\end{equation*}
for $W_k\in I_{X_k}$, $k=1,\ldots, n$. For $\cC$-submodules $W_k\subseteq\til{W}_k$, $k=1,\ldots, n$, in $I_{X_k}$, we define the morphism $f_{W_1,\ldots, W_n}^{\til{W}_1,\ldots,\til{W}_n}=\alpha_{\cF;X_1,\ldots,X_n}((W_1,\ldots, W_n)\rightarrow(\til{W}_1,\ldots,\til{W}_n))$ to be
\begin{equation}\label{eqn:alphaF_on_morphisms}
 f_{W_1,\ldots, W_n}^{\til{W}_1,\ldots,\til{W}_n}=\cF(f_{W_1}^{\til{W}_1},\ldots, f_{W_n}^{\til{W}_n}).
\end{equation}

We can now define a functor $\widehat{\cF}:\cW^n\rightarrow\ind(\cC)$ on objects by
\begin{equation*}
 \widehat{\cF}(X_1,\ldots,X_n)=\varinjlim\alpha_{\cF;X_1,\ldots,X_n}.
\end{equation*}
We also need to define $\widehat{\cF}(F_1,\ldots,F_n)$ for morphisms $F_k: X_k\rightarrow\til{X}_k$ in $\cW$, $k=1,\ldots, n$. For this, observe that $F_k(W_k)=\im(F_k\circ i_{W_k})$ is a quotient of $W_k$ and hence a $\cC$-submodule of $\til{X}_k$. Thus if we use $F_k\vert_{W_k}: W_k\rightarrow F_k(W_k)$ to denote the $\cC$-morphism induced by $F_k\circ i_{W_k}$, we can attempt to define $\widehat{\cF}(F_1,\ldots,F_n)$ to be the unique homomorphism such that
\begin{equation*}
 \xymatrixcolsep{6.5pc}
 \xymatrix{
 \cF(W_1,\ldots, W_n) \ar[d]_{\phi_{W_1,\ldots, W_n}} \ar[r]^(.44){\cF(F_1\vert_{W_1},\ldots,F_n\vert_{W_n})} & \cF(F_1(W_1),\ldots,F_n(W_n)) \ar[d]^{\phi_{F_1(W_1),\ldots,F_n(W_n)}}  \\
 \varinjlim\alpha_{\cF;X_1,\ldots, X_n} \ar[r]_{\widehat{\cF}(F_1,\ldots, F_n)} & \varinjlim\alpha_{\cF;\til{X}_1,\ldots\til{X}_n} \\
 }
\end{equation*}
commutes for $(W_1,\ldots, W_n)\in I_{X_1}\times\cdots\times I_{X_n}$. From the universal property of the direct limit $\varinjlim \alpha_{\cF;X_1,\ldots,X_n}$, the morphism $\widehat{\cF}(F_1,\ldots,F_n)$ will indeed exist provided that
\begin{equation*}
 \left(\varinjlim\alpha_{\cF;\til{X}_1,\ldots,\til{X}_n}, \lbrace\phi_{\cF_1(W_1),\ldots,F_n(W_n)}\circ\cF(F_1\vert_{W_1},\ldots F_n\vert_{W_n})\rbrace_{(W_1,\ldots,W_n)\in I_{X_1}\times\cdots \times I_{X_n}}\right)
\end{equation*}
is a target of the direct system $\alpha_{\cF;X_1,\ldots,X_n}$. That is, we need to show
\begin{align*}
 \phi_{F_1(\til{W}_1),\ldots,F_n(\til{W}_n)} & \circ  \cF(F_1\vert_{\til{W}_1},\ldots,F_n\vert_{\til{W}_n})  \circ  f_{W_1,\ldots, W_n}^{\til{W}_1,\ldots,\til{W}_n} \nonumber\\
 &=  \phi_{F_1(W_1),\ldots,F_n(W_n)}\circ\cF\left(F_1\vert_{W_1},\ldots,F_n\vert_{W_n}\right)
\end{align*}
whenever $W_k\subseteq\til{W}_k$  are $\cC$-submodules in $I_{X_k}$ for $k=1,\ldots, n$. Using \eqref{eqn:alphaF_on_morphisms}, this follows from
\begin{align*}
& \phi_{F_1(\til{W}_1),\ldots,F_n(\til{W}_n)}\circ  \cF(F_1\vert_{\til{W}_1},\ldots,F_n\vert_{\til{W}_n})\circ\cF\left(f_{W_1}^{\til{W}_1},\ldots, f_{W_n}^{\til{W}_n}\right)\nonumber\\
& \hspace{4em}= \phi_{F_1(\til{W}_1),\ldots,F_n(\til{W}_n)}\circ\cF\left(F_1\vert_{\til{W}_1}\circ f_{W_1}^{\til{W}_1},\ldots,F_n\vert_{\til{W}_n}\circ f_{W_n}^{\til{W}_n}\right)\nonumber\\
& \hspace{4em}= \phi_{F_1(\til{W}_1),\ldots,F_n(\til{W}_n)}\circ\cF\left((f_{F_1(W_1)}^{F_1(\til{W}_1)}\circ F_1\vert_{W_1},\ldots,f_{F_n(W_n)}^{F_n(\til{W}_n)}\circ F_n\vert_{W_n}\right)\nonumber\\
& \hspace{4em} =  \phi_{F_1(\til{W}_1),\ldots,F_n(\til{W}_n)}\circ\cF\left(f_{F_1(W_1)}^{F_1(\til{W}_1)},\ldots, ,f_{F_n(W_n)}^{F_n(\til{W}_n)}\right)\circ\cF\left(F_1\vert_{W_1},\ldots,F_n\vert_{W_n}\right)\nonumber\\
& \hspace{4em} =  \phi_{F_1(W_1),\ldots,F_n(W_n)}\circ\cF\left(F_1\vert_{W_1},\ldots,F_n\vert_{W_n}\right),
\end{align*}
so $\widehat{\cF}(F_1,\ldots,F_n)$ is well defined.

To verify that $\widehat{\cF}$ is actually a functor, we first observe that $\widehat{\cF}(\Id_{X_1},\ldots,\Id_{X_n})$ satisfies
\begin{align*}
 \widehat{\cF}(\Id_{X_1},\ldots,\Id_{X_n})\circ\phi_{W_1,\ldots, W_n} =\phi_{W_1,\ldots, W_n}\circ\cF(\Id_{W_1},\ldots,\Id_{W_n})=\phi_{W_1,\ldots, W_n}
\end{align*}
for all $(W_1,\ldots, W_n)\in I_{X_1}\times\cdots\times I_{X_n}$ and thus is the identity on $\varinjlim\alpha_{\cF;X_1,\ldots,X_n}$. Then, for two composable morphisms $(F_1,\ldots, F_n)$ and $(G_1,\ldots, G_n)$ in $\cW^n$, we have
\begin{align*}
 \widehat{\cF}(F_1,\ldots, & F_n)\circ  \widehat{\cF}(G_1,\ldots, G_n)\circ\phi_{W_1,\ldots W_n}\nonumber\\
 &= \widehat{\cF}(F_1,\ldots,F_n)\circ\phi_{G_1(W_1),\ldots,G_n(W_n)}\circ\cF(G_1\vert_{W_1},\ldots,G_n\vert_{W_n})\nonumber\\
 &=\phi_{F_1(G_1(W_1)),\ldots,F_n(G_n(W_n))}\circ\cF(F_1\vert_{G_1(W_1)},\ldots, F_n\vert_{G_n(W_n)})\circ\cF(G_1\vert_{W_1},\ldots,G_n\vert_{W_n})\nonumber\\
 & =\phi_{(F_1\circ G_1)(W_1),\ldots,(F_n\circ G_n)(W_n)}\circ\cF((F_1\circ G_1)\vert_{W_1},\ldots,(F_n\circ G_n)\vert_{W_n})\nonumber\\
 & =\widehat{\cF}(F_1\circ G_1,\ldots,F_n\circ G_n)\circ\phi_{W_1,\ldots, W_n}
\end{align*}
for all $(W_1,\ldots, W_n)\in I_{X_1}\times\cdots\times I_{X_n}$, and thus we must have
\begin{equation*}
 \widehat{\cF}(F_1,\ldots, F_n)\circ\widehat{\cF}(G_1,\ldots, G_n) =\widehat{\cF}(F_1\circ G_1,\ldots,F_n\circ G_n).
\end{equation*}
This proves that $\widehat{\cF}:\cW^n\rightarrow\ind(\cC)$ indeed satisfies the properties of a functor.

We have now shown that a functor $\cF: \cC^n\rightarrow\cC$ can be extended to a functor $\widehat{\cF}:\cW^n\rightarrow\ind(\cC)$. We can also extend any natural transformation $\Psi:\cF\rightarrow\cG$ between two functors $\cF,\cG: \cC^n\rightarrow\cC$ to a natural transformation $\widehat{\Psi}: \widehat{\cF}\rightarrow\widehat{\cG}$.  Indeed, for objects $X_1,\ldots, X_n$ in $\cW$, we define
\begin{equation*}
 \widehat{\Psi}_{X_1,\ldots,X_n}: \widehat{\cF}(X_1,\ldots, X_n)\rightarrow\widehat{\cG}(X_1,\ldots, X_n)
\end{equation*}
to be the unique morphism such that the diagram
\begin{equation*}
 \xymatrixcolsep{4pc}
 \xymatrix{
 \cF(W_1,\ldots, W_n) \ar[r]^{\Psi_{W_1,\ldots, W_n}} \ar[d]_{\phi_{W_1,\ldots, W_n}} & \cG(W_1,\ldots, W_n) \ar[d]^{\psi_{W_1,\ldots, W_n}} \\
 \varinjlim\alpha_{\cF;X_1,\ldots, X_n} \ar[r]_{\widehat{\Psi}_{X_1,\ldots, X_n}} & \varinjlim\alpha_{\cG;X_1\ldots, X_n} \\
 }
\end{equation*}
commutes for all $(W_1,\ldots, W_n)\in I_{X_1}\times\cdots\times I_{X_n}$. The existence and uniqueness of $\widehat{\Psi}_{X_1,\ldots,X_n}$ follows from the universal property of the direct limit $\varinjlim \alpha_{\cF;X_1,\ldots,X_n}$ together with the calculation
\begin{align*}
 \psi_{\til{W}_1,\ldots,\til{W}_n}\circ \Psi_{\til{W}_1,\ldots,\til{W}_n} &\circ  \cF(f_{W_1}^{\til{W}_n},\ldots, f_{W_n}^{\til{W}_n})\nonumber\\
 &= \psi_{\til{W}_1,\ldots,\til{W}_n}\circ\cG(f_{W_1}^{\til{W}_1},\ldots,f_{W_n}^{\til{W}_n})\circ\Psi_{W_1,\ldots, W_n}\nonumber\\
 &= \psi_{W_1,\ldots,W_n}\circ\Psi_{W_1,\ldots,W_n},
\end{align*}
where $W_k\subseteq\til{W}_k$ are $\cC$-submodules in $I_{X_k}$ for each $k$. To show that $\widehat{\Psi}$ is actually a natural transformation, let $F_k: X_k\rightarrow\til{X}_k$, $k=1,\ldots, n$, be morphisms in $\cW$ and consider the commutative diagrams
\begin{equation*}
 \xymatrixcolsep{6pc}
 \xymatrix{
 \cF(W_1,\ldots, W_n) \ar[r]^{\Psi_{W_1,\ldots, W_n}} \ar[d]_{\phi_{W_1,\ldots, W_n}} & \cG(W_1,\ldots, W_n) \ar[d]^{\psi_{W_1,\ldots, W_n}} \ar[r]^(.44){\cG(F_1\vert_{W_1},\ldots F_n\vert_{W_n})} & \cG(F_1(W_1),\ldots,F_n(W_n)) \ar[d]^{\psi_{F_1(W_1),\ldots,F_n(W_n)}}  \\
 \varinjlim\alpha_{\cF;X_1,\ldots, X_n} \ar[r]_{\widehat{\Psi}_{X_1,\ldots, X_n}} & \varinjlim\alpha_{\cG;X_1\ldots, X_n} \ar[r]_{\widehat{\cG}(F_1,\ldots,F_n)} & \varinjlim\alpha_{\cG;\til{X}_1,\ldots,\til{X}_n} \\
 }
\end{equation*}
and
\begin{equation*}
 \xymatrixcolsep{5.2pc}
 \xymatrix{
\cF(W_1,\ldots, W_n) \ar[d]_{\phi_{W_1,\ldots,W_n}} \ar[r]^(.44){\cF(F_1\vert_{W_1}\,\ldots F_n\vert_{W_n})} &  \cF(F_1(W_1),\ldots, F_n(W_n)) \ar[r]^{\Psi_{F_1(W_1),\ldots, F_n(W_n)}} \ar[d]_{\phi_{F_1(W_1),\ldots, F_n(W_n)}} & \cG(F_1(W_1),\ldots, F_n(W_n)) \ar[d]_{\psi_{F_1(W_1),\ldots, F_n(W_n)}}  \\
\varinjlim\alpha_{\cF;X_1,\ldots, X_n} \ar[r]_{\widehat{\cF}(F_1,\ldots,F_n)} & \varinjlim\alpha_{\cF;\til{X}_1,\ldots, \til{X}_n} \ar[r]_{\widehat{\Psi}_{\til{X}_1,\ldots, \til{X}_n}} & \varinjlim\alpha_{\cG;\til{X}_1\ldots, \til{X}_n}   \\
 }
\end{equation*}
 for any $(W_1,\ldots, W_n)\in I_{X_1}\times\cdots\times I_{X_n}$. The top rows of both diagrams are identical because $\Psi$ is a natural transformation, so we get
 \begin{equation*}
  \widehat{\cG}(F_1,\ldots,F_n)\circ\widehat{\Psi}_{X_1,\ldots,X_n}\circ\phi_{W_1,\ldots,W_n}=\widehat{\Psi}_{\til{X}_1,\ldots,\til{X}_n}\circ\widehat{\cF}(F_1,\ldots,F_n)\circ\phi_{W_1,\ldots,W_n}
 \end{equation*}
for all $(W_1,\ldots,W_n)\in I_{X_1}\times\cdots\times I_{X_n}$, and thus also
\begin{equation*}\
 \widehat{\cG}(F_1,\ldots,F_n)\circ\widehat{\Psi}_{X_1,\ldots,X_n}=\widehat{\Psi}_{\til{X}_1,\ldots,\til{X}_n}\circ\widehat{\cF}(F_1,\ldots,F_n)
\end{equation*}
since the direct limit $\varinjlim\alpha_{\cF;X_1,\ldots,X_n}$ is spanned by the images of the $\phi_{W_1,\ldots,W_n}$. This proves that $\widehat{\Psi}$ is a natural transformation.

The first example of a functor $\cF: \cC^n\rightarrow\cC$ that we would like to consider is the identity functor $\Id_\cC: \cC\rightarrow\cC$. In this case,  $\widehat{\Id}_\cC: \cW\rightarrow\ind(\cC)$ is defined on objects by
  \begin{equation*}
   \widehat{\Id}_\cC(X) =\varinjlim\alpha_X,
  \end{equation*}
while for a morphism $F: X\rightarrow\til{X}$ in $\cW$, $\widehat{\Id}_\cC(F)$ is characterized by the commutativity of
\begin{equation*}
\xymatrixcolsep{3pc}
 \xymatrix{
 W \ar[d]_{\phi_W} \ar[r]^{F\vert_W} & F(W) \ar[d]^{\phi_{F(W)}} \\
 \varinjlim\alpha_X \ar[r]_{\widehat{\Id}_\cC(F)} & \varinjlim\alpha_{\til{X}} \\
 }
\end{equation*}
for all $W\in I_X$. The morphisms $Q_X: \varinjlim\alpha_X\rightarrow X$ in Proposition \ref{prop:alpha_X} for weak modules $X$  determine a natural transformation $Q: \widehat{\Id}_\cC\rightarrow\Id_\cW$: to show this, we use commutative diagrams
\begin{equation*}
\xymatrixcolsep{3pc}
 \xymatrix{
  W \ar[d]_{\phi_W} \ar[r]^{F\vert_W}  & F(W) \ar[d]_{\phi_{F(W)}} \ar[rd]^{i_{F(W)}} & \\
 \varinjlim\alpha_X \ar[r]_{\widehat{\Id}_\cC(F)} & \varinjlim\alpha_{\til{X}} \ar[r]_(.6){Q_{\til{X}}} & \til{X}\\
 } \quad\text{and}\qquad
 \xymatrix{
 W \ar[d]_{\phi_W} \ar[rd]^{i_W} & & \\
 \varinjlim\alpha_X\ar[r]_{Q_X} & X \ar[r]_{F} & \til{X} \\
 }
\end{equation*}
for any morphism $F: X\rightarrow\til{X}$ and for $W\in I_X$. Then $Q_{\til{X}}\circ\widehat{\Id}_\cC(F)=F\circ Q_X$ as required since
\begin{equation*}
 i_{F(W)}\circ F\vert_W = F\circ i_W.
\end{equation*}
for all $W\in I_X$. By Proposition \ref{prop:alpha_X}, $Q$ is a natural isomorphism when restricted to $\ind(\cC)$.

\section{The direct limit completion as a braided tensor category}\label{sec:ind_tens_cat}

We continue to work under Assumption \ref{assum:Sec4} and now additionally impose:
\begin{assum}\label{assum:Sec5}
The category $\cC$ of grading-restricted generalized $V$-modules satisfies:
\begin{enumerate}

\item The vertex operator algebra $V$ is an object of $\cC$.

\item The category $\cC$ is a braided tensor category with unit object $V$. We will use $\tens$ to denote the tensor product bifunctor on $\cC$, $l$ and $r$ for the left and right unit isomorphisms, $\cA$ for the associativity isomorphisms, and $\cR$ for the braiding isomorphisms.

\item The braided tensor category $\cC$ has a twist natural isomorphism $\theta: \Id_\cC\rightarrow\Id_\cC$ such that $\theta_V=\Id_V$ and the balancing equation
\begin{equation*}
 \theta_{W_1\tens W_2}=\cR_{W_2,W_1}\circ\cR_{W_1,W_2}\circ(\theta_{W_1}\tens\theta_{W_2})
\end{equation*}
holds for modules $W_1$, $W_2$ in $\cC$.

 \item For any module $W$ in $\cC$, the functors $W\tens\bullet$ and $\bullet\tens W$ are right exact.
\end{enumerate}
\end{assum}
\begin{rem}
For this section, we do not assume that the braided tensor category structure on $\cC$ is necessarily the vertex algebraic one of \cite{HLZ8} described in Section \ref{sec:vertex_tensor}. In fact, the results of this section hold for more general abelian categories $\cC$ and $\cW$ satisfying suitable assumptions. If the braided tensor category structure on $\cC$ is indeed the vertex algebraic one, the right exactness of $W\tens\bullet$ and $\bullet\tens W$ are automatic \cite[Proposition 4.26]{HLZ3}.
\end{rem}

Our goal in this section is to extend the braided tensor category structure on $\cC$ to $\ind(\cC)$ in a natural way. To obtain a tensor product bifunctor on $\ind(\cC)$, and indeed on $\cW$, we notice that $\tens:\cC\times\cC\rightarrow\cC$ induces
   $$\widehat{\tens}: \cW\times\cW\rightarrow\ind(\cC),$$
as described at the end of the preceding section.   For objects $X_1$, $X_2$ in $\cW$, we use the notation $\alpha_{X_1}\tens\alpha_{X_2}$ for the direct system $\alpha_{\tens;X_1,X_2}$ of the preceding section, so that by definition,
\begin{equation*}
 X_1\widehat{\tens} X_2=\varinjlim\alpha_{X_1}\tens\alpha_{X_2}.
\end{equation*}
Then the tensor product of morphisms $F_1:X_1\rightarrow\til{X}_1$ and $F_2: X_2\rightarrow\til{X}_2$ in $\cW$ is characterized by the commutative diagrams
\begin{equation*}
 \xymatrixcolsep{4pc}
 \xymatrix{
 W_1\tens W_2 \ar[d]_{\phi_{W_1,W_2}} \ar[r]^(.42){F_1\vert_{W_1}\tens F_2\vert_{W_2}}  & F_1(W_1)\tens F_2(W_2) \ar[d]^{\phi_{F_1(W_1),F_2(W_2)}} \\
 \varinjlim\alpha_{X_1}\tens\alpha_{X_2} \ar[r]_{F_1\widehat{\tens} F_2} & \varinjlim\alpha_{\til{X}_1}\tens\alpha_{\til{X}_2} \\
 }
\end{equation*}
for $\cC$-submodules $W_1\subseteq X_1$ and $W_2\subseteq X_2$.

We proceed to construct unit, associativity, braiding, and twist isomorphisms on $\ind(\cC)$. We will see that the braiding on $\ind(\cC)$ can be obtained simply as the natural isomorphism $\widehat{\cR}$ induced from the braiding isomorphisms in $\cC$, as described in the previous section. The unit and associativity isomorphisms on $\ind(\cC)$ cannot quite be obtained in this way, so we will need additional notation: we will denote them using fraktur fonts to distinguish them from the structure isomorphisms of the original braided tensor category $\cC$. Actually, we will construct unit, associativity, and braiding homomorphisms on the entire weak module category $\cW$, but the left and right unit homomorphisms $\mathfrak{l}_X$ and $\mathfrak{r}_X$ will not be isomorphisms unless $X$ is an object of $\ind(\cC)$. For example, the definition of $\widehat{\tens}$ shows that $V\widehat{\tens} X =0$ if $X$ is a weak $V$-module whose only $\cC$-submodule is $0$, so in this case $\mathfrak{l}_X$ would be $0$.

The unit object of $\ind(\cC)$ will be $V$, just as in $\cC$. For a weak module $X$, we would like to define left and right unit morphisms $\mathfrak{l}_X: V\widehat{\tens} X\rightarrow X$ and $\mathfrak{r}_X: X\widehat{\tens}V\rightarrow X$ to be the unique morphisms such that the diagrams
 \begin{equation*}
 \xymatrixcolsep{3.5pc}
 \xymatrix{
 U\tens W \ar[d]_{\phi_{U,W}} \ar[r]^{i_U\tens\Id_{W}} &V\tens W \ar[d]^{i_W\circ l_{W}} \\
 \varinjlim\alpha_V\tens\alpha_X \ar[r]_(.6){\mathfrak{l}_X} & X \\
 } \quad\text{and}\quad
 \xymatrix{
 W\tens U \ar[d]_{\phi_{W,U}} \ar[r]^{\Id_{W}\tens i_U} & W\tens V \ar[d]^{i_W\circ r_{W}} \\
 \varinjlim\alpha_X\tens\alpha_V \ar[r]_(.6){\mathfrak{r}_X} & X \\
 }
 \end{equation*}
commute for all objects $(U,W)\in I_V\times I_X$. The universal property of $\varinjlim\alpha_V\tens\alpha_X$ guarantees that $\mathfrak{l}_X$ will exist provided
\begin{equation*}
 \left(X,\lbrace i_W\circ l_W\circ(i_U\tens\Id_W)\rbrace_{(U,W)\in I_V\times I_X}\right)
\end{equation*}
is a target of the direct system $\alpha_V\tens\alpha_X$. This follows from the calculation
\begin{align*}
 i_{\til{W}}\circ l_{\til{W}}\circ(i_{\til{U}}\tens\Id_{\til{W}})\circ(f_U^{\til{U}}\tens f_{W}^{\til{W}}) & = i_{\til{W}}\circ l_{\til{W}}\circ(i_U\tens f^{\til{W}}_W)\nonumber\\
 & =i_{\til{W}}\circ f^{\til{W}}_W\circ l_{W}\circ(i_U\tens\Id_{W})\nonumber\\
 & =i_W\circ l_{W}\circ(i_U\tens\Id_{W})
\end{align*}
for any $U\subseteq\til{U}$ in $I_V$ and $W\subseteq\til{W}$ in $I_X$, and $\mathfrak{r}_X$ is similarly well defined.

To show that $\mathfrak{l}: V\widehat{\tens}\bullet\rightarrow\Id_\cW$ is a natural transformation, we need to show that 
\begin{equation}\label{eqn:l_nat_trans}
\mathfrak{l}_{\til{X}}\circ(\Id_V\widehat{\tens}F)=F\circ\mathfrak{l}_X
\end{equation}
for any morphism $F: X\rightarrow\til{X}$ in $\cW$. Using the commutative diagram
\begin{equation*}
  \xymatrixcolsep{3.5pc}
 \xymatrix{
 U\tens W \ar[r]^(.47){\Id_U\tens F\vert_W} \ar[d]_{\phi_{U,W}} &
 U\tens F(W) \ar[d]_{\phi_{U,F(W)}} \ar[r]^(.5){i_U\tens\Id_{F(W)}} & V\tens F(W) \ar[d]^{i_{F(W)}\circ l_{F(W)}} \\
 \varinjlim\alpha_{V}\tens\alpha_X \ar[r]_{\Id_{V}\widehat{\tens} F} & \varinjlim\alpha_{V}\tens\alpha_{\til{X}} \ar[r]_(.6){\mathfrak{l}_{\til{X}}} & \til{X} \\
 }
\end{equation*}
for $(U,W)\in I_V\times I_X$, we get
\begin{align*}
 \mathfrak{l}_{\til{X}}\circ(\Id_V\widehat{\tens}F)\circ\phi_{U,W} & = i_{F(W)}\circ l_{F(W)}\circ(i_U\tens F\vert_W)\nonumber\\
 &= i_{F(W)}\circ F\vert_W\circ l_W\circ(i_U\tens\Id_W)\nonumber\\
 &= F\circ i_W\circ l_W\circ(i_U\tens\Id_W)\nonumber\\
 &=F\circ \mathfrak{l}_X\circ\phi_{U,W}
\end{align*}
for all $(U,W)\in I_V\times I_X$, which proves \eqref{eqn:l_nat_trans} since $\varinjlim\alpha_V\tens\alpha_X$ is spanned by the images of the $\phi_{U,W}$. The proof that $\mathfrak{r}: \bullet\widehat{\tens}\vac\rightarrow\Id_\cW$ is a natural transformation is similar.

Although $\mathfrak{l}_X$ and $\mathfrak{r}_X$ cannot be isomorphisms when $X$ is not in $\ind(\cC)$, we do have:
\begin{prop}
 For any weak $V$-module $X$, $\mathfrak{l}_X$ and $\mathfrak{r}_X$ are injective. If $X$ is an object of $\ind(\cC)$, then $\mathfrak{l}_X$ and $\mathfrak{r}_X$ are also surjective.
\end{prop}
\begin{proof}
 We give the proofs for $\mathfrak{l}_X$, since the proofs for $\mathfrak{r}_X$ are essentially the same. To show that $\mathfrak{l}_X$ is injective, suppose $\mathfrak{l}_X(b)=0$ for $b\in V\widehat{\tens}X$. We may assume that $b=\phi_{U,W}(\til{w})$ for $U\subseteq V$ in $I_V$, $W\subseteq X$ in $I_X$, and $\til{w}\in U\tens W$. Then
 \begin{equation*}
  0=\mathfrak{l}_X(b)=\mathfrak{l}_X(\phi_{U,W}(\til{w}))=(i_W\circ l_W\circ(i_U\tens\Id_W))(\til{w}) =(i_W\circ l_W)\left((f_U^V\tens f_W^W)(\til{w})\right).
 \end{equation*}
 Since $i_W\circ l_W$ is injective, this means $(f_U^V\tens f_W^W)(\til{w})=0,$ and then
 \begin{equation*}
  b=\phi_{U,W}(\til{w})=\left(\phi_{V,W}\circ(f^V_U\tens f^W_W)\right)(\til{w})=0
 \end{equation*}
as well. Thus $\ker \mathfrak{l}_X=0$.

To show that $\mathfrak{l}_X$ is surjective when $X$ is in $\ind(\cC)$, suppose $b\in X$. Since $X$ is a union of $\cC$-submodules by Proposition \ref{prop:indC_as_unions}, we have $b=i_W(w)$ for some $W\in I_X$ and $w\in W$. Then because $l_W$ is surjective, $w=l_W(\til{w})$ for some $\til{w}\in V\tens W$, so
\begin{equation*}
 b=i_W(l_W(\til{w})) =\mathfrak{l}_X(\phi_{V,W}(\til{w}))\in\im \mathfrak{l}_X.
\end{equation*}
Thus $\im \mathfrak{l}_X=X$.
\end{proof}

The next task is to construct natural associativity isomorphisms in $\ind(\cC)$ (and indeed in $\cW$). The triple tensoring functors $\tens\circ(\Id_\cC\times\tens), \tens\circ(\tens\times\Id_\cC): \cC\times\cC\times\cC\rightarrow\cC$ induce functors $\cW\times\cW\times\cW\rightarrow\ind(\cC)$; we denote them by
\begin{equation*}
 (X_1,X_2,X_3)\mapsto\varinjlim\alpha_{X_1}\tens(\alpha_{X_2}\tens\alpha_{X_3}),\qquad(X_1,X_2,X_3)\mapsto\varinjlim(\alpha_{X_1}\tens\alpha_{X_2})\tens\alpha_{X_3}
\end{equation*}
on objects. Moreover, as in the previous section, there is a natural isomorphism $\widehat{\cA}$ between these two functors induced by the natural associativity isomorphism $\cA: \tens\circ(\Id_\cC\times\tens)\rightarrow\tens\circ(\tens\times\Id_\cC)$. However, these two functors from $\cW\times\cW\times\cW$ to $\ind(\cC)$ are \textit{not} equal to
\begin{equation*}
 (X_1,X_2,X_3)\mapsto X_1\widehat{\tens}(X_2\widehat{\tens} X_3),\qquad (X_1,X_2,X_3)\mapsto (X_1\widehat{\tens}X_2)\widehat{\tens} X_3.
\end{equation*}
To get associativity isomorphisms on $\cW$, we need an analogue of Fubini's Theorem showing that the two iterated tensor products are isomorphic to the two multiple tensor products.

 We first need to construct  natural transformations
\begin{align*}
 T_{X_1,(X_2,X_3)}: & \varinjlim\alpha_{X_1}\tens(\alpha_{X_2}\tens\alpha_{X_3}) \rightarrow X_1\widehat{\tens}(X_2\widehat{\tens} X_3)\quad\left( \;\; = \varinjlim\alpha_{X_1}\tens\alpha_{X_2\widehat{\tens} X_3} \;\; \right)\nonumber\\
 T_{(X_1,X_2),X_3}: & \varinjlim\, (\alpha_{X_1}\tens\alpha_{X_2})\tens\alpha_{X_3} \rightarrow (X_1\widehat{\tens}X_2)\widehat{\tens} X_3\quad\left( \;\; = \varinjlim\alpha_{X_1\widehat{\tens}X_2}\tens\alpha_{X_3} \;\; \right)
\end{align*}
Actually, for proving the pentagon axiom later, we need more general homomorphisms. For weak modules $X_1$, $X_2$ and a direct system $\alpha: I\rightarrow\cC$, we define the direct systems
\begin{align*}
 \alpha\tens(\alpha_{X_1}\tens\alpha_{X_2}): I\times I_{X_1}\times I_{X_2} & \rightarrow \cC\nonumber\\
 (i, W_1, W_2) & \mapsto W(i)\tens(W_1\tens W_2)\nonumber\\
 f_{i,W_1,W_2}^{j,\til{W}_1,\til{W}_2} & = f^j_i\tens f_{W_1}^{\til{W}_1}\tens f_{W_2}^{\til{W}_2}
\end{align*}
and
\begin{align*}
 \alpha\tens\alpha_{X_1\widehat{\tens} X_2}: I\times I_{X_1\widehat{\tens}X_2} & \rightarrow\cC\nonumber\\
 (i,W_{1,2}) & \mapsto W(i)\tens W_{1,2}\nonumber\\
 f_{i,W_{1,2}}^{j,\til{W}_{1,2}} & = f_i^j\tens f_{W_{1,2}}^{\til{W}_{1,2}}.
\end{align*}
Recall that $W_{1,2}$ and $\til{W}_{1,2}$ here are $\cC$-submodules of $X_1\widehat{\tens} X_2$; we cannot necessarily assume they factor as tensor products of $\cC$-submodules in $X_1$ and $X_2$. We define $(\alpha_{X_1}\tens\alpha_{X_2})\tens\alpha$, $\alpha_{X_1\widehat{\tens}X_2}\tens\alpha$ similarly. Also for $W_1\in I_{X_1}$ and $W_2\in I_{X_2}$, we write the morphism $\phi_{W_1,W_2}:W_1\tens W_2\rightarrow X_1\widehat{\tens} X_2$ as a composition $\phi_{W_1,W_2}=i_{W_1,W_2}\circ\phi_{W_1,W_2}'$, where
\begin{equation*}
 \phi_{W_1,W_2}': W_1\tens W_2\rightarrow\im\phi_{W_1,W_2}
\end{equation*}
is the surjection induced by $\phi_{W_1,W_2}$ and
\begin{equation*}
 i_{W_1,W_2}: \im\phi_{W_1,W_2}\rightarrow X_1\widehat{\tens} X_2
\end{equation*}
is the inclusion of the image into $X_1\widehat{\tens}X_2$.
Now we can attempt to define
 \begin{align*}
 T_{\alpha,(X_1,X_2)}& :  \varinjlim\alpha\tens(\alpha_{X_1}\tens\alpha_{X_2}) \rightarrow\varinjlim \alpha\tens\alpha_{X_1\widehat{\tens} X_2}\nonumber\\
 T_{(X_1,X_2),\alpha} & :  \varinjlim\, (\alpha_{X_1}\tens\alpha_{X_2})\tens\alpha \rightarrow  \varinjlim\alpha_{X_1\widehat{\tens}X_2}\tens\alpha
\end{align*}
to be the unique homomorphisms such that the diagrams
\begin{equation*}
\xymatrixcolsep{4.5pc}
 \xymatrix{
 W(i)\tens(W_1\tens W_2) \ar[d]_{\phi_{i,W_1,W_2}} \ar[r]^(.5){\Id_{W(i)}\tens\phi_{W_1,W_2}'} & W(i)\tens \im\phi_{W_1,W_2} \ar[d]^{\phi_{i,\,\im\phi_{W_1,W_2}}}  \\
 \varinjlim\alpha\tens(\alpha_{X_1}\tens\alpha_{X_2}) \ar[r]_(.55){T_{\alpha,(X_1,X_2)}} &
 \varinjlim\alpha\tens\alpha_{X_1\widehat{\tens}X_2}
 }
\end{equation*}
and
\begin{equation*}
\xymatrixcolsep{4.5pc}
 \xymatrix{
 (W_1\tens W_2)\tens W(i) \ar[r]^(.5){\phi_{W_1,W_2}'\tens\Id_{W(i)}} \ar[d]_{\phi_{W_1,W_2,i}} & \im\phi_{W_1,W_2}\tens W(i) \ar[d]^{\phi_{\im\phi_{W_1,W_2},\,i}}  \\
 \varinjlim(\alpha_{X_1}\tens\alpha_{X_2})\tens\alpha \ar[r]_(.55){T_{(X_1,X_2),\alpha}} &
 \varinjlim\alpha_{X_1\widehat{\tens}X_2}\tens\alpha
 }
\end{equation*}
commute for $(i,W_1,W_2)\in I\times I_{X_1}\times I_{X_2}$. To verify that $T_{\alpha,(X_1,X_2)}$ exists, suppose $i\leq j$ in $I$ and $W_k\subseteq\til{W}_k$ in $I_{X_k}$ for $k=1,2$. Then
\begin{align*}
 \phi_{j,\,\im\phi_{\til{W}_1,\til{W}_2}} & \circ  (\Id_{W(j)}
 \tens\phi_{\til{W}_1,\til{W}_2}')\circ\left(f_i^j\tens(f_{W_1}^{\til{W}_1}\tens f_{W_2}^{\til{W}_2})\right) \nonumber\\
 & =\phi_{j,\,\im\phi_{\til{W}_1,\til{W}_2}}\circ\left(f_i^j\tens f_{\im\phi_{W_1,W_2}}^{\im\phi_{\til{W}_1,\til{W}_2}}\right)\circ(\Id_{W(i)}\tens\phi_{W_1,W_2}')\nonumber\\
 & =\phi_{i,\,\im\phi_{W_1,W_2}}\circ(\Id_{W(i)}\tens\phi_{W_1,W_2}'),
\end{align*}
where the first equality follows from
\begin{align*}
 i_{\til{W}_1,\til{W}_2}\circ\phi'_{\til{W}_1,\til{W}_2}\circ & (f_{W_1}^{\til{W}_1}\tens f_{W_2}^{\til{W}_2}) =\phi_{\til{W}_1,\til{W}_2}\circ (f_{W_1}^{\til{W}_1}\tens f_{W_2}^{\til{W}_2})\nonumber\\
 &= \phi_{W_1,W_2}=i_{W_1,W_2}\circ\phi_{W_1,W_2}' = i_{\til{W}_1,\til{W}_2}\circ f_{\im\phi_{W_1,W_2}}^{\im\phi_{\til{W}_1,\til{W}_2}}\circ\phi_{W_1,W_2}'
\end{align*}
together with the injectivity of $i_{\til{W}_1,\til{W}_2}$. Now the universal property of the direct limit $\varinjlim\alpha\tens(\alpha_{X_1}\tens\alpha_{X_2})$ shows that $T_{\alpha,(X_1,X_2)}$ exists, and $T_{(X_1,X_2),\alpha}$ is well defined by a similar argument.

\begin{prop}\label{prop:Fubini}
 The homomorphisms $T_{\alpha,(X_1,X_2)}$ and $T_{(X_1,X_2),\alpha}$ are isomorphisms.
\end{prop}
\begin{proof}
 We give the prove for $T_{\alpha,(X_1,X_2)}$, and the proof for $T_{(X_1,X_2),\alpha}$ is essentially the same. To show that $T_{\alpha,(X_1,X_2)}$ is injective, suppose $T_{\alpha,(X_1,X_2)}(b)=0$ for some
 \begin{equation*}
  b=\phi_{i,W_1,W_2}(w)\in\varinjlim\alpha\tens(\alpha_{X_1}\tens\alpha_{X_2}),
 \end{equation*}
where $i\in I$, $W_1\in I_{X_1}$, $W_2\in I_{X_2}$, and $w\in W(i)\tens(W_1\tens W_2)$. Thus
\begin{equation*}
 \phi_{i,\,\im\phi_{W_1,W_2}}\left((\Id_{W(i)}\tens\phi'_{W_1,W_2})(w)\right) =0.
\end{equation*}
Then by Lemma \ref{lem:ker_phi_i}, there exist $j\geq i$ in $I$ and $W_{1,2}\supseteq \im\phi_{W_1,W_2}$ in $I_{X_1\widehat{\tens}X_2}$ such that
\begin{equation*}
 \left(f_i^j\tens (f_{\im\phi_{W_1,W_2}}^{W_{1,2}}\circ\phi_{W_1,W_2}')\right)(w)=0.
\end{equation*}
Since Lemma \ref{lem:im_f_in_im_phi_i} implies that $W_{1,2}\subseteq \im\phi_{\til{W}_1,\til{W}_2}$ for some $\til{W}_1\in I_{X_1}$, $\til{W}_2\in I_{X_2}$, and since
\begin{equation*}
 f_{\im\phi_{W_1,W_2}}^{\im\phi_{\til{W}_1,\til{W}_2}}=f_{W_{1,2}}^{\im\phi_{\til{W}_1,\til{W}_2}}\circ f_{\im\phi_{W_1,W_2}}^{W_{1,2}},
\end{equation*}
we get
\begin{equation}\label{eqn:w_in_kernel}
 \left(f_i^j\tens (f_{\im\phi_{W_1,W_2}}^{\im\phi_{\til{W}_1,\til{W}_2}}\circ\phi_{W_1,W_2}')\right)(w)=0.
\end{equation}
Now, by Lemma \ref{lem:ker_fin_gen}, $\ker\phi_{\til{W}_1,\til{W}_2}=\ker(f_{\til{W}_1}^{U_1}\tens f_{\til{W}_2}^{U_2})$ for some $U_1\in I_{X_1}$, $U_2\in I_{X_2}$. Then because $I_{X_1}$ and $I_{X_2}$ are directed sets, we have $\til{U}_1\in I_{X_1}$ and $\til{U}_2\in I_{X_2}$ such that $W_1,U_1\subseteq \til{U}_1$ and $W_2,U_2\subseteq \til{U}_2$.

We now have the following commutative diagram:
\begin{equation*}
 \xymatrixcolsep{3pc}
 \xymatrix{
 & \im\phi_{\til{W}_1,\til{W}_2} \ar[d]^{i_{\til{W}_1,\til{W}_2}} & \\
 \im\phi_{W_1,W_2} \ar[ru]^{f_{\im\phi_{W_1,W_2}}^{\im\phi_{\til{W}_1,\til{W}_2}}\,\,} \ar[r]^{i_{W_1,W_2}} & X_1\widehat{\tens}X_2 & \til{W}_1\tens \til{W}_2 \ar[l]_{\phi_{\til{W}_1,\til{W}_2}} \ar[lu]_{\phi_{\til{W}_1,\til{W}_2}'} \ar[d]^{f_{\til{W}_1}^{U_1}\tens f_{\til{W}_2}^{U_2}} \\
 W_1\tens W_2 \ar[u]^{\phi_{W_1,W_2}'} \ar[ru]_{\phi_{W_1,W_2}} \ar[rd]_{f_{W_1}^{\til{U}_1}\tens f_{W_2}^{\til{U}_2}} & \im\phi_{\til{U}_1,\til{U}_2} \ar[u]_{i_{\til{U}_1,\til{U}_2}} & U_1\tens U_2 \ar[lu]_{\phi_{U_1,U_2}} \ar[ld]^(.4){ f_{U_1}^{\til{U}_1}\tens f_{U_2}^{\til{U}_2}} \\
 & \til{U}_1\tens \til{U}_2 \ar[u]_{\phi_{\til{U}_1,\til{U}_2}'} & \\
 }
\end{equation*}
At this point, the key observation is that, because $\ker\phi_{\til{W}_1,\til{W}_2}=\ker(f_{\til{W}_1}^{U_1}\tens f_{\til{W}_2}^{U_2})$, we have a homomorphism $F: \im\phi_{\til{W}_1,\til{W}_2}\rightarrow U_1\tens U_2$ such that
\begin{equation*}
 F\circ\phi_{\til{W}_1,\til{W}_2}'=f_{\til{W}_1}^{U_1}\tens f_{\til{W}_2}^{U_2}.
\end{equation*}
We claim that
\begin{equation}\label{eqn:F_reln}
 i_{\til{W}_1,\til{W}_2}=\phi_{U_1,U_2}\circ F.
\end{equation}
Indeed,
\begin{equation*}
 i_{\til{W}_1,\til{W}_2}\circ\phi_{\til{W}_1,\til{W}_2}' = \phi_{\til{W}_1,\til{W}_2} =\phi_{U_1,U_2}\circ(f_{\til{W}_1}^{U_1}\tens f_{\til{W}_2}^{U_2}) =\phi_{U_1,U_2}\circ F\circ\phi_{\til{W}_1,\til{W}_2}',
\end{equation*}
so that \eqref{eqn:F_reln} holds by the surjectivity of $\phi_{\til{W}_1,\til{W}_2}'$.

The commutative diagram and \eqref{eqn:F_reln} together imply that
\begin{equation*}
 \phi_{\til{U}_1,\til{U}_2}\circ(f_{W_1}^{\til{U}_1}\tens f_{W_2}^{\til{U}_2}) = \phi_{\til{U}_1,\til{U}_2}\circ(f_{U_1}^{\til{U}_1}\tens f_{U_2}^{\til{U}_2})\circ F\circ f_{\im\phi_{W_1,W_2}}^{\im\phi_{\til{W}_1,\til{W}_2}}\circ\phi_{W_1,W_2}'.
\end{equation*}
Since $i_{\til{U}_1,\til{U}_2}$ is injective, we can replace $\phi_{\til{U}_1,\til{U}_2}$ with $\phi_{\til{U}_1,\til{U}_2}'$ in this relation. Then Lemma \ref{lem:ker_fin_gen} again implies that $\ker\phi_{\til{U}_1,\til{U}_2} =\ker(f_{\til{U}_1}^{V_1}\tens f_{\til{U}_2}^{V_2})$ for some $V_1\in I_{X_1}$ and $V_2\in I_{X_2}$. This means we have $G:\im\phi_{\til{U}_1,\til{U}_2}\rightarrow V_1\tens V_2$ such that
\begin{equation*}
 G\circ\phi_{\til{U}_1,\til{U}_2}'=f_{\til{U}_1}^{V_1}\tens f_{\til{U}_2}^{V_2},
\end{equation*}
and we conclude that
\begin{equation*}
 f_{W_1}^{V_1}\tens f_{W_2}^{V_2} =(f_{U_1}^{V_1}\tens f_{U_2}^{V_2})\circ F\circ f_{\im\phi_{W_1,W_2}}^{\im\phi_{\til{W}_1,\til{W}_2}}\circ\phi_{W_1,W_2}'.
\end{equation*}
But then using \eqref{eqn:w_in_kernel},
\begin{align*}
 \left(f_{i}^{j}\tens(f_{W_1}^{V_1}\tens f_{W_2}^{V_2})\right)(w) = 0,
\end{align*}
so that
\begin{equation*}
 b=\phi_{i,W_1,W_2}(w)=\left(\phi_{j,V_1,V_2}\circ(f_{i}^{j}\tens(f_{W_1}^{V_1}\tens f_{W_2}^{V_2}))\right)(w) =0.
\end{equation*}
Thus $\ker T_{\alpha,(X_1,X_2)}=0$.

Now to show that $T_{\alpha,(X_1,X_2)}$ is surjective, take $b\in\varinjlim\alpha\tens\alpha_{X_1\widehat{\tens}X_2}$. We may assume $b=\phi_{i, W_{1,2}}(w)$ for some $i\in I$, $W_{1,2}\in I_{X_1\widehat{\tens} X_2}$, and $w\in W(i)\tens W_{1,2}$. By Lemma \ref{lem:im_f_in_im_phi_i}, $W_{1,2}\subseteq\im\phi_{W_1,W_2}$ for some $W_1\in I_{X_1}$, $W_2\in I_{X_2}$. Now since $W(i)\tens\bullet$ is right exact,
\begin{equation*}
 \Id_{W(i)}\tens\phi_{W_1,W_2}': W(i)\tens(W_1\tens W_2)\rightarrow W(i)\tens\im\phi_{W_1,W_2}
\end{equation*}
is surjective, and there exists $\til{w}\in W(i)\tens(W_1\tens W_2)$ such that
\begin{equation*}
 (\Id_{W(i)}\tens\phi_{W_1,W_2}')(\til{w})=(\Id_{W(i)}\tens f_{W_{1,2}}^{\im\phi_{W_1,W_2}})(w).
\end{equation*}
Thus
\begin{align*}
 b=\phi_{i, W_{1,2}}(w) & =\left(\phi_{i,\,\im\phi_{W_1,W_2}}\circ(\Id_{W(i)}\tens f_{W_{1,2}}^{\im\phi_{W_1,W_2}})\right)(w)\nonumber\\
 & =\left(\phi_{i,\im\phi_{W_1,W_2}}\circ(\Id_{W(i)}\tens\phi_{W_1,W_2}')\right)(\til{w}) = T_{\alpha,(X_1,X_2)}\left(\phi_{i,W_1,W_2}(\til{w})\right),
\end{align*}
and we conclude $\im T_{\alpha,(X_1,X_2)}=\varinjlim\alpha\tens\alpha_{X_1\widehat{\tens} X_2}$.
\end{proof}

Now for weak modules $X_1$, $X_2$, and $X_3$, we can define $T_{X_1,(X_2,X_3)}=T_{\alpha_{X_1},(X_2,X_3)}$ and $T_{(X_1,X_2),X_3}=T_{(X_1,X_2),\alpha_{X_3}}$. To show that $T_{\bullet,(\bullet,\bullet)}$ is a natural transformation, consider the following commutative diagrams for morphisms $F_k: X_k\rightarrow\til{X}_k$, $k=1,2,3$, in $\cW$:
\begin{equation*}
\xymatrixcolsep{7pc}
\xymatrix{
 W_1\tens(W_2\tens W_3) \ar[d]_{\Id_{W_1}\tens\phi_{W_2,W_3}'} \ar[r]^(.45){\phi_{W_1,W_2,W_3}} & \varinjlim\alpha_{X_1}\tens(\alpha_{X_2}\tens\alpha_{X_3}) \ar[d]^{T_{X_1,(X_2,X_3)}} \\
 W_1\tens\im\phi_{W_2,W_3} \ar[d]_{F_1\vert_{W_1}\tens(F_2\widehat{\tens}F_3)\vert_{\im\phi_{W_2,W_3}}} \ar[r]^{\phi_{W_1,\im\phi_{W_2,W_3}}} & \varinjlim\alpha_{X_1}\tens\alpha_{X_2\widehat{\tens}X_3} \ar[d]^{F_1\widehat{\tens}(F_2\widehat{\tens}F_3)} \\
 F_1(W_1)\tens(F_2\widehat{\tens}F_3)(\im\phi_{W_2,W_3}) \ar[r]_(.58){\phi_{F_1(W_1),(F_2\widehat{\tens}F_3)(\im\phi_{W_2,W_3})}} & \varinjlim\alpha_{\til{X}_1}\tens\alpha_{\til{X}_2\widehat{\tens}\til{X}_3} \\
 }
\end{equation*}
and
\begin{equation*}
 \xymatrixcolsep{7pc}
 \xymatrix{
W_1\tens(W_2\tens W_3) \ar[r]^{\phi_{W_1,W_2,W_3}} \ar[d]_{F_1\vert_{W_1}\tens (F_2\vert_{W_2}\tens F_3\vert_{W_3})} & \varinjlim\alpha_{X_1}\tens(\alpha_{X_2}\tens\alpha_{X_3}) \ar[d]^{\widehat{F_1\tens(F_2\tens F_3)}} \\
 F_1(W_1)\tens(F_2(W_2)\tens F_3(W_3)) \ar[d]_{\Id_{F_1(W_1)}\tens\phi_{F_2(W_2), F_3(W_3)}'} \ar[r]^(.54){\phi_{F_1(W_1),F_2(W_2),F_3(W_3)}} & \varinjlim\alpha_{\til{X}_1}\tens(\alpha_{\til{X}_2}\tens\alpha_{\til{X}_3}) \ar[d]^{T_{\til{X}_1,(\til{X}_2,\til{X}_3)}} \\
F_1(W_1)\tens\im\phi_{F_2(W_2),F_3(W_3)} \ar[r]_(.56){\phi_{F_1(W_1),\im\phi_{F_2(W_2),F_3(W_3)}}} & \varinjlim\alpha_{\til{X}_1}\tens\alpha_{\til{X}_2\widehat{\tens}\til{X}_3} \\
 }
\end{equation*}
where $(W_1,W_2,W_3)\in I_{X_1}\times I_{X_2}\times I_{X_3}$. Since by definition
\begin{align*}
 (F_2\widehat{\tens}F_3)\vert_{\im\phi_{W_2,W_3}}\circ\phi_{W_2,W_3}'=(F_2\widehat{\tens}F_3)\circ\phi_{W_2,W_3} =\phi_{F_2(W_2),F_3(W_3)}\circ(F_2\vert_{W_2}\tens F_3\vert_{W_3}),
\end{align*}
the commutative diagrams show that
\begin{equation*}
 T_{\til{X}_1,(\til{X}_2,\til{X}_3)}\circ\widehat{F_1\tens(F_2\tens F_3)}\circ\phi_{W_1,W_2,W_3} =(F_1\widehat{\tens}(F_2\widehat{\tens}F_3))\circ T_{X_1,(X_2,X_3)}\circ\phi_{W_1,W_2,W_3}
\end{equation*}
for all $\cC$-submodules $W_k\subseteq X_k$, $k=1,2,3$, and thus
\begin{equation*}
 T_{\til{X}_1,(\til{X}_2,\til{X}_3)}\circ\widehat{F_1\tens(F_2\tens F_3)} =(F_1\widehat{\tens}(F_2\widehat{\tens}F_3))\circ T_{X_1,(X_2,X_3)}
\end{equation*}
as required. The proof that $T_{(\bullet,\bullet),\bullet}$ is a natural transformation is essentially the same.

We can now define a natural associativity isomorphism for the tensor product $\widehat{\tens}$, which we denote by $\mathfrak{A}$, using the natural isomorphisms $T_{\bullet,(\bullet,\bullet)}$ and $T_{(\bullet,\bullet),\bullet}$ combined with the natural isomorphism $\widehat{\cA}$ induced from the associativity isomorphisms in $\cC$:
\begin{equation*}
 \fA_{X_1,X_2,X_3}=T_{(X_1,X_2),X_3}\circ\widehat{\cA}_{X_1,X_2,X_3}\circ T_{X_1,(X_2,X_3)}^{-1}
\end{equation*}
for any weak modules $X_1$, $X_2$, and $X_3$. Note that $\fA_{X_1,X_2,X_3}$ is defined and is an isomorphism even if $X_1$, $X_2$, and $X_3$ are not objects of $\ind(\cC)$, although in this case both triple tensor products could well be $0$.

Finally, we consider the braiding and twist. The natural braiding isomorphism $\cR: \tens\rightarrow\tens\circ\sigma$, where $\sigma$ is the permutation functor, induces a natural braiding isomorphism $\widehat{\cR}: \widehat{\tens}\rightarrow\widehat{\tens}\circ\sigma$ such that
\begin{equation*}
 \xymatrixcolsep{4pc}
 \xymatrix{
 W_1\tens W_2 \ar[d]_{\phi_{W_1,W_2}} \ar[r]^{\cR_{W_1,W_2}} & W_2\tens W_1 \ar[d]^{\phi_{W_2,W_1}} \\
 \varinjlim\alpha_{X_1}\tens\alpha_{X_2} \ar[r]_{\widehat{\cR}_{X_1,X_2}} & \varinjlim\alpha_{X_2}\tens\alpha_{X_1} \\
 }
\end{equation*}
commutes for objects $(W_1,W_2)\in I_{X_1}\times I_{X_2}$. Note that $\widehat{\cR}_{X_1,X_2}$ is an isomorphism even if $X_1$ and $X_2$ are not objects of $\ind(\cC)$. Similarly, the twist $\theta$ induces a natural isomorphism $\widehat{\theta}:\widehat{\Id}_\cC\rightarrow\widehat{\Id}_\cC$ such that
\begin{equation*}
 \xymatrixcolsep{3pc}
 \xymatrix{
 W \ar[d]_{\phi_W} \ar[r]^{\theta_W} & W \ar[d]^{\phi_W} \\
 \varinjlim\alpha_X \ar[r]_{\widehat{\theta}_X} & \varinjlim\alpha_X\\
 }
\end{equation*}
for $W\in I_X$. Restricting to $\ind(\cC)$, we get a natural isomorphism $\Theta:\Id_{\ind(\cC)}\rightarrow\Id_{\ind(\cC)}$ defined by $\Theta_X=Q_X\circ\widehat{\theta}_X\circ Q_X^{-1}$ for generalized modules $X$ in $\ind(\cC)$. For any generalized module $X$ in $\ind(\cC)$ and $W\in I_X$, we claim that $\Theta_X\vert_W=\theta_W$, that is, the diagram
\begin{equation}\label{diag:Theta}
 \xymatrix{
 W \ar[r]^{\theta_W} \ar[d]_{i_W} & W \ar[d]^{i_W} \\
 X\ar[r]_{\Theta_X} & X\\
 }
\end{equation}
commutes. To prove this, we calculate
\begin{align*}
\Theta_X\circ i_W  = Q_X\circ\widehat{\theta}_X\circ Q_{X}^{-1}\circ i_W & = Q_X\circ\widehat{\theta}_X\circ Q_X^{-1}\circ Q_X\circ\phi_W\nonumber\\ &=Q_X\circ\widehat{\theta}_X\circ\phi_W =Q_X\circ\phi_W\circ\theta_W =i_W\circ\theta_W,
\end{align*}
using the definitions.

We can now prove that $\ind(\cC)$ is a braided tensor category with twist:
\begin{thm}
 The structure $(\ind(\cC), \widehat{\tens},V,\mathfrak{l},\mathfrak{r},\fA,\widehat{\cR},\Theta)$ is a $\CC$-linear braided tenor category with twist.
\end{thm}
\begin{proof}
We have already observed in Proposition \ref{prop:IndC_abelian} that $\ind(\cC)$ is a $\CC$-linear abelian category, but we still need to show that the tensor product of morphisms is bilinear. It is sufficient to show that for homomorphisms $F_1,G: X_1\rightarrow\til{X}_1$, $F_2: X_2\rightarrow\til{X}_2$ in $\cW$ and $h\in\CC$,
\begin{equation}\label{eqn:tens_bilinear}
(hF_1+G)\widehat{\tens}F_2=h(F_1\widehat{\tens}F_2)+G\widehat{\tens} F_2.
\end{equation}
The proof for $F_1\widehat{\tens}(hF_2+G)$ will be essentially the same, or alternatively, use the natural braiding isomorphism $\widehat{\cR}$ (plus bilinearity of composition).

To prove \eqref{eqn:tens_bilinear}, it is enough to show that
\begin{equation*}
 ((hF_1+G)\widehat{\tens}F_2)\circ\phi_{W_1,W_2}=h(F_1\widehat{\tens}F_2)\circ\phi_{W_1,W_2}+(G\widehat{\tens} F_2)\circ\phi_{W_1,W_2}
\end{equation*}
for all $W_1\in I_{W_1}$, $W_2\in I_{W_2}$. Take $\til{W}_1\in I_{\til{X}_1}$ such that $F_1(W_1),G(W_1),(hF_1+G)(W_1)\subseteq\til{W}_1$ (for example, take $\til{W}_1=F_1(W_1)+G(W_1)$). Then we find
\begin{align*}
 ((hF_1+G)\widehat{\tens}F_2)\circ & \phi_{W_1,W_2}  =\phi_{(hF_1+G)(W_1),F_2(W_2)}\circ((hF_1+G)\vert_{W_1}\tens F_2\vert_{W_2})\nonumber\\
 & =\phi_{\til{W}_1,F_2(W_2)}\circ(f_{(hF_1+G)(W_1)}^{\til{W}_1}\tens\Id_{F_2(W_2)})\circ((hF_1+G)\vert_{W_1}\tens F_2\vert_{W_2})\nonumber\\
 & =h\phi_{\til{W}_1,F_2(W_2)}\circ(f_{F_1(W_1)}^{\til{W}_1}\tens\Id_{F_2(W_2)})\circ(F_1\vert_{W_1}\tens F_2\vert_{W_2})\nonumber\\
 & \hspace{4em} + \phi_{\til{W}_1,F_2(W_2)}\circ(f_{G(W_1)}^{\til{W}_1}\tens\Id_{F_2(W_2)})\circ(G\vert_{W_1}\tens F_2\vert_{W_2})\nonumber\\
 & =h\phi_{F_1(W_1),F_2(W_2)}\circ(F_1\vert_{W_1}\tens F_2\vert_{W_2}) +\phi_{G(W_1),F_2(W_2)}\circ(G\vert_{W_1}\tens F_2\vert_{W_2})\nonumber\\
 & =h(F_1\widehat{\tens}F_2)\circ\phi_{W_1,W_2}+(G\widehat{\tens}F_2)\circ\phi_{W_1,W_2},
\end{align*}
where we have used the bilinearity of $\tens$ for the third equality.

Now we prove the triangle axiom. We need to show that for modules $X_1$, $X_2$ in $\ind(\cC)$, the diagram
 \begin{equation*}
  \xymatrixcolsep{3pc}
  \xymatrix{
  X_1\widehat{\tens}(V\widehat{\tens} X_2) \ar[r]^(.49){\fA_{X_1,V,X_2}} \ar[rd]_{\Id_{X_1}\widehat{\tens} \mathfrak{l}_{X_2}} & (X_1\widehat{\tens} V)\widehat{\tens}X_2 \ar[d]^{\mathfrak{r}_{X_1}\widehat{\tens}\Id_{X_2}} \\
  & X_1\widehat{\tens}X_2 \\
  }
 \end{equation*}
commutes, or equivalently,
\begin{equation}\label{eqn:triangle}
 (\Id_{X_1}\widehat{\tens} \mathfrak{l}_{X_2})\circ T_{X_1,(V,X_2)}=(\mathfrak{r}_{X_1}\widehat{\tens}\Id_{X_2})\circ T_{(X_1,V),X_2}\circ\widehat{\cA}_{X_1,V,X_2}.
\end{equation}
Using the definitions, the left side here is characterized by the commutative diagram
\begin{equation*}
 \xymatrixcolsep{5pc}
 \xymatrix{
 W_1\tens(U\tens W_2) \ar[d]_{\phi_{W_1,U,W_2}} \ar[r]^{\Id_{W_1}\tens\phi_{U,W_2}'}
 & W_1\tens\im\phi_{U,W_2} \ar[d]^{\phi_{W_1,\im\phi_{U,W_2}}}  \ar[r]^(.45){\Id_{W_1}\tens \mathfrak{l}_{X_2}\vert_{\im\phi_{U,W_2}}} & W_1\tens \mathfrak{l}_{X_2}(\im\phi_{U,W_2}) \ar[d]^{\phi_{W_1,\mathfrak{l}_{X_2}(\im\phi_{U,W_2})}} \\
 \varinjlim\alpha_{X_1}\tens(\alpha_V\tens\alpha_{X_2}) \ar[r]_(.54){T_{X_1,(V,X_2)}} & \varinjlim\alpha_{X_1}\tens\alpha_{V\widehat{\tens}\alpha_{X_2}} \ar[r]_(.54){\Id_{X_1}\widehat{\tens} \mathfrak{l}_{X_2}} & \varinjlim\alpha_{X_1}\tens\alpha_{X_2} \\
 }
\end{equation*}
for $W_1\in I_{X_1}$, $U\in I_V$, and $W_2\in I_{X_2}$. The definition of $\mathfrak{l}_{X_2}$ implies
\begin{equation*}
 \mathfrak{l}_{X_2}\vert_{\im\phi_{U,W_2}}\circ\phi_{U,W_2}' =\mathfrak{l}_{X_2}\circ\phi_{U,W_2}=i_{W_2}\circ l_{W_2}\circ(i_U\tens\Id_{W_2}).
\end{equation*}
In particular, $\mathfrak{l}_{X_2}(\im\phi_{U,W_2})\subseteq W_2$, so
\begin{align}\label{eqn:triangle_left}
 \phi_{W_1,\mathfrak{l}_{X_2}(\im\phi_{U,W_2})}\circ & (\Id_{W_1}\tens(\mathfrak{l}_{X_2}\vert_{\im\phi_{U,W_2}}\circ\phi_{U,W_2}'))\nonumber\\
 & =\phi_{W_1,W_2}\circ(\Id_{W_1}\tens f_{\mathfrak{l}_{X_2}(\im\phi_{U,W_2})}^{W_2})\circ(\Id_{W_1}\tens(\mathfrak{l}_{X_2}\vert_{\im\phi_{U,W_2}}\circ\phi_{U,W_2}'))\nonumber\\
 & =\phi_{W_1,W_2}\circ(\Id_{W_1}\tens l_{W_2})\circ(\Id_{W_1}\tens(i_U\tens\Id_{W_2})).
\end{align}
On the other hand, the right side of \eqref{eqn:triangle} is characterized by the diagram
\begin{equation*}
 \xymatrixcolsep{5pc}
 \xymatrix{
W_1\tens(U\tens W_2) \ar[d]_{\cA_{W_1,U,W_2}} \ar[r]^(.45){\phi_{W_1,U,W_2}} & \varinjlim\alpha_{X_1}\tens(\alpha_V\tens\alpha_{X_2}) \ar[d]^{\widehat{\cA}_{X_1,V,X_2}} \nonumber\\
(W_1\tens U)\tens W_2 \ar[d]_{\phi_{W_1,U}'\tens\Id_{W_2}} \ar[r]^(.45){\phi_{W_1,U,W_2}} & \varinjlim(\alpha_{X_1}\tens\alpha_V)\tens\alpha_{X_2} \ar[d]^{T_{(X_1,V),X_2}} \\
\im\phi_{W_1,U}\tens W_2 \ar[d]_{\mathfrak{r}_{X_1}\vert_{\im\phi_{W_1,U}}\tens\Id_{W_2}} \ar[r]^{\phi_{\im\phi_{W_1,U},W_2}} & \varinjlim\alpha_{X_1\widehat{\tens}V}\tens\alpha_{X_2} \ar[d]^{\mathfrak{r}_{X_1}\widehat{\tens}\Id_{X_2}} \\
\mathfrak{r}_{X_1}(\im\phi_{W_1,U})\tens W_2 \ar[r]_(.54){\phi_{\mathfrak{r}_{X_1}(\im\phi_{W_1,U}),W_2}} & \varinjlim\alpha_{X_1}\tens\alpha_{X_2} \\
 }
\end{equation*}
As with the left unit isomorphisms, we calculate
\begin{align}\label{eqn:right_triangle}
& \hspace{-3em}\phi_{\mathfrak{r}_{X_1}(\im\phi_{W_1,U}),W_2}\circ  ((\mathfrak{r}_{X_1}\vert_{\im\phi_{W_1,U}}\circ\phi_{W_1,U}')\tens\Id_{W_2})\circ\cA_{W_1,U,W_2}\nonumber\\
 & =\phi_{W_1,W_2}\circ(f_{\mathfrak{r}_{X_1}(\im\phi_{W_1,U})}^{W_1}\tens\Id_{W_2})\circ((\mathfrak{r}_{X_1}\vert_{\im\phi_{W_1,U}}\circ\phi_{W_1,U}')\tens\Id_{W_2})\circ\cA_{W_1,U,W_2}\nonumber\\
 & =\phi_{W_1,W_2}\circ(r_{W_1}\tens\Id_{W_2})\circ((\Id_{W_1}\tens i_U)\tens\Id_{W_2})\circ\cA_{W_1,U,W_2}\nonumber\\
 & =\phi_{W_1,W_2}\circ(r_{W_1}\tens\Id_{W_2})\circ\cA_{W_1,V,W_2}\circ(\Id_{W_1}\tens(i_U\tens\Id_{W_2})).
\end{align}
Now the triangle axiom for $\ind(\cC)$ follows from \eqref{eqn:triangle_left}, \eqref{eqn:right_triangle}, and the triangle axiom in $\cC$.

To prove the pentagon axiom, we first observe that the various quadruple tensoring functors $\cC^4\rightarrow\cC$ induce functors $\cW^4\rightarrow\ind(\cC)$, as in Section \ref{sec:ind_comp}, which are all pairwise naturally isomorphic via suitable associativity isomorphisms. For example, we use
\begin{equation*}
 (X_1,X_2,X_3,X_4)\mapsto\varinjlim\alpha_{X_1}\tens(\alpha_{X_2}\tens(\alpha_{X_3}\tens\alpha_{X_4}))
\end{equation*}
to denote the functor induced by
\begin{equation*}
 (W_1,W_2,W_3,W_4)\mapsto W_1\tens(W_2\tens(W_3\tens W_4)).
\end{equation*}
Now we claim that the following diagram commutes for modules $X_1$, $X_2$, $X_3$, $X_4$ in $\ind(\cC)$:
\begin{equation}\label{diag:pentagon_two}
\xymatrixcolsep{2.02pc}
 \xymatrix{
 \varinjlim\alpha_{X_1}\tens(\alpha_{X_2}\tens(\alpha_{X_3}\tens\alpha_{X_4})) \ar[d]_{\cong} \ar[r]^(.53){T_{1(2(34))}} & \varinjlim\alpha_{X_1}\tens(\alpha_{X_2}\tens\alpha_{X_3\widehat{\tens}X_4}) \ar[d]_{\widehat{\cA}_{X_1,X_2,X_3\widehat{\tens}X_4}} \ar[r]^(.47){T_{X_1,(X_2,X_3\widehat{\tens}X_4)}} & X_1\widehat{\tens}(X_2\widehat{\tens}(X_3\widehat{\tens}X_4)) \ar[d]^{\fA_{X_1,X_2,X_3\widehat{\tens}X_4}} \\
 \varinjlim(\alpha_{X_1}\tens\alpha_{X_2})\tens(\alpha_{X_3}\tens\alpha_{X_4}) \ar[r]^{T_{\alpha_{X_1}\tens\alpha_{X_2},(X_3,X_4)}} \ar[d]_{\cong} \ar[rd]_(.45){T_{(X_1,X_2),\alpha_{X_3}\tens\alpha_{X_4}}} & \varinjlim(\alpha_{X_1}\tens\alpha_{X_2})\tens\alpha_{X_3\widehat{\tens}X_4} \ar[r]^(.47){T_{(X_1,X_2),X_3\widehat{\tens}X_4}} & (X_1\widehat{\tens}X_2)\widehat{\tens}(X_3\widehat{\tens}X_4) \ar[d]^{\fA_{X_1\widehat{\tens}X_2,X_3,X_4}} \\
 \varinjlim((\alpha_{X_1}\tens\alpha_{X_2})\tens\alpha_{X_3})\tens\alpha_{X_4} \ar[rd]_{T_{((12)3)4}}  & \varinjlim\alpha_{X_1\widehat{\tens}X_2}\tens(\alpha_{X_3}\tens\alpha_{X_4}) \ar[d]^(.45){\widehat{\cA}_{X_1\widehat{\tens}X_2,X_3,X_4}} \ar[ru]_(.55){\qquad T_{X_1\widehat{\tens}X_2,(X_3,X_4)}} & ((X_1\widehat{\tens}X_2)\widehat{\tens}X_3)\widehat{\tens}X_4 \\
 & \varinjlim(\alpha_{X_1\widehat{\tens}X_2}\tens\alpha_{X_3})\tens\alpha_{X_4} \ar[ru]_(.55){\qquad T_{(X_1\widehat{\tens}X_2,X_3),X_4}} & \\
 }
\end{equation}
Here we can define the homomorphisms $T_{1(2(34))}$ and $T_{((12)3)4}$ to be such that the corresponding squares commute; they are isomorphisms because $T_{\alpha_{X_1}\tens\alpha_{X_2},(X_3,X_4)}$ and $T_{(X_1,X_2),\alpha_{X_3}\tens\alpha_{X_4}}$ are isomorphisms by Proposition \ref{prop:Fubini}. Thus to prove that \eqref{diag:pentagon_two} commutes, we just need to check the square in the middle. Indeed, for $W_k\in I_{X_k}$, $k=1,2,3,4$, we have
\begin{align*}
T_{X_1\widehat{\tens}X_2,(X_3,X_4)}\circ & T_{(X_1,X_2),\alpha_{X_3}\tens\alpha_{X_4}}\circ\phi_{W_1,W_2,W_3,W_4}\nonumber\\
& = T_{X_1\widehat{\tens}X_2,(X_3,X_4)}\circ\phi_{\im\phi_{W_1,W_2},W_3,W_4}\circ(\phi_{W_1,W_2}'\tens\Id_{W_3\tens W_4})\nonumber\\
& =\phi_{\im\phi_{W_1,W_2},\im\phi_{W_3,W_4}}\circ(\Id_{\im_{\phi_{W_1,W_2}}}\tens\phi_{W_3,W_4}')\circ(\phi_{W_1,W_2}'\tens\Id_{W_3\tens W_4})\nonumber\\
& =\phi_{\im\phi_{W_1,W_2},\im\phi_{W_3,W_4}}\circ(\phi_{W_1,W_2}'\tens\Id_{\im\phi_{W_3,W_4}})\circ(\Id_{W_1\tens W_2}\tens\phi_{W_3,W_4}')\nonumber\\
& =T_{(X_1,X_2),X_3\widehat{\tens}X_4}\circ\phi_{W_1,W_2,\im\phi_{W_3,W_4}}\circ(\Id_{W_1\tens W_2}\tens\phi_{W_3,W_4}')\nonumber\\
& =T_{(X_1,X_2),X_3\widehat{\tens}X_4}\circ T_{\alpha_{X_1}\tens\alpha_{X_2},(X_3,X_4)}\circ\phi_{W_1,W_2,W_3,W_4}.
\end{align*}
We can also alternatively characterize $T_{1(2(34))}$ as follows:
\begin{align*}
 T_{1(2(34))}\circ & \phi_{W_1,W_2,W_3,W_4}  =\widehat{\cA}^{-1}_{X_1,X_2,X_3\widehat{\tens}X_4}\circ T_{\alpha_{X_1}\tens\alpha_{X_2},(X_3,X_4)}\circ\phi_{W_1,W_2,W_3,W_4}\circ\cA_{W_1,W_2,W_3\tens W_4}\nonumber\\
 & = \widehat{\cA}^{-1}_{X_1,X_2,X_3\widehat{\tens}X_4}\circ\phi_{W_1,W_2,\im\phi_{W_3,W_4}}\circ(\Id_{W_1,\tens W_2}\tens\phi_{W_3,W_4}')\circ\cA_{W_1,W_2,W_3\tens W_4}\nonumber\\
 & =\phi_{W_1,W_2,\im\phi_{W_3,W_4}}\circ\cA^{-1}_{W_1,W_2,\im\phi_{W_3,W_4}}\circ(\Id_{W_1,\tens W_2}\tens\phi_{W_3,W_4}')\circ\cA_{W_1,W_2,W_3\tens W_4}\nonumber\\
 & =\phi_{W_1,W_2,\im\phi_{W_3,W_4}}\circ(\Id_{W_1}\tens(\Id_{W_2}\tens\phi_{W_3,W_4}'))
\end{align*}
for all $W_k\in I_{X_k}$, $k=1,2,3,4$. Similarly, $T_{((12)3)4}$ is characterized by
\begin{equation*}
 T_{((12)3)4}\circ  \phi_{W_1,W_2,W_3,W_4} =\phi_{\im\phi_{W_1,W_2},W_3,W_4}\circ((\phi_{W_1,W_2}'\tens\Id_{W_3})\tens\Id_{W_4})
\end{equation*}
for all $W_k\in I_{X_k}$.

We also claim the following diagram commutes:
\begin{equation}\label{diag:pentagon_three}
\xymatrixcolsep{1.64pc}
 \xymatrix{
 \varinjlim\alpha_{X_1}\tens(\alpha_{X_2}\tens(\alpha_{X_3}\tens\alpha_{X_4})) \ar[r]_(.55){\stackrel{}{T_{1(2(34))}}} \ar[d]_{\cong} & \varinjlim\alpha_{X_1}\tens(\alpha_{X_2}\tens\alpha_{X_3\widehat{\tens}X_4}) \ar[r]_(.68){\stackrel{}{T_{X_1,(X_2,X_3\widehat{\tens}X_4)}}} & X_1\widehat{\tens}(X_2\widehat{\tens}(X_3\widehat{\tens}X_4)) \ar[d]^{\Id_{X_1}\widehat{\tens}\fA_{X_2,X_3,X_4}} \\
 \varinjlim\alpha_{X_1}\tens((\alpha_{X_2}\tens\alpha_{X_3})\tens\alpha_{X_4}) \ar[r]_(.55){\stackrel{}{T_{1((23)4)}}} \ar[d]_{\cong} & \varinjlim\alpha_{X_1}\tens(\alpha_{X_2\widehat{\tens}X_3}\tens\alpha_{X_4}) \ar[d]^{\widehat{\cA}_{X_1,X_2\widehat{\tens}X_3,X_4}} \ar[r]_(.68){\stackrel{}{T_{X_1,(X_2\widehat{\tens}X_3,X_4)}}} & X_1\widehat{\tens}((X_2\widehat{\tens}X_3)\widehat{\tens} X_4) \ar[d]^{\fA_{X_1,X_2\widehat{\tens}X_3,X_4}} \\
 \varinjlim(\alpha_{X_1}\tens(\alpha_{X_2}\tens\alpha_{X_3}))\tens\alpha_{X_4} \ar[d]_{\cong} \ar[r]_(.55){\stackrel{}{T_{(1(23))4}}} & \varinjlim(\alpha_{X_1}\tens\alpha_{X_2\widehat{\tens}X_3})\tens\alpha_{X_4} \ar[r]_(.68){\stackrel{}{T_{(X_1,X_2\widehat{\tens}X_3),X_4}}} & (X_1\widehat{\tens}(X_2\widehat{\tens}X_3))\widehat{\tens}X_4 \ar[d]^{\fA_{X_1,X_2,X_3}\widehat{\tens}\Id_{X_4}} \\
 \varinjlim((\alpha_{X_1}\tens\alpha_{X_2})\tens\alpha_{X_3})\tens\alpha_{X_4} \ar[r]_(.55){\stackrel{}{T_{((12)3)4}}} & \varinjlim(\alpha_{X_1\widehat{\tens}X_2}\tens\alpha_{X_3})\tens\alpha_{X_4} \ar[r]_(.68){\stackrel{}{T_{(X_1\widehat{\tens}X_2,X_3),X_4}}} & ((X_1\widehat{\tens}X_2)\widehat{\tens}X_3)\widehat{\tens}X_4 \\
 }
\end{equation}
Again, we define the isomorphisms $T_{1((23)4)}$ and $T_{(1(23))4}$ such that the corresponding hexagons in the diagram commute. It remains to prove that the square in the center left of the diagram commutes.

We need an alternate description of $T_{1((23)4)}$, so for $W_k\in I_{X_k}$, $k=1,2,3,4$, we calculate:
\begin{align*}
& T_{X_1,(X_2\widehat{\tens}X_3,X_4)}\circ  T_{1((23)4)}\circ\phi_{W_1,W_2,W_3,W_4} \nonumber\\
 & =(\Id_{X_1}\widehat{\tens}\fA_{X_2,X_3,X_4})\circ T_{X_1,(X_2,X_3\widehat{\tens} X_4)}\circ T_{1(2(34))}\circ\phi_{W_1,W_2,W_3,W_4}\circ(\Id_{W_1}\tens\cA_{W_2,W_3,W_4}^{-1})\nonumber\\
 & = (\Id_{X_1}\widehat{\tens}\fA_{X_2,X_3,X_4})\circ T_{X_1,(X_2,X_3\widehat{\tens} X_4)}\circ\phi_{W_1,W_2,\im\phi_{W_3,W_4}}\circ\nonumber\\
 &\hspace{8em}\circ(\Id_{W_1}\tens(\Id_{W_2}\tens\phi_{W_3,W_4}'))\circ(\Id_{W_1}\tens\cA_{W_2,W_3,W_4}^{-1})\nonumber\\
 & = (\Id_{X_1}\widehat{\tens}\fA_{X_2,X_3,X_4})\circ\phi_{W_1,\im\phi_{W_2,\im\phi_{W_3,W_4}}}\circ(\Id_{W_1}\tens\phi_{W_2,\im\phi_{W_3,W_4}}')\circ\nonumber\\
 &\hspace{8em}\circ(\Id_{W_1}\tens(\Id_{W_2}\tens\phi_{W_3,W_4}'))\circ(\Id_{W_1}\tens\cA_{W_2,W_3,W_4}^{-1})\nonumber\\
 & =\phi_{W_1,\fA_{X_2,X_3,X_4}(\im\phi_{W_2,\im\phi_{W_3,W_4}})}\circ(\Id_{W_1}\tens\fA_{X_2,X_3,X_4}\vert_{\im\phi_{W_2,\im\phi_{W_3,W_4}}})\circ\nonumber\\
 &\hspace{8em}\circ(\Id_{W_1}\tens\phi_{W_2,\im\phi_{W_3,W_4}}')\circ(\Id_{W_1}\tens(\Id_{W_2}\tens\phi_{W_3,W_4}'))\circ(\Id_{W_1}\tens\cA_{W_2,W_3,W_4}^{-1}).
\end{align*}
Now,
\begin{align*}
 \fA_{X_2,X_3,X_4} & \vert_{\im\phi_{W_2,\im\phi_{W_3,W_4}}}\circ\phi_{W_2,\im\phi_{W_3,W_4}}'\circ(\Id_{W_2}\tens\phi_{W_3,W_4}')\nonumber\\
 & =T_{(X_2,X_3),X_4}\circ\widehat{\cA}_{X_2,X_3,X_4}\circ T_{X_2,(X_3,X_4)}^{-1}\circ\phi_{W_2,\im\phi_{W_3,W_4}}\circ(\Id_{W_2}\tens\phi_{W_3,W_4}')\nonumber\\
 & =T_{(X_2,X_3),X_4}\circ\widehat{\cA}_{X_2,X_3,X_4}\circ\phi_{W_2,W_3,W_4}\nonumber\\
 & =T_{(X_2,X_3),X_4}\circ\phi_{W_2,W_3,W_4}\circ\cA_{W_2,W_3,W_4}\nonumber\\
 & =\phi_{\im\phi_{W_2,W_3},W_4}'\circ(\phi_{W_2,W_3}'\tens\Id_{W_4})\circ\cA_{W_2,W_3,W_4}.
\end{align*}
So this means
\begin{align*}
  T_{X_1,(X_2\widehat{\tens}X_3,X_4)}\circ  & T_{1((23)4)}\circ\phi_{W_1,W_2,W_3,W_4} \nonumber\\
  & =\phi_{W_1,\im\phi_{\im\phi_{W_2,W_3}, W_4}}\circ(\Id_{W_1}\tens\phi'_{\im\phi_{W_2,W_3},W_4})\circ(\Id_{W_1}\tens(\phi_{W_2,W_3}'\tens\Id_{W_4}))\nonumber\\
  & =T_{X_1,(X_2\widehat{\tens}X_3,X_4)}\circ\phi_{W_1,\im\phi_{W_2,W_3},W_4}\circ(\Id_{W_1}\tens(\phi_{W_2,W_3}'\tens\Id_{W_4})).
\end{align*}
That is, $T_{1((23)4)}$ is the unique homomorphism such that
\begin{equation*}
 T_{1((23)4)}\circ\phi_{W_1,W_2,W_3,W_4} = \phi_{W_1,\im\phi_{W_2,W_3},W_4}\circ(\Id_{W_1}\tens(\phi_{W_2,W_3}'\tens\Id_{W_4}))
\end{equation*}
for all $W_k\in I_{X_k}$, $k=1,2,3,4$. Similarly, we can prove that
\begin{equation*}
 T_{(1(23))4}\circ\phi_{W_1,W_2,W_3,W_4} =\phi_{W_1,\im\phi_{W_2,W_3},W_4}\circ((\Id_{W_1}\tens\phi_{W_2,W_3}')\tens\Id_{W_4})
\end{equation*}
for $W_k\in I_{X_k}$.

Now we can prove that the center left square in \eqref{diag:pentagon_three} commutes:
\begin{align*}
\widehat{\cA}_{X_1,X_2\widehat{\tens}X_3,X_4}\circ & T_{1((23)4)}\circ \phi_{W_1,W_2,W_3,W_4}\nonumber\\ & =\widehat{\cA}_{X_1,X_2\widehat{\tens}X_3,X_4}\circ\phi_{W_1,\im\phi_{W_2,W_3},W_4}\circ(\Id_{W_1}\tens(\phi_{W_2,W_3}'\tens\Id_{W_4}))\nonumber\\
& =\phi_{W_1,\im\phi_{W_2,W_3},W_4}\circ\cA_{W_1,\im\phi_{W_2,W_3},W_4}\circ(\Id_{W_1}\tens(\phi_{W_2,W_3}'\tens\Id_{W_4}))\nonumber\\
& =\phi_{W_1,\im\phi_{W_2,W_3},W_4}\circ((\Id_{W_1}\tens\phi_{W_2,W_3}')\tens\Id_{W_4})\circ\cA_{W_1,W_2\tens W_3,W_4}\nonumber\\
& =T_{(1(23))4}\circ\phi_{W_1,W_2,W_3,W_4}\circ\cA_{W_1,W_2\tens W_3,W_4}
\end{align*}
for all $W_1$, $W_2$, $W_3$, $W_4$, as required. Now the pentagon axiom for $\fA$ follows from the commutative diagrams \eqref{diag:pentagon_two} and \eqref{diag:pentagon_three}, the fact that $T_{1(2(34))}$ and $T_{X_1,(X_2,X_3\widehat{\tens}X_4)}$ are isomorphisms, and the pentagon axiom for $\cA$.

Next, the hexagon axioms for $\ind(\cC)$ follow from the hexagon axioms in $\cC$, the fact that $T_{X_1,(X_2,X_3)}$ is an isomorphism, and the following diagrams for modules $X_1$, $X_2$, $X_3$ in $\ind(\cC)$, provided they commute:
\begin{equation*}
\xymatrixcolsep{4pc}
\xymatrix{
W_1\tens(W_2\tens W_3) \ar[r]^(.45){\phi_{W_1,W_2,W_3}} \ar[d]_{\cA_{W_1,W_2,W_3}} & \varinjlim\alpha_{X_1}\tens(\alpha_{X_2}\tens\alpha_{X_3}) \ar[r]^(.53){T_{X_1,(X_2,X_3)}} \ar[d]^{\widehat{\cA}_{X_1,X_2,X_3}} & \varinjlim\alpha_{X_1}\tens\alpha_{X_2\widehat{\tens}X_3} \ar[d]^{\fA_{X_1,X_2,X_3}} \\ (W_1\tens W_2)\tens W_3 \ar[r]^(.45){\phi_{W_1,W_2,W_3}} \ar[d]_{\cR^{\pm 1}} & \varinjlim(\alpha_{X_1}\tens\alpha_{X_2})\tens\alpha_{X_3} \ar[r]^(.53){T_{(X_1,X_2),X_3}} & \varinjlim\alpha_{X_1\widehat{\tens}X_2}\tens\alpha_{X_3} \ar[d]^{\widehat{\cR}^{\pm 1}} \\ W_3\tens(W_1\tens W_2) \ar[r]^(.45){\phi_{W_3,W_1,W_2}} \ar[d]_{\cA_{W_3,W_1,W_2}} &  \ar[d]^{\widehat{\cA}_{X_3,X_1,X_2}}  \varinjlim\alpha_{X_3}\tens(\alpha_{X_1}\tens\alpha_{X_2}) \ar[r]^(.53){T_{X_3,(X_1,X_2)}} & \varinjlim\alpha_{X_3}\tens\alpha_{X_1\widehat{\tens}X_2} \ar[d]^{\fA_{X_3,X_1,X_2}} \\ (W_3\tens W_1)\tens W_2 \ar[r]_(.45){\phi_{W_3,W_1,W_2}} & \varinjlim(\alpha_{X_3}\tens\alpha_{X_1})\tens\alpha_{X_2} \ar[r]_(.53){T_{(X_3,X_1),X_2}} & \varinjlim\alpha_{X_3\widehat{\tens}X_1}\tens\alpha_{X_2} \\
 }
\end{equation*}
and
\begin{equation*}
\xymatrixcolsep{4pc}
\xymatrix{
W_1\tens(W_2\tens W_3) \ar[r]^(.45){\phi_{W_1,W_2,W_3}} \ar[d]_{\Id_{W_1}\tens\cR^{\pm 1}} & \varinjlim\alpha_{X_1}\tens(\alpha_{X_2}\tens\alpha_{X_3}) \ar[r]^(.53){T_{X_1,(X_2,X_3)}}  & \varinjlim\alpha_{X_1}\tens\alpha_{X_2\widehat{\tens}X_3} \ar[d]^{\Id_{X_1}\widehat{\tens}\widehat{\cR}^{\pm 1}} \\ W_1\tens (W_3\tens W_2) \ar[r]^(.45){\phi_{W_1,W_3,W_2}} \ar[d]_{\cA_{W_1,W_3,W_2}} & \varinjlim\alpha_{X_1}\tens(\alpha_{X_3}\tens\alpha_{X_2}) \ar[d]^{\widehat{\cA}_{X_1,X_3,X_2}} \ar[r]^(.53){T_{X_1,(X_3,X_2)}} & \varinjlim\alpha_{X_1}\tens\alpha_{X_3\widehat{\tens}X_2} \ar[d]^{\fA_{X_1,X_3,X_2}} \\ (W_1\tens W_3)\tens W_2 \ar[r]^(.45){\phi_{W_1,W_3,W_2}} \ar[d]_{\cR^{\pm 1}\tens\Id_{W_2}} &    \varinjlim(\alpha_{X_1}\tens\alpha_{X_3})\tens\alpha_{X_2} \ar[r]^(.53){T_{(X_1,X_3),X_2}} & \varinjlim\alpha_{X_1\widehat{\tens}X_3}\tens\alpha_{X_2} \ar[d]^{\widehat{\cR}^{\pm 1}\widehat{\tens}\Id_{X_2}} \\ (W_3\tens W_1)\tens W_2 \ar[r]_(.45){\phi_{W_3,W_1,W_2}} & \varinjlim(\alpha_{X_3}\tens\alpha_{X_1})\tens\alpha_{X_2} \ar[r]_(.53){T_{(X_3,X_1),X_2}} & \varinjlim\alpha_{X_3\widehat{\tens}X_1}\tens\alpha_{X_2} \\
 }
\end{equation*}
for $W_k\in I_{X_k}$. We just need to consider the cycles in the diagrams that involve braiding isomorphisms. Indeed, we have
\begin{align*}
 \widehat{\cR}^{\pm 1}\circ T_{(X_1,X_2),X_3}\circ\phi_{W_1,W_2,W_3} & = \widehat{\cR}^{\pm 1}\circ\phi_{\im\phi_{W_1,W_2}, W_3}\circ(\phi_{W_1,W_2}'\tens\Id_{W_3})\nonumber\\
 & =\phi_{W_3,\im\phi_{W_1,W_2}}\circ\cR^{\pm 1}\circ(\phi_{W_1,W_2}'\tens\Id_{W_3}\nonumber\\
 & =\phi_{W_3,\im\phi_{W_1,W_2}}\circ(\Id_{W_3}\tens\phi_{W_1,W_2}')\circ\cR^{\pm 1}\nonumber\\
 & =T_{X_3,(X_1,X_2)}\circ\phi_{W_3,W_1,W_2}\circ\cR^{\pm 1}
\end{align*}
and
\begin{align*}
( \Id_{X_1}\widehat{\tens}\widehat{\cR}^{\pm 1})\circ T_{X_1,(X_2,X_3)}\circ & \phi_{W_1,W_2,W_3}  =( \Id_{X_1}\widehat{\tens}\widehat{\cR}^{\pm 1})\circ\phi_{W_1,\im\phi_{W_2,W_3}}\circ(\Id_{W_1}\tens\phi_{W_2,W_3}')\nonumber\\
& =\phi_{W_1,\widehat{\cR}^{\pm 1}(\im\phi_{W_2,W_3})}\circ(\Id_{W_1}\tens\widehat{\cR}^{\pm 1}\vert_{\im\phi_{W_2,W_3}})\circ(\Id_{W_1}\tens\phi_{W_2,W_3}')\nonumber\\
& = \phi_{W_1,\im\phi_{W_3,W_2}}\circ(\Id_{W_1}\tens\phi_{W_3,W_2}')\circ(\Id_{W_1}\tens\cR^{\pm 1})\nonumber\\
& =T_{X_1,(X_3,X_2)}\circ\phi_{W_1,W_3,W_2}\circ(\Id_{W_1}\tens\cR^{\pm 1}).
\end{align*}
Similarly,
\begin{equation*}
 (\widehat{\cR}^{\pm 1}\widehat{\tens}\Id_{X_2})\circ T_{(X_1,X_3),X_2}\circ\phi_{W_1,W_3,W_2} = T_{(X_3,X_1),X_2}\circ\phi_{W_3,W_1,W_2}\circ(\cR^{\pm 1}\tens\Id_{W_2}),
\end{equation*}
so the hexagon axioms hold for $\widehat{\cR}$.

Finally, we need to show that $\Theta$ is a twist on $\ind(\cC)$, that is, $\Theta_V=\Id_V$ and the balancing equation
 \begin{equation*}
  \Theta_{X_1\widehat{\tens} X_2}=\widehat{\cR}_{X_2,X_1}\circ\widehat{\cR}_{X_1,X_2}\circ(\Theta_{X_1}\widehat{\tens}\Theta_{X_2})
 \end{equation*}
holds for modules $X_1$, $X_2$ in $\ind(\cC)$. Because $\theta_V=\Id_V$, the definition shows that $\widehat{\theta}_V=\Id_{\varinjlim\alpha_V}$. So $\Theta_V=Q_V\circ Q_V^{-1} = \Id_V$. For the balancing equation, we need to show that
 \begin{equation}\label{eqn:Theta_balance}
  \Theta_{X_1\widehat{\tens} X_2}\circ\phi_{W_1,W_2} =\widehat{\cR}_{X_2,X_1}\circ\widehat{\cR}_{X_1,X_2}\circ(\Theta_{X_1}\widehat{\tens}\Theta_{X_2})\circ\phi_{W_1,W_2}
 \end{equation}
for $(W_1,W_2)\in I_{X_1}\times I_{X_2}$. On the one hand, the definitions,  \eqref{diag:Theta}, and the balancing equation for $\theta$ show
\begin{align*}
 \widehat{\cR}_{X_2,X_1}\circ\widehat{\cR}_{X_1,X_2} & \circ  (\Theta_{X_1}\widehat{\tens}\Theta_{X_2})\circ\phi_{W_1,W_2}\nonumber\\
 &= \widehat{\cR}_{X_2,X_1}\circ\widehat{\cR}_{X_1,X_2}\circ\phi_{\Theta_{X_1}(W_1),\Theta_{X_2}(W_2)}\circ(\Theta_{X_1}\vert_{W_1}\tens\Theta_{X_2}\vert_{W_2})\nonumber\\
 & =\widehat{\cR}_{X_2,X_1}\circ\widehat{\cR}_{X_1,X_2}\circ\phi_{W_1,W_2}\circ(\theta_{W_1}\tens\theta_{W_2})\nonumber\\
 & =\phi_{W_1,W_2}\circ\cR_{W_2,W_1}\circ\cR_{W_1,W_2}\circ(\theta_{W_1}\tens\theta_{W_2})\nonumber\\
 & =\phi_{W_1,W_2}\circ\theta_{W_1\tens W_2}.
\end{align*}
On the other hand, since $\im\phi_{W_1,W_2}$ is a $\cC$-submodule of $X_1\widehat{\tens}X_2$, we can use \eqref{diag:Theta} and the naturality of $\theta$ to conclude
\begin{align*}
 \Theta_{X_1\widehat{\tens}X_2}\circ\phi_{W_1,W_2} &=\Theta_{X_1\widehat{\tens}X_2}\circ i_{\im\phi_{W_1,W_2}}\circ\phi_{W_1,W_2}'\nonumber\\ & =i_{\im\phi_{W_1,W_2}}\circ\theta_{\im\phi_{W_1,W_2}}\circ\phi_{W_1,W_2}'\nonumber\\
 & =i_{\im\phi_{W_1,W_2}}\circ\phi_{W_1,W_2}'\circ\theta_{W_1\tens W_2}\nonumber\\
 & =\phi_{W_1,W_2}\circ\theta_{W_1\tens W_2},
\end{align*}
recalling that $\phi_{W_1,W_2}'$ denotes the surjection $W_1\tens W_2\rightarrow \im\phi_{W_1,W_2}$ induced by $\phi_{W_1,W_2}$. This proves \eqref{eqn:Theta_balance}.
\end{proof}

Since $\cC$ is a subcategory of $\ind(\cC)$, we should check that the braided tensor category structure on $\ind(\cC)$ is actually an extension of the braided tensor category structure on $\cC$:
\begin{thm}\label{thm:braid_tens_inc}
 The embedding of $\cC$ into $\ind(\cC)$ is a braided tensor functor. In particular, the restriction of the braided tensor category structure with twist $(\ind(\cC),\widehat{\tens},V,\mathfrak{l},\mathfrak{r},\fA,\widehat{\cR},\Theta)$ to $\cC$ is equivalent to $(\cC,\tens,V,l,r,\cA,\cR,\theta)$.
\end{thm}
\begin{proof}
First observe from \eqref{diag:Theta} that $\Theta_W=\theta_W$ for modules $W$ in $\cC$.

Now to show that the embedding of $\cC$ into $\ind(\cC)$ is a braided tensor functor, we need a natural isomorphism $Q: \widehat{\tens}\vert_{\cC\times\cC}\rightarrow\tens$ such that the diagrams
 \begin{equation}\label{diag:Q_unit_compat}
  \xymatrix{
  V\widehat{\tens} W \ar[r]^{Q_{V,W}} \ar[rd]_{\mathfrak{l}_W} & V\tens W \ar[d]^{l_W} \\
  & W \\
  } \qquad\text{and}\qquad
  \xymatrix{
  W\widehat{\tens}V \ar[r]^{Q_{W,V}} \ar[rd]_{\mathfrak{r}_W} & W\tens V \ar[d]^{r_W} \\
  & W \\
  }
 \end{equation}
commute for any module $W$ in $\cC$, the diagram
\begin{equation}\label{diag:Q_assoc_compat}
 \xymatrixcolsep{4pc}
 \xymatrix{
 W_1\widehat{\tens}(W_2\widehat{\tens}W_3) \ar[r]^{\fA_{W_1,W_2,W_3}} \ar[d]_{\Id_{W_1}\widehat{\tens}Q_{W_2,W_3}} & (W_1\widehat{\tens}W_2)\widehat{\tens} W_3 \ar[d]^{Q_{W_1,W_2}\widehat{\tens}\Id_{W_3}} \\
 W_1\widehat{\tens}(W_2\tens W_3) \ar[d]_{Q_{W_1,W_2\tens W_3}} & (W_1\tens W_2)\widehat{\tens} W_3 \ar[d]^{Q_{W_1\tens W_2,W_3}} \\
 W_1\tens(W_2\tens W_3) \ar[r]_{\cA_{W_1,W_2,W_3}} & (W_1\tens W_2)\tens W_3) \\
 }
\end{equation}
commutes for any modules $W_1$, $W_2$, and $W_3$ in $\cC$,
and the diagram
\begin{equation}\label{diag:Q_braid_compat}
\xymatrixcolsep{3pc}
\xymatrix{
W_1\widehat{\tens} W_2 \ar[r]^{\widehat{\cR}_{W_1,W_2}} \ar[d]_{Q_{W_1,W_2}} & W_2\widehat{\tens}W_1 \ar[d]^{Q_{W_2,W_1}} \\
W_1\tens W_2 \ar[r]_{\cR_{W_1,W_2}} & W_2\tens W_1 \\
}
\end{equation}
commutes for modules $W_1$ and $W_2$ in $\cC$.

For modules $W_1$ and $W_2$ in $\cC$, we define $Q_{W_1,W_2}$ to be the unique morphism such that
\begin{equation*}
 \xymatrixcolsep{3pc}
 \xymatrix{
 U_1\tens U_2 \ar[d]_{\phi_{U_1,U_2}} \ar[rd]^{i_{U_1}\tens i_{U_2}} \\
 \varinjlim\alpha_{W_1}\tens\alpha_{W_2} \ar[r]_(.58){Q_{W_1,W_2}} & W_1\tens W_2 \\
 }
\end{equation*}
commutes for all $U_1\in I_{W_1}$ and $U_2\in I_{W_2}$. It is clear from the universal property of $\varinjlim\alpha_{W_1}\tens\alpha_{W_2}$ that $Q_{W_1,W_2}$ exists, and it is an isomorphism with inverse $\phi_{W_1,W_2}$. Indeed, the definition of $Q_{W_1,W_2}$ shows that
\begin{equation*}
 Q_{W_1,W_2}\circ\phi_{W_1,W_2} =\Id_{W_1}\tens\Id_{W_2}=\Id_{W_1\tens W_2}.
\end{equation*}
On the other hand, the diagram
\begin{equation*}
 \xymatrixcolsep{4pc}
 \xymatrix{
 & U_1\tens U_2 \ar[ld]_{\phi_{U_1,U_2}} \ar[d]^{i_{U_1}\tens i_{U_2}} \ar[rd]^{\phi_{U_1,U_2}} & \\
 \varinjlim\alpha_{W_1}\tens\alpha_{W_2} \ar[r]_(.55){Q_{W_1,W_2}} & W_1\tens W_2 \ar[r]_(.45){\phi_{W_1,W_2}} & \varinjlim\alpha_{W_1}\tens\alpha_{W_2} \\
 }
\end{equation*}
for $U_1\in I_{W_1}$ and $U_2\in I_{W_2}$ shows that
\begin{equation*}
 \phi_{W_1,W_2}\circ Q_{W_1,W_2} =\Id_{W_1\widehat{\tens} W_2}.
\end{equation*}
(The triangle on the right side here commutes because $i_{U_1}=f_{U_1}^{W_1}$ and $i_{U_2}=f_{U_2}^{W_2}$.) We also need to show that $Q$ is a natural isomorphism, that is,
\begin{equation}\label{eqn:Q_nat_trans}
 Q_{\til{W}_1,\til{W}_2}\circ(F_1\widehat{\tens} F_2)=(F_1\tens F_2)\circ Q_{W_1,W_2}
\end{equation}
for homomorphisms $F_1: W_1\rightarrow \til{W}_1$ and $F_2:W_2\rightarrow\til{W}_2$ in $\cC$. From the commutative diagram
\begin{equation*}
 \xymatrixcolsep{4pc}
 \xymatrix{
  U_1\tens U_2 \ar[d]_{\phi_{U_1,U_2}} \ar[r]^(.42){F_1\vert_{U_1}\tens F_2\vert_{U_2}} & F_1(U_1)\tens F_2(U_2) \ar[d]_{\phi_{F_1(U_1),F_2(U_2)}} \ar[rd]^{\quad i_{F_1(U_1)}\tens i_{F_2(U_2)}} \\
 \varinjlim\alpha_{W_1}\tens\alpha_{W_2} \ar[r]_{F_1\widehat{\tens} F_2} & \varinjlim\alpha_{\til{W}_1}\tens\alpha_{\til{W}_2} \ar[r]_(.55){Q_{\til{W}_1,\til{W}_2}} & \til{W}_1\tens\til{W}_2 \\
 }
\end{equation*}
for $U_1\in I_{W_1}$ and $U_2\in I_{W_2}$, we get
\begin{align*}
 Q_{\til{W}_1,\til{W}_2}\circ(F_1\widehat{\tens} F_2)\circ\phi_{U_1,U_2} & = (i_{F_1(U_1)}\tens i_{F_2(U_2)})\circ (F_1\vert_{U_1}\tens F_2\vert_{U_2})\nonumber\\
 & =(F_1\tens F_2)\circ (i_{U_1}\tens i_{U_2}) \nonumber\\
 & =(F_1\tens F_2)\circ Q_{W_1,W_2}\circ\phi_{U_1,U_2}.
\end{align*}
This proves \eqref{eqn:Q_nat_trans}.

Now the diagrams in \eqref{diag:Q_unit_compat} commute as a consequence of the definitions of $\mathfrak{l}_W$ and $\mathfrak{r}_W$ together with the diagrams
\begin{equation*}
\xymatrixcolsep{3pc}
 \xymatrix{
  U\tens\til{U} \ar[d]_{\phi_{U,\til{U}}} \ar[rd]^(.55){i_{U}\tens i_{\til{U}}} \ar[r]^{i_U\tens\Id_{\til{U}}} &  V\tens\til{U} \ar[r]^{l_{\til{U}}} \ar[d]^{\Id_V\tens i_{\til{U}}}  & \til{U} \ar[d]^{i_{\til{U}}} \\
 V\widehat{\tens}W \ar[r]_{Q_{V,W}} & V\tens W \ar[r]_(.6){l_W} & W\\
 }
\end{equation*}
and
\begin{equation*}
 \xymatrixcolsep{3pc}
 \xymatrix{
  \til{U}\tens U \ar[d]_{\phi_{\til{U},U}} \ar[rd]^(.55){i_{\til{U}}\tens i_U} \ar[r]^{\Id_{\til{U}}\tens i_U} & \til{U}\tens V \ar[d]^{i_{\til{U}}\tens\Id_V} \ar[r]^{r_{\til{U}}}  & \til{U} \ar[d]^{i_{\til{U}}} \\
 W\widehat{\tens}V \ar[r]_{Q_{W,V}} & W\tens V \ar[r]_(.6){r_W} & W\\
 }
\end{equation*}
for $U\in I_V$, $\til{U}\in I_W$, which commute by the naturality of the unit isomorphisms in $\cC$.

For the diagram \eqref{diag:Q_assoc_compat}, we need to show that
\begin{align*}
 \cA_{W_1,W_2,W_3}\circ Q_{W_1,W_2\tens W_3} & \circ  (\Id_{W_1}\widehat{\tens} Q_{W_2,W_3})\circ T_{W_1,(W_2,W_3)}\nonumber\\
 &= Q_{W_1\tens W_2,W_3}\circ(Q_{W_1,W_2}\widehat{\tens}\Id_{W_3})\circ T_{(W_1,W_2),W_3}\circ\widehat{\cA}_{W_1,W_2,W_3}
\end{align*}
as morphisms in $\cW$ from $\varinjlim\alpha_{W_1}\tens(\alpha_{W_2}\tens\alpha_{W_3})$ to $(W_1\tens W_2)\tens W_3$. This follows from the calculations
\begin{align*}
& \cA_{W_1,W_2,W_3}\circ  Q_{W_1,W_2\tens W_3}\circ  (\Id_{W_1}\widehat{\tens} Q_{W_2,W_3})\circ T_{W_1,(W_2,W_3)}\circ \phi_{U_1,U_2,U_3}\nonumber\\
 &= \cA_{W_1,W_2,W_3}\circ Q_{W_1,W_2\tens W_3}\circ (\Id_{W_1}\widehat{\tens} Q_{W_2,W_3})\circ \phi_{U_1,\im\phi_{U_2,U_3}}\circ(\Id_{U_1}\tens\phi_{U_2,U_3}')\nonumber\\
 & = \cA_{W_1,W_2,W_3}\circ Q_{W_1,W_2\tens W_3}\circ \phi_{U_1,Q_{W_2,W_3}(\im\phi_{U_2,U_3})}\circ(\Id_{U_1}\tens Q_{W_2,W_3}\vert_{\im\phi_{U_2,U_3}})\circ(\Id_{U_1}\tens\phi_{U_2,U_3}') \nonumber\\
 & = \cA_{W_1,W_2,W_3}\circ Q_{W_1,W_2\tens W_3}\circ \phi_{U_1,\im(i_{U_2}\tens i_{U_3})}\circ(\Id_{U_1}\tens(i_{U_2}\tens i_{U_3})')\nonumber\\
 & =\cA_{W_1,W_2,W_3}\circ(i_{U_1}\tens(i_{U_2}\tens i_{U_3}))\nonumber\\
 & =((i_{U_1}\tens i_{U_2})\tens i_{U_3})\circ\cA_{U_1,U_2,U_3}
\end{align*}
and
\begin{align*}
& Q_{W_1\tens W_2,W_3}\circ(Q_{W_1,W_2}  \widehat{\tens}\Id_{W_3})\circ  T_{(W_1,W_2),W_3}\circ\widehat{\cA}_{W_1,W_2,W_3}\circ\phi_{U_1,U_2,U_3}\nonumber\\
 & =Q_{W_1\tens W_2,W_3}\circ(Q_{W_1,W_2}\widehat{\tens}\Id_{W_3})\circ  T_{(W_1,W_2),W_3}\circ\phi_{U_1,U_2,U_3}\circ\cA_{U_1,U_2,U_3}\nonumber\\
 & =Q_{W_1\tens W_2,W_3}\circ(Q_{W_1,W_2}\widehat{\tens}\Id_{W_3})\circ \phi_{\im\phi_{U_1,U_2},U_3}\circ(\phi_{U_1,U_2}'\tens\Id_{U_3})\circ\cA_{U_1,U_2,U_3}\nonumber\\
 & = Q_{W_1\tens W_2,W_3}\circ\phi_{Q_{W_1,W_2}(\im\phi_{U_1,U_2}),U_3}\circ(Q_{W_1,W_2}\vert_{\im\phi_{U_1,U_2}}\tens\Id_{U_3})\circ(\phi_{U_1,U_2}'\tens\Id_{U_3})\circ\cA_{U_1,U_2,U_3}\nonumber\\
 & = Q_{W_1\tens W_2,W_3}\circ\phi_{\im(i_{U_1}\tens i_{U_2}),U_3}\circ((i_{U_1}\tens i_{U_2})'\tens\Id_{U_3})\circ\cA_{U_1,U_2,U_3}\nonumber\\
 & =((i_{U_1}\tens i_{U_2})\tens i_{U_3})\circ\cA_{U_1,U_2,U_3}
\end{align*}
for objects $U_j\in I_{W_j}$, $j=1,2,3$. Here $(i_{U_1}\tens i_{U_2})': U_1\tens U_2\rightarrow\im(i_{U_1}\tens i_{U_2})$ is the surjection induced by $i_{U_1}\tens i_{U_2}$, and $(i_{U_2}\tens i_{U_3})'$ is similar.

Finally, \eqref{diag:Q_braid_compat} commutes as a result of the calculation
\begin{align*}
 Q_{W_2,W_1}\circ\widehat{\cR}_{W_1,W_2}\circ\phi_{U_1,U_2} & = Q_{W_2,W_1}\circ\phi_{U_2,U_1}\circ\cR_{U_1,U_2}\nonumber\\
 & = (i_{U_2}\tens i_{U_1})\circ\cR_{U_1,U_2}\nonumber\\
 & =\cR_{W_1,W_2}\circ(i_{U_1}\tens i_{U_2})\nonumber\\
 & =\cR_{W_1,W_2}\circ Q_{W_1,W_2}\circ\phi_{U_1,U_2}
\end{align*}
for objects $U_1\in I_{W_1}$ and $U_2\in I_{W_2}$.
\end{proof}

\section{The direct limit completion as a vertex tensor category}\label{sec:ind_lim_vrtx_tens}

Although the last section shows that braided tensor category structure on $\cC$ extends to braided tensor category structure on $\ind(\cC)$, we have not yet shown that this braided tensor category structure on $\ind(\cC)$ is the correct vertex algebraic structure. Especially, we have not yet related this braided tensor category structure to intertwining operators. In this section, we complete the proof of Theorem \ref{thm:main_thm}; the remaining conditions for this theorem that we need to impose now are the following:
\begin{assum}\label{assum:Sec6}
 The category $\cC$ of grading-restricted generalized $V$-modules satisfies:
 \begin{enumerate}
  \item The braided tensor category structure on $\cC$ is induced from vertex tensor category structure as described in Section \ref{sec:vertex_tensor}.

  \item For any intertwining operator $\cY$ of type $\binom{X}{W_1\,W_2}$ where $W_1$, $W_2$ are modules in $\cC$ and $X$ is a generalized module in $\ind(\cC)$, the image $\im\cY\subseteq X$ is a module in $\cC$.
 \end{enumerate}
\end{assum}

Using these conditions, as well as the assumptions of the previous sections, we first prove:
\begin{thm}\label{thm:indC_tens_prod_intw_op}
 Let $X_1$ and $X_2$ be modules in $\ind(\cC)$.
 \begin{enumerate}
  \item There is an intertwining operator $\cY_{X_1,X_2}$ of type $\binom{X_1\widehat{\tens} X_2}{X_1\,X_2}$ such that 
  if $W_1\subseteq X_1$ and $W_2\subseteq X_2$ are any $\cC$-submodules, then
  \begin{equation*}
   \cY_{X_1,X_2}\circ(i_{W_1}\otimes i_{W_2}) =\phi_{W_1,W_2}\circ\cY_{W_1,W_2},
  \end{equation*}
where $i_{W_1}: W_1\rightarrow X_1$, $i_{W_2}: W_2\rightarrow X_2$ are the inclusions and $\cY_{W_1,W_2}$ is the tensor product intertwining operator of type $\binom{W_1\tens W_2}{W_1\,W_2}$ in $\cC$.

\item For any weak module $X_3$ in $\ind(\cC)$, the linear map
\begin{align*}
 \hom_V(X_1\widehat{\tens} X_2, X_3) & \rightarrow \cV^{X_3}_{X_1, X_2}\nonumber\\
  F &\mapsto F\circ\cY_{X_1,X_2}
\end{align*}
is an isomorphism.
 \end{enumerate}
\end{thm}
\begin{proof}
 To construct $\cY_{X_1,X_2}$, we first need to show that $(X_1\otimes X_2,\lbrace i_{W_1}\otimes i_{W_2}\rbrace_{(W_1,W_2)\in I_{X_1}\times I_{X_2}})$ is a direct limit (in the category of vector spaces) of the inductive system
 \begin{equation*}
  \alpha_{X_1}\otimes\alpha_{X_2}: I_{X_1}\times I_{X_2}\rightarrow\cV ec
 \end{equation*}
defined by
\begin{equation*}
 (\alpha_{X_1}\otimes\alpha_{X_2})(W_1,W_2)=W_1\otimes W_2
\end{equation*}
for $W_1\in I_{X_1}$, $W_2\in I_{X_2}$, and by
\begin{equation*}
f_{W_1,W_2}^{\til{W}_1,\til{W}_2} =f_{W_1}^{\til{W}_1}\otimes f_{W_2}^{\til{W}_2}.
\end{equation*}
In fact, since $(X_1\otimes X_2,\lbrace i_{W_1}\otimes i_{W_2}\rbrace_{(W_1,W_2)\in I_{X_1}\times I_{X_2}})$ is a target of $\alpha_{X_1}\otimes\alpha_{X_2}$, there is a unique linear map
\begin{equation*}
 F_{X_1,X_2}: \varinjlim\alpha_{X_1}\otimes\alpha_{X_2}\rightarrow X_1\otimes X_2
\end{equation*}
such that $F_{X_1,X_2}\circ \phi_{W_1,W_2}=i_{W_1}\otimes i_{W_2}$ for $W_1\in I_{X_1}$, $W_2\in I_{X_2}$. The map $F_{X_1,X_2}$ is surjective because $X_1$ and $X_2$ are both the unions of their $\cC$-submodules. To show that $F_{X_1,X_2}$ is also injective, we use the injectivity of $i_{W_1}\otimes i_{W_2}$, which follows from the exactness of the tensor product on $\cV ec$: If $F_{X_1,X_2}(i_{W_1,W_2}(\til{w}))=0$ for some $\til{w}\in W_1\otimes W_2$, then $(i_{W_1}\otimes i_{W_2})(\til{w})=0$ as well, so that $\til{w}=0$. Thus $F_{X_1,X_2}$ is an isomorphism that identifies $(X_1\otimes X_2,\lbrace i_{W_1}\otimes i_{W_2}\rbrace_{(W_1,W_2)\in I_{X_1}\times I_{X_2}})$ with the direct limit of $\alpha_{X_1}\otimes \alpha_{X_2}$.

Now by the universal property of direct limits in $\cV ec$, there is a unique linear map
$$\cY_{X_1,X_2}(\cdot,x)\cdot : X_1\otimes X_2\rightarrow(X_1\widehat{\tens} X_2)[\log x]\lbrace x\rbrace$$
such that the diagram
\begin{equation*}
 \xymatrixcolsep{3pc}
 \xymatrix{
 W_1\otimes W_2 \ar[d]_{i_{W_1}\otimes i_{W_2}} \ar[r]^(.4){\cY_{W_1,W_2}} & (W_1\tens W_2)[\log x]\lbrace x\rbrace \ar[d]^{\phi_{W_1,W_2}} \\
 X_1\otimes X_2 \ar[r]_(.4){\cY_{X_1,X_2}} & (X_1\widehat{\tens}X_2)[\log x]\lbrace x\rbrace \\
 }
\end{equation*}
commutes for all $W_1\in I_{X_1}$ and $W_2\in I_{X_2}$. Indeed, $\cY_{X_1,X_2}$ is well defined because
\begin{align*}
 \phi_{\til{W}_1,\til{W}_2}\circ\cY_{\til{W}_1,\til{W}_2}\circ(f_{W_1}^{\til{W}_1}\otimes f_{W_2}^{\til{W}_2}) & =\phi_{\til{W}_1,\til{W}_2}\circ(f_{W_1}^{\til{W}_1}\tens f_{W_2}^{\til{W}_2})\circ\cY_{W_1,W_2} =\phi_{W_1,W_2}\circ\cY_{W_1,W_2},
\end{align*}
by definition of the tensor product of morphisms in $\cC$. Then $\cY_{X_1,X_2}$ is an intertwining operator because each $\cY_{W_1,W_2}$ is an intertwining operator and because $X_j$ for $j=1,2$ is the union of its $\cC$-submodules. For example, $\cY_{X_1,X_2}$ satisfies the $L(-1)$-derivative property because
\begin{align*}
\dfrac{d}{dx} \cY_{X_1,X_2}(i_{W_1}(w_1),x)i_{W_2}(w_2) & =\phi_{W_1,W_2}\left(\dfrac{d}{dx}\cY_{W_1,W_2}(w_1,x)w_2\right)\nonumber\\
& =\phi_{W_1,W_2}\left(\cY_{W_1,W_2}(L_{W_1}(-1)w_1,x)w_2\right)\nonumber\\
& =\cY_{X_1,X_2}(i_{W_1}(L_{W_1}(-1)w_1),x)i_{W_2}(w_2)\nonumber\\
& =\cY_{X_1,X_2}(L_{X_1}(-1)i_{W_1}(w_1),x)i_{W_2}(w_2)
\end{align*}
for $W_1\in I_{X_1}$, $W_2\in I_{X_2}$, $w_1\in W_1$, and $w_2\in W_2$. Similarly, the Jacobi identity follows from the relations
\begin{align*}
 Y_{X_1\widehat{\tens}X_2}(v,x_1) & \cY_{X_1,X_2}(i_{W_1}(w_1),x_2)i_{W_2}(w_2)\nonumber\\
 & = Y_{X_1\widehat{\tens}X_2}(v,x_1)\phi_{W_1,W_2}\left(\cY_{W_1,W_2}(w_1,x_2)w_2\right)\nonumber\\
 & =\phi_{W_1,W_2}\left(Y_{W_1\tens W_2}(v,x_1)\cY_{W_1,W_2}(w_1,x_2)w_2\right),
 \end{align*}
 \begin{align*}
 \cY_{X_1,X_2}(i_{W_1} & (w_1),x_2) Y_{X_2}(v,x_1)i_{W_2}(w_2)\nonumber\\
 & = \cY_{X_1,X_2}(i_{W_1}(w_1),x_2)i_{W_2}\left(Y_{W_2}(v,x_1)w_2\right)\nonumber\\
 & =\phi_{W_1,W_2}\left(\cY_{W_1,W_2}(w_1,x_2)Y_{W_2}(v,x_1)w_2\right),
 \end{align*}
 \begin{align*}
 \cY_{X_1,X_2}(Y_{X_1} & (v,x_0)i_{W_1}(w_1),x_2)i_{W_2}(w_2)\nonumber\\
 & =\cY_{X_1,X_2}(i_{W_1}(Y_{W_1}(v,x_0)w_1),x_2)i_{W_2}(w_2)\nonumber\\
 &=\phi_{W_1,W_2}\left( \cY_{W_1,W_2}(Y_{W_1},v,x_0)w_1,x_2)w_2\right)
\end{align*}
and the fact that $\cY_{W_1,W_2}$ satisfies the Jacobi identity. This proves part (1) of the theorem.

For part (2), we need to show that if $\cY$ is an intertwining operator of type $\binom{X_3}{X_1\,X_2}$, then there is a unique weak $V$-module homomorphism $F: X_1\widehat{\tens}X_2\rightarrow X_3$ such that $F\circ\cY_{X_1,X_2}=\cY$. For each $W_1\in I_{X_1}$ and $W_2\in I_{X_2}$, $\cY\circ(i_{W_1}\otimes i_{W_2})$ is an intertwining operator of type $\binom{X_3}{W_1\,W_2}$, so by Assumption \ref{assum:Sec6}(2), $\im\cY\circ(i_{W_1}\otimes i_{W_2})$ is an object of $\cC$. Then by the universal property of the tensor product in $\cC$, there is a unique $V$-module homomorphism $F_{W_1,W_2}: W_1\tens W_2\rightarrow X_3$ such that
\begin{equation*}
 \cY\circ(i_{W_1}\otimes i_{W_2})=F_{W_1,W_2}\circ\cY_{W_1,W_2}.
\end{equation*}
By the universal property of the direct limit $X_1\widehat{\tens}X_2=\varinjlim\alpha_{X_1}\tens\alpha_{X_2}$, there is then a unique weak $V$-module homomorphism $F: X_1\widehat{\tens} X_2\rightarrow X_3$ such that $F\circ\phi_{W_1,W_2}=F_{W_1,W_2}$, provided
\begin{equation}\label{eqn:tens_univ_prop}
 F_{\til{W}_1,\til{W}_2}\circ(f_{W_1}^{\til{W_1}}\tens f_{W_2}^{\til{W}_2})=F_{W_1,W_2}
\end{equation}
for all $W_1\subseteq \til{W}_1$ in $I_{X_1}$ and $W_2\subseteq\til{W}_2$ in $I_{X_2}$. To show this, we calculate
\begin{align*}
  F_{\til{W}_1,\til{W}_2}\circ(f_{W_1}^{\til{W_1}}\tens f_{W_2}^{\til{W}_2})\circ\cY_{W_1,W_2} & = F_{\til{W}_1,\til{W}_2}\circ\cY_{\til{W}_1,\til{W}_2}\circ(f_{W_1}^{\til{W_1}}\otimes f_{W_2}^{\til{W}_2})\nonumber\\
 & =\cY\circ(i_{\til{W}_1}\otimes i_{\til{W}_2})\circ(f_{W_1}^{\til{W_1}}\otimes f_{W_2}^{\til{W}_2})\nonumber\\
 & =\cY\circ(i_{W_1}\otimes i_{W_2})\nonumber\\
 & = F_{W_1,W_2}\circ\cY_{W_1,W_2}.
\end{align*}
Then \eqref{eqn:tens_univ_prop} follows from the surjectivity of $\cY_{W_1,W_2}$. This shows that $F$ exists. Then for $W_1\in I_{X_1}$, $W_2\in I_{X_2}$, we have
\begin{align*}
 F\circ\cY_{X_1,X_2}\circ(i_{W_1}\otimes i_{W_2}) &=F\circ\phi_{W_1,W_2}\circ\cY_{W_1,W_2}\nonumber\\
 &= F_{W_1,W_2}\circ\cY_{W_1,W_2}\nonumber\\
 &=\cY\circ(i_{W_1}\otimes i_{W_2}).
\end{align*}
Since $X_1\otimes X_2$ is the union of the images of such $i_{W_1}\otimes i_{W_2}$, we get $F\circ\cY_{X_1,X_2}=\cY$. Finally, the uniqueness of $F$ follows from surjectivity of the intertwining operator $\cY_{X_1,X_2}$. This surjectivity is a consequence of the definition of $\cY_{X_1,X_2}$, since $X_1\widehat{\tens}X_2$ is spanned by the images of $\phi_{W_1,W_2}$ for $W_1\in I_{X_1}$ and $W_2\in I_{X_2}$, and since each $\cY_{W_1,W_2}$ is surjective.
\end{proof}

We will use the intertwining operators $\cY_{X_1,X_2}$ of the preceding theorem to describe the tensor category structure on $\ind(\cC)$. For describing the associativity isomorphisms, we will use the fact that every weak module $X$ in $\ind(\cC)$ is a generalized $V$-module, so that the contragredient $X'=\bigoplus_{h\in\CC} X_{[h]}^*$ exists, though not necessarily as a module in $\ind(\cC)$. Note that the full vector space dual $(X')^*$ contains the algebraic completion $\overline{X}=\prod_{h\in\CC} X_{[h]}$, but these two vector spaces are not equal unless all the $X_{[h]}$ are finite dimensional. A grading-preserving linear map $f: W\rightarrow X$ between two generalized $V$-modules extends to a linear map $\overline{f}: \overline{W}\rightarrow\overline{X}$ in the obvious way. It also induces a grading-preserving linear map
\begin{equation*}
 f^*: X'\rightarrow W'
\end{equation*}
that satisfies
\begin{equation*}
 \left\langle f^*(b'), \overline{w}\right\rangle =\left\langle b', \overline{f}(\overline{w})\right\rangle
\end{equation*}
for $b'\in X'$ and $\overline{w}\in\overline{W}$.

\begin{thm}\label{thm:indC_vrtx_tens}
 Under Assumptions \ref{assum:Sec4}, \ref{assum:Sec5}, and \ref{assum:Sec6}, the braided tensor category with twist structure on $\ind(\cC)$ is given as follows:
 \begin{enumerate}
 \item The tensor product of objects $X_1$ and $X_2$ in $\ind(\cC)$ is $X_1\widehat{\tens}X_2$ and the tensor product of morphisms $F_1: X_1\rightarrow\til{X}_1$ and $F_2: X_2\rightarrow\til{X}_2$ is characterized by
 \begin{equation*}
  (F_1\widehat{\tens}F_2)\left(\cY_{X_1,X_2}(b_1,x)b_2\right) =\cY_{\til{X}_1,\til{X}_2}(F_1(b_1),x)F_2(b_2)
 \end{equation*}
for $b_1\in X_1$, $b_2\in X_2$.

 \item The unit object of $\ind(\cC)$ is $V$, and for a module $X$ in $\ind(\cC)$ the unit isomorphisms are given by
  \begin{equation*}
   \mathfrak{l}_X\left(\cY_{V,X}(v,x)b\right) =Y_X(v,x)b \quad\text{and}\quad \mathfrak{r}_X\left(\cY_{X,V}(b,x)v\right) = e^{xL(-1)} Y_X(v,-x)b
  \end{equation*}
for $v\in V$, $b\in X$.

\item For modules $X_1$, $X_2$, and $X_3$ in $\ind(\cC)$, for vectors $b_1\in X_1$, $b_2\in X_2$, and $b_3\in X_3$, and for $r_1,r_2\in\RR$ such that $r_1>r_2>r_1-r_2>0$, the series
\begin{align}\label{eqn:assoc_prod_series}
  \sum_{h\in\CC} \cY_{X_1,X_2\widehat{\tens} X_3}(b_1,e^{\ln r_1})\pi_h\left(\cY_{X_2,X_3}(b_2,e^{\ln r_2})b_3\right)
\end{align}
of vectors in $\overline{X_1\widehat{\tens}(X_2\widehat{\tens} X_3)}$ converges absolutely to a vector
$$\cY_{X_1,X_2\widehat{\tens} X_3}(b_1,e^{\ln r_1})\cY_{X_2,X_3}(b_2,e^{\ln r_2})b_3\in\overline{X_1\widehat{\tens}(X_2\widehat{\tens}X_3)}\subseteq((X_1\widehat{\tens}(X_2\widehat{\tens}X_3))')^*,$$
and the series
\begin{align}\label{eqn:assoc_it_series}
 \sum_{h\in\CC} \cY_{X_1\widehat{\tens}X_2,X_3}\left(\pi_h\left(\cY_{X_1,X_2}(b_1,e^{\ln(r_1-r_2)})b_2\right),e^{\ln r_2}\right)b_3
\end{align}
of vectors in $\overline{(X_1\widehat{\tens}X_2)\widehat{\tens}X_3}$ converges absolutely to a vector $$\cY_{X_1\widehat{\tens}X_2,X_3}\left(\cY_{X_1,X_2}(b_1,e^{\ln(r_1-r_2)})b_2,e^{\ln r_2}\right)b_3\in\overline{(X_1\widehat{\tens}X_2)\widehat{\tens}X_3}\subseteq(((X_1\widehat{\tens}X_2)\widehat{\tens}X_3)')^*.$$
Moreover,
\begin{align}\label{eqn:IndC_vrtx_assoc}
\overline{\cA_{X_1,X_2,X_3}} & \left(\cY_{X_1,X_2\widehat{\tens} X_3}(b_1,e^{\ln r_1})\cY_{X_2,X_3}(b_2,e^{\ln r_2})b_3\right)\nonumber\\
&= \cY_{X_1\widehat{\tens}X_2,X_3}\left(\cY_{X_1,X_2}(b_1,e^{\ln(r_1-r_2)})b_2,e^{\ln r_2}\right)b_3.
\end{align}

\item For modules $X_1$ and $X_2$ in $\ind(\cC)$, the braiding isomorphism $\widehat{\cR}_{X_1,X_2}$ satisfies
\begin{equation*}
 \widehat{\cR}_{X_1,X_2}\left(\cY_{X_1,X_2}(b_1,x)b_2\right) =e^{xL(-1)}\cY_{X_2,X_1}(b_2,e^{\pi i} x)b_1
\end{equation*}
for $b_1\in X_1$, $b_2\in X_2$.

\item For a module $X$ in $\ind(\cC)$, $\Theta_X=e^{2\pi i L(0)}$.
 \end{enumerate}
 In particular, $\ind(\cC)$ is a braided tensor category with twist as described in Section \ref{sec:vertex_tensor}.
\end{thm}
\begin{proof}
For morphisms $F_1: X_1\rightarrow\til{X}_1$ and $F_2: X_2\rightarrow\til{X}_2$ in $\ind(\cC)$, we may assume $b_1=i_{W_1}(w_1)$ for some $W_1\in I_{W_1}$, $w_1\in W_1$ and $b_2=i_{W_2}(w_2)$ for some $W_2\in I_{W_2}$, $w_2\in W_2$. Then by the definitions,
\begin{align*}
 (F_1\widehat{\tens}F_2)\left(\cY_{X_1,X_2}(b_1,x)b_2\right) & = (F_1\widehat{\tens}F_2)\left(\cY_{X_1,X_2}(i_{W_1}(w_1),x)i_{W_2}(w_2)\right)\nonumber\\
 & = \left((F_1\widehat{\tens}F_2)\circ\phi_{W_1,W_2}\right)\left(\cY_{W_1,W_2}(w_1,x)w_2\right)\nonumber\\
 & =\left(\phi_{F_1(W_1),F_2(W_2)}\circ(F_1\vert_{W_1}\tens F_2\vert_{W_2})\right)\left(\cY_{W_1,W_2}(w_1,x)w_2\right)\nonumber\\
 & =\phi_{F_1(W_1),F_2(W_2)}\left(\cY_{F_1(W_1),F_2(W_2)}(F_1(i_{W_1}(w_1)),x)F_2(i_{W_2}(w_2))\right)\nonumber\\
 & =\cY_{\til{X}_1,\til{X}_2}(F_1(b_1),x)F_2(b_2)
\end{align*}
as required.

 For the unit isomorphisms $\mathfrak{l}_X$ and $\mathfrak{r}_X$, we may assume that $b\in X$ is given by $b=i_W(w)$ for some $W\in I_X$ and $w\in W$. Then
 \begin{align*}
  \mathfrak{l}_X\left(\cY_{V,X}(v,x)b\right) & = \mathfrak{l}_X\left(\cY_{V,X}(i_V(v),x)i_W(w)\right)\nonumber\\
  & = (\mathfrak{l}_X\circ\phi_{V,W})\left(\cY_{V,W}(v,x)w\right)\nonumber\\
  & =(i_W\circ l_{W})\left(\cY_{V,W}(v,x)w\right)\nonumber\\
  & = i_W\left(Y_W(v,x)w\right)\nonumber\\
  & =Y_X(v,x)i_W(w) = Y_X(v,x)b
 \end{align*}
and
\begin{align*}
  \mathfrak{r}_X\left(\cY_{X,V}(b,x)v\right) & = \mathfrak{r}_X\left(\cY_{X,V}(i_W(w),x)i_V(v)\right)\nonumber\\
  & = (\mathfrak{r}_X\circ\phi_{W,V})\left(\cY_{W,V}(w,x)v\right)\nonumber\\
  & =(i_W\circ r_{W})\left(\cY_{W,V}(w,x)v\right)\nonumber\\
  & = i_W\left(e^{xL(-1)}Y_W(v,-x)w\right)\nonumber\\
  & =e^{xL(-1)}Y_X(v,-x)i_W(w) = e^{xL(-1)}Y_X(v,-x)b
 \end{align*}
 by the definition of the unit isomorphisms in $\cC$.

 For the associativity isomorphism $\cA_{X_1,X_2,X_3}$, take $b_1\in X_1$, $b_2\in X_2$, and $b_3\in X_3$. For $j=1,2,3$, we may assume that $b_j=i_{W_j}(w_j)$ for some $W_j\in I_{X_j}$ and $w_j\in W_j$. Then for $b'\in(X_1\widehat{\tens}(X_2\widehat{\tens}X_3))'$ and $h\in\CC$, we use the definitions to get
 \begin{align*}
  &\left\langle b', \cY_{X_1,X_2\widehat{\tens} X_3}(b_1,e^{\ln r_1})\pi_h\left(\cY_{X_2,X_3}(b_2,e^{\ln r_2})b_3\right)\right\rangle\nonumber\\
  &\hspace{1em} =\left\langle b', \cY_{X_1,X_2\widehat{\tens} X_3}(i_{W_1}(w_1),e^{\ln r_1})\pi_h\left(\cY_{X_2,X_3}(i_{W_2}(w_2),e^{\ln r_2})i_{W_3}(w_3)\right)\right\rangle\nonumber\\
  &\hspace{1em} = \left\langle b',   \cY_{X_1,X_2\widehat{\tens} X_3}(i_{W_1}(w_1),e^{\ln r_1})(\pi_h\circ\overline{\phi_{W_2,W_3}})\left(\cY_{W_2,W_3}(w_2,e^{\ln r_2})w_3\right)\right\rangle\nonumber\\
  &\hspace{1em} =\left\langle b', \cY_{X_1,X_2\widehat{\tens} X_3}(i_{W_1}(w_1),e^{\ln r_1})(\phi_{W_2,W_3}\circ\pi_h)\left(\cY_{W_2,W_3}(w_2,e^{\ln r_2})w_3\right)\right\rangle\nonumber\\
  &\hspace{1em} = \left\langle b', \overline{\phi_{W_1,\im\phi_{W_2,W_3}}}\left(\cY_{W_1,\im\phi_{W_2,W_3}}(w_1,e^{\ln r_1})(\phi_{W_2,W_3}'\circ\pi_h)\left(\cY_{W_2,W_3}(w_2,e^{\ln r_2})w_3\right)\right)\right\rangle\nonumber\\
  &\hspace{1em} =\left\langle b',\overline{\phi_{W_1,\im\phi_{W_2,W_3}}\circ(\Id_{W_1}\tens\phi_{W_2,W_3}')}\left(\cY_{W_1,W_2\tens W_3}(w_1,e^{\ln r_1})\pi_h\left(\cY_{W_2,W_3}(w_2,e^{\ln r_2})w_3\right)\right)\right\rangle\nonumber\\
  & \hspace{1em} = \left\langle (\phi_{W_1,\im\phi_{W_2,W_3}}\circ(\Id_{W_1}\tens\phi_{W_2,W_3}'))^*(b'), \cY_{W_1,W_2\tens W_3}(w_1,e^{\ln r_1})\pi_h\left(\cY_{W_2,W_3}(w_2,e^{\ln r_2})w_3\right)\right\rangle
 \end{align*}
Now the convergence of products of intertwining operators in $\cC$ implies that the sum over $h\in\CC$ converges absolutely to some
\begin{equation*}
 \cY_{X_1,X_2\widehat{\tens} X_3}(b_1,e^{\ln r_1})\cY_{X_2,X_3}(b_2,e^{\ln r_2})b_3\in((X_1\widehat{\tens}(X_2\widehat{\tens}X_3))')^*.
\end{equation*}
Moreover,
\begin{align*}
  \cY_{X_1,X_2\widehat{\tens} X_3} & (b_1,e^{\ln r_1})  \cY_{X_2,X_3}(b_2,e^{\ln r_2})b_3\nonumber\\ &=\left\langle(\phi_{W_1,\im\phi_{W_2,W_3}}\circ(\Id_{W_1}\tens\phi_{W_2,W_3}'))^*(\cdot), \cY_{W_1,W_2\tens W_3}(w_1,e^{\ln r_1})\cY_{W_2,W_3}(w_2,e^{\ln r_2})w_3\right\rangle\nonumber\\
  & =\overline{\phi_{W_1,\im\phi_{W_2,W_3}}\circ(\Id_{W_1}\tens\phi_{W_2,W_3}')}\left(\cY_{W_1,W_2\tens W_3}(w_1,e^{\ln r_1})\cY_{W_2,W_3}(w_2,e^{\ln r_2})w_3\right)\nonumber\\
&  \in\im\overline{\phi_{W_1,\im\phi_{W_2,W_3}}\circ(\Id_{W_1}\tens\phi_{W_2,W_3}')}\subseteq\overline{X_1\widehat{\tens}(X_2\widehat{\tens}X_3)}.
\end{align*}
Similarly, the series in \eqref{eqn:assoc_it_series} converges absolutely to
\begin{align*}
 \cY_{X_1\widehat{\tens}X_2,X_3} & \left(\cY_{X_1,X_2}(b_1,e^{\ln(r_1-r_2)})b_2,e^{\ln r_2}\right)b_3\nonumber\\
 & =\overline{\phi_{\im\phi_{W_1,W_2},W_3}\circ(\phi_{W_1,W_2}'\tens\Id_{W_3})}\left(\cY_{W_1\tens W_2,W_3}\left(\cY_{W_1,W_2}(w_1,e^{\ln(r_1-r_2)})w_2,e^{\ln r_2}\right)w_3\right)\nonumber\\
 &\in\overline{(X_1\widehat{\tens}X_2)\widehat{\tens}X_3}.
\end{align*}
Then since
\begin{align*}
 \fA_{X_1,X_2,X_3}\circ\phi_{W_1,\im\phi_{W_2,W_3}}\circ(\Id_{W_1}\tens\phi_{W_2,W_3}') &=\fA_{X_1,X_2,X_3}\circ T_{X_1,(X_2,X_3)}\circ\phi_{W_1,W_2,W_3}\nonumber\\
 & =T_{(X_1,X_2),X_3}\circ\widehat{\cA}_{X_1,X_2,X_3}\circ\phi_{W_1,W_2,W_3}\nonumber\\
 & =T_{(X_1,X_2),X_3}\circ\phi_{W_1,W_2,W_3}\circ\cA_{W_1,W_2,W_3}\nonumber\\
 & =\phi_{\im\phi_{W_1,W_2},W_3}\circ(\phi_{W_1,W_2}'\tens\Id_{W_3})\circ\cA_{W_1,W_2,W_3},
\end{align*}
\eqref{eqn:IndC_vrtx_assoc} now follows from the definition of the associativity isomorphisms in $\cC$.

For the braiding isomorphism $\widehat{\cR}_{X_1,X_2}$, take $b_1\in X_1$, $b_2\in X_2$ such that for $j=1,2$, $b_j=i_{W_j}(w_j)$ for some $W_j\in I_{X_j}$ and $w_j\in W_j$. Then
\begin{align*}
 \widehat{\cR}_{X_1,X_2}\left(\cY_{X_1,X_2}(b_1,x)b_2\right) & =\widehat{\cR}_{X_1,X_2}\left(\cY_{X_1,X_2}(i_{W_1}(w_1),x)i_{W_2}(w_2)\right)\nonumber\\
 & =(\widehat{\cR}_{X_1,X_2}\circ\phi_{W_1,W_2})\left(\cY_{W_1,W_2}(w_1,x)w_2\right)\nonumber\\
 & =(\phi_{W_2,W_1}\circ\cR_{W_1,W_2})\left(\cY_{W_1,W_2}(w_1,x)w_2\right)\nonumber\\
 & =\phi_{W_2,W_1}\left(e^{xL(-1)}\cY_{W_2,W_1}(w_2,e^{\pi i} x) w_1\right)\nonumber\\
 & = e^{xL(-1)}\cY_{X_2,X_1}(i_{W_2}(w_2),e^{\pi i}x)i_{W_1}(w_1) = e^{xL(-1)}\cY_{X_2,X_1}(b_2,e^{\pi i}x)b_1
\end{align*}
by the definition of the braiding isomorphisms in $\cC$.

Finally, for an object $X$ in $\ind(\cC)$ and $b=i_W(w)\in X$,
\begin{equation*}
 \Theta_X(b)=\Theta_X(i_W(w))=i_W(\theta_W(w))=i_W(e^{2\pi i L(0)}w)=e^{2\pi i L(0)}i_W(w) =e^{2\pi i L(0)}b
\end{equation*}
by \eqref{diag:Theta} and the fact that $i_W$ is a homomorphism of generalized $V$-modules.
\end{proof}

The previous theorem completes the proof of Theorem \ref{thm:main_thm}, except for the assertion in Theorem \ref{thm:main_thm} that $\ind(\cC)$ has the $P(z)$-vertex tensor category structure of \cite{HLZ1}-\cite{HLZ8} extending that on $\cC$. The existence of $P(z)$-vertex tensor category structure on $\ind(\cC)$ can be proved exactly as in \cite[Section 3.5]{CKM-ext}. In particular, for generalized modules $X_1$, $X_2$ in $\ind(\cC)$ and $z\in\CC^\times$, we can take the $P(z)$-tensor product $X_1\widehat{\tens}_{P(z)} X_2$ to be  $X_1\widehat{\boxtimes}X_2$ equipped with the $P(z)$-intertwining map $\cY_{X_1,X_2}(\cdot,e^{\log z})\cdot$ for some choice of branch $\log z$ of logarithm. Although the module category considered in \cite{CKM-ext} is different (it is the local module category $\rep^0 A$ of a vertex operator superalgebra extension of a vertex operator subalgebra $V$), the tensor product and all structure isomorphisms in the category of \cite{CKM-ext} are characterized by intertwining operators exactly as in Theorem \ref{thm:indC_vrtx_tens}. So the construction of $P(z)$-tensor products and parallel transport, unit, associativity, and braiding isomorphisms of the $P(z)$-vertex tensor category structure works the same for $\ind(\cC)$ as it does for $\rep^0 A$.

Finally, the assertion that the $P(z)$-vertex tensor category structure on $\ind(\cC)$ extends that on $\cC$ amounts to the assertion that the embedding $\cC\hookrightarrow\ind(\cC)$ is a ``vertex tensor functor'' in the sense of \cite[Section 3.6]{CKM-ext}. This is proved in exactly the same way as \cite[Theorem 3.68]{CKM-ext} (where the functor under consideration is the induction functor into the local module category $\rep^0 A$ of a vertex operator (super)algebra extension); see also \cite[Remark 2.7]{McR1}.

\section{Examples and applications}\label{sec:exam}

In this section, we give examples of vertex operator algebra module categories that satisfy the conditions of Theorem \ref{thm:main_thm}: the basic examples are $C_1$-cofinite module categories for Virasoro and affine vertex operator algebras. We then apply extension theory \cite{CKM-ext} to demonstrate the existence of braided tensor categories of generalized modules for certain infinite-order extensions of Virasoro and affine vertex operator algebras, such as singlet vertex operator algebras.

For a vertex operator algebra $V$, the category $\cC^1_V$ of $C_1$-cofinite grading-restricted generalized $V$-modules will satisfy   the conditions of Theorem \ref{thm:main_thm}, provided it satisfies the hypothesis of Theorem \ref{thm:C1_vrtx_tens}:
\begin{thm}\label{thm:ind_C_1_V}
 Suppose the category $\cC^1_V$ of $C_1$-cofinite grading-restricted generalized modules for a vertex operator algebra $V$ is closed under contragredient modules. Then $\cC^1_V$ satisfies the conditions of Theorem \ref{thm:main_thm}, so that $\ind(\cC^1_V)$ admits vertex and braided tensor category structures extending those on $\cC^1_V$.
\end{thm}
\begin{proof}
 We verify the conditions of Theorem \ref{thm:main_thm} one by one:
 \begin{enumerate}
  \item $V$ is an object of $\cC^1_V$: Since $v_{-1}\vac=v$ for any $v\in V$, $C_1(V)$ contains $\bigoplus_{n\geq 1} V_{(n)}$. Then the grading-restriction conditions on $V$ imply that $\dim V/C_1(V)<\infty$.

  \item $\cC^1_V$ is closed under submodules, quotients, and finite direct sums: It is easy to see that quotients and finite direct sums of $C_1$-cofinite modules are $C_1$-cofinite. For submodules, we use the fact that a submodule of a $V$-module $W$ is the contragredient of a quotient of $W'$, together with the assumption that contragredients of $C_1$-cofinite modules are $C_1$-cofinite.

  \item Every module in $\cC^1_V$ is finitely generated by Proposition \ref{prop:C1-fin-gen}.

  \item The category $\cC^1_V$ admits $P(z)$-vertex and braided tensor category structures as described in Section \ref{sec:vertex_tensor} by Theorem \ref{thm:C1_vrtx_tens}.

  \item The intertwining operator condition of Theorem \ref{thm:main_thm} follows from Corollary \ref{cor:ImY_C1}.
 \end{enumerate}
\end{proof}

Perhaps the first candidates for vertex operator algebras which satisfy the conditions of Theorem \ref{thm:ind_C_1_V} are simple Virasoro and affine vertex operator algebras:
\begin{exam}\label{exam:Vir_tens_cat}
 For any central charge $c\in\CC$, it is shown in \cite{CJORY} that the category $\cC^1_c$ of $C_1$-cofinite modules for the simple Virasoro vertex operator algebra $L(c,0)$ is equal to the category of finite-length $L(c,0)$-modules that have $C_1$-cofinite composition factors. In particular, $\cC_c^1$ is closed under contragredients, so Theorem \ref{thm:ind_C_1_V} implies that $\ind(\cC^1_c)$ has vertex and braided tensor category structures extending those on $\cC_c^1$.
\end{exam}

\begin{exam}
Now consider the simple vertex operator (super)algebra $V_k(\mathfrak{g})$ associated to a finite-dimensional Lie (super)algebra $\mathfrak{g}$ at non-critical level $k$. The category of $C_1$-cofinite grading-restricted generalized modules is called the Kazhdan-Lusztig category $KL_k(\mathfrak{g})$; it is not hard to show that it is also the category of finitely-generated grading-restricted generalized $V_k(\mathfrak{g})$-modules. (Proposition \ref{prop:C1-fin-gen} shows that $C_1$-cofinite modules are finitely generated. On the other hand, the grading restriction conditions imply that a finitely-generated grading-restricted generalized module is also finitely generated as a $U(\widehat{\mathfrak{g}}_-)$-module, which implies $C_1$-cofiniteness.) For $\mathfrak{g}$ a simple Lie algebra, it is known that $KL_k(\mathfrak{g})$ is closed under contragredients when $k\in\CC\setminus\mathbb{Q}$ (in which case $KL_k(\mathfrak{g})$ is semisimple), when $k+h^\vee\in\mathbb{Q}_{<0}$ \cite{KL2}, when $k$ is admissible \cite{Ar}, and for certain non-admissible $k$ such that $k+h^\vee\in\mathbb{Q}_{>0}$ \cite{CY}. For $\mathfrak{h}$ an abelian Lie algebra, $V_k(\mathfrak{h})$ for $k\neq 0$ is a Heisenberg vertex operator algebra; modules in $KL_k(\mathfrak{h})$ are finite-length extensions of irreducible Fock modules and thus $KL_k(\mathfrak{h})$ is closed under contragredients. In all these cases, Theorem \ref{thm:ind_C_1_V} shows that $\ind(KL_k(\mathfrak{g}))$ has vertex and braided tensor category structures extending those on $KL_k(\mathfrak{g})$.
\end{exam}

Now the reason we want vertex and braided tensor category structures on categories such as $\ind(\cC^1_c)$ and $\ind(KL_k(\mathfrak{g}))$ is that we want to apply the extension theory of \cite{CKM-ext} to infinite-order extensions of vertex operator algebras. An extension of a vertex operator algebra $V$ is a vertex operator algebra $A$ that contains $V$ as a vertex operator subalgebra. In particular, $V$ and $A$ have the same conformal vector, so that $V$ conformally embeds into $A$. For example, any vertex operator algebra $V$ of central charge $c$ contains a Virasoro vertex operator subalgebra. Although $V$ will not usually be $C_1$-cofinite as a Virasoro module, $V$ might be an object of $\ind(\cC^1_c)$. In this case, we would like to view $V$ as a commutative algebra in the braided tensor category $\ind(\cC^1_c)$ (as in \cite{HKL}) and then apply \cite{CKM-ext} to obtain vertex and braided tensor category structures on the category of generalized $V$-modules in $\ind(\cC^1_c)$.

We recall the definition of commutative algebra in a braided tensor category:
\begin{defi}
 Let $(\cC,\tens,\vac,l,r,\cA,\cR)$ be a braided tensor category. A \textit{commutative algebra} in $\cC$ is an object $A$ equipped with a multiplication morphism $\mu_A: A\tens A\rightarrow A$ and a unit morphism $\iota_A:\vac\rightarrow A$ that satisfy the following properties:
 \begin{enumerate}
  \item Unit: $\mu_A\circ(\iota_A\tens\Id_A)=l_A$
  \item Associativity: $\mu_A\circ(\Id_A\tens\mu_A)=\mu_A\circ(\mu_A\tens\Id_A)\circ\cA_{A,A,A}$
  \item Commutativity: $\mu_A=\mu_A\circ\cR_{A,A}$.
 \end{enumerate}
\end{defi}

The following result is a version of Theorem 3.2 (and Remark 3.3) of \cite{HKL} that is relevant for our setting, although the statement differs somewhat from \cite[Theorem 3.2]{HKL} since here we choose to avoid the implicit assumption in \cite{HKL} that the unit $\iota_A$ is injective. In the statement, $Y$ represents the vertex operators for the vertex operator algebras $V$ and $A$, while $Y_A: V\otimes A\rightarrow A((x))$ is the vertex operator for $A$ as a $V$-module.
\begin{thm}\label{thm:ext_alg}
 Let $(V,Y,\vac,\omega)$ be a vertex operator algebra and $\cC$ a category of grading-restricted generalized $V$-modules that satisfies the conditions of Theorem \ref{thm:main_thm}. Then the following two categories are isomorphic:
 \begin{enumerate}
  \item Vertex operator algebras $(A,Y,\vac_A,\omega_A)$ such that:
  \begin{itemize}
   \item $A$ is a $V$-module in $\ind(\cC)$,
   \item $Y_A(v,x)=Y(v_{-1}\vac_A,x)$ for $v\in V$, and
   \item $\omega_A=L(-2)\vac_A \quad (\,=\omega_{-1}\vac_A\,)$.
  \end{itemize}
\item Commutative algebras $(A,\mu_A,\iota_A)$ in the braided tensor category $\ind(\cC)$ such that $A$ is $\ZZ$-graded by $L(0)$-eigenvalues and satisfies the grading restriction conditions.
 \end{enumerate}
\end{thm}
\begin{proof}
 The proof is essentially the same as that of \cite[Theorem 3.2]{HKL}, so we will just briefly indicate how the isomorphism of categories goes. Given a vertex operator algebra $A$ as in the statement of the theorem, the grading of $A$ by $L(0)$-eigenvalues as a $V$-module must agree with the $\ZZ$-grading of $A$ as a vertex operator algebra since
 \begin{equation*}
  Y_A(\omega,x) =Y(L(-2)\vac_A, x)=Y(\omega_A, x).
 \end{equation*}
In particular, $A$ is $\ZZ$-graded by $L(0)$-eigenvalues and satisfies the grading restriction conditions. The algebra unit $\iota_A: V\rightarrow A$ is defined by $\iota_A(v)=v_{-1}\vac_A$. This is a $V$-module homomorphism by \cite[Proposition 4.7.7]{LL}, since $\vac_A$ is a vacuum-like vector for $V$:
 \begin{equation*}
  Y_A(v,x)\vac_A=Y(v_{-1}\vac_A,x)\vac_A\in A[[x]]
 \end{equation*}
for $v\in V$. (In fact, $\iota_A$ is also a homomorphism of vertex operator algebras.) The algebra multiplication $\mu_A: A\widehat{\tens}A\rightarrow A$ is the unique $V$-module homomorphism, guaranteed by Theorem \ref{thm:indC_tens_prod_intw_op}(2), such that $\mu_A\circ\cY_{A,A}=Y$; note that $Y$ is indeed a $V$-module intertwining operator of type $\binom{A}{A\,A}$ since $Y_A=Y\circ(\iota_A\otimes\Id_A)$ and $L(-1)$ agrees with $L_A(-1)=(\omega_A)_{-1}$ on $A$. Using Theorem \ref{thm:indC_vrtx_tens}, $\mu_A$ is commutative and associative exactly as in \cite[Theorem 3.2]{HKL}, and the unit axiom follows from
\begin{equation*}
(\mu_A\circ(\iota_A\widehat{\tens}\Id_A))\left(\cY_{V,A}(v,x)a\right) =Y(v_{-1}\vac_A,x)a=Y_A(v,x)a=\mathfrak{l}_A\left(\cY_{V,A}(v,x)a\right)
\end{equation*}
for $v\in V$, $a\in A$.

On the other hand, if $(A,\mu_A,\iota_A)$ is a commutative associative algebra in $\ind(\cC)$ with a restricted $\ZZ$-grading by $L(0)$-eigenvalues, then $A$ is a vertex operator algebra with vertex operator $Y=\mu_A\circ\cY_{A,A}$, vacuum $\vac_A=\iota_A(\vac)$, and conformal vector $\omega_A=\iota_A(\omega)$, just as in \cite[Theorem 3.2]{HKL}, .
\end{proof}

\begin{rem}
The notion of commutative algebra generalizes to superalgebra, that is, an object graded by parity for which the multiplication is supercommutative (see \cite[Definition 2.14]{CKL}). Commutative superalgebras correspond precisely to vertex superalgebra extensions \cite[Theorem 3.13]{CKL}, and similar to the above theorem, the argument also works in the direct limit completion.
\end{rem}

Theorem \ref{thm:ext_alg} means that we can use the methods of braided tensor categories to study suitable (possibly infinite-order) extensions of a vertex operator algebra $V$. In fact, we can apply the extension theory developed in \cite{CKM-ext} with the braided tensor category $\cC$ replaced by $\ind(\cC)$: although it was assumed in \cite{CKM-ext} that the generalized $V$-modules under consideration were grading-restricted, this assumption was not used in any of the proofs. So \cite[Theorem 3.65]{CKM-ext} applies to our setting:
\begin{thm}\label{thm:Rep0A}
 Let $V$ be a vertex operator algebra, $\cC$ a category of grading-restricted generalized $V$-modules that satisfies the conditions of Theorem \ref{thm:main_thm}, $A$ a vertex operator algebra that satisfies the conditions of Theorem \ref{thm:ext_alg}(1), and $\rep^0 A$ the category of generalized $A$-modules $X$ which are objects of $\ind(\cC)$ as $V$-modules with respect to the vertex operator $Y_X(\iota_A(\cdot),x)$. Then $\rep^0 A$ has vertex and braided tensor category structures as described in Section \ref{sec:vertex_tensor}.
\end{thm}

We conclude by applying Theorem \ref{thm:Rep0A} to several vertex operator algebras $A$:
\begin{exam}
The lattice vertex operator algebra $V_L$ with $L=\sqrt{2}\ZZ$ and its irreducible module $V_{\frac{1}{\sqrt{2}}+L}$ are objects in the Virasoro direct limit completion $\ind(\cC^1_1)$ at central charge $1$. Since $V_L$ is a strongly rational vertex operator algebra, Theorem \ref{thm:Rep0A} does not give any new information about $V_L$. However, Theorem \ref{thm:Rep0A} does imply that any intermediate subalgebra $L(1,0)\subseteq A\subseteq V_L$ has a braided tensor category of generalized modules that includes at least those $A$-modules occurring in the decomposition of $V_{\frac{1}{\sqrt{2}}\ZZ}$ as an $A$-module.

 For example, consider the fixed-point subalgebras of automorphism groups of $V_L$. Since $\mathrm{Aut}(V_L)$ contains $SO(3)$, all finite subgroups of $SO(3)$, especially the alternating group $A_5$, act on $V_L$. It is expected that the fixed-point subalgebra $V_L^{A_5}$ should be strongly rational, but since $A_5$ is non-abelian simple, the results of \cite{Mi_C2, CM} do not apply. In any case, Theorem \ref{thm:Rep0A} now implies that at least $V_L^{A_5}$ has a braided tensor category of modules that includes the irreducible $V_L^{A_5}$-modules occurring in the decomposition of $V_L$. By \cite[Corollary 4.8]{McR1}, these irreducible $V_L^{A_5}$-modules generate a tensor subcategory that is braided tensor equivalent to $\rep A_5$; in particular, this braided tensor subcategory is rigid.  It would be interesting to see whether this rigid symmetric tensor category of $V_L^{A_5}$-modules might be useful for proving that $V_L^{A_5}$ is $C_2$-cofinite; by \cite[Main Theorem 2]{McR2}, this would be enough for showing that $V_L^{A_5}$ is strongly rational.
\end{exam}

\begin{exam}
 For an integer $p\geq 2$, let $\cM(p)$ denote the singlet vertex operator algebra: this is a subalgebra of the rank-one Heisenberg vertex operator algebra with a modified conformal vector. It was shown in \cite{Ad} that $\cM(p)$ is an infinite direct sum of irreducible $C_1$-cofinite modules for its Virasoro subalgebra $L(c_p,0)$, where $c_p=13-6p-6p^{-1}$. Thus $\cM(p)$ is an algebra in the braided tensor category $\ind(\cC^1_{c_p})$ and we can use Theorem \ref{thm:Rep0A} to conclude that $\rep^0\cM(p)$ has vertex and braided tensor category structures.

 The braided tensor category $\rep^0\cM(p)$ is too large since it includes generalized modules with infinite-dimensional conformal weight spaces. However, it contains interesting subcategories that are closed under the tensor product. For example, Corollary \ref{cor:ImY_C1} implies that $\cC^1_{\cM(p)}\cap\rep^0\cM(p)$ is a braided monoidal subcategory of $\rep^0\cM(p)$ (it will also be an abelian category if it is closed under submodules). The tensor structure of modules in this subcategory will be studied in detail in \cite{CMY}. The existence of braided monoidal category structure on $\cC^1_{\cM(p)}\cap\rep^0\cM(p)$ partially resolves a conjecture from \cite{CMR}; however, we have not shown here that the full category $\cC^1_{\cM(p)}$ has braided tensor category structure. This is because the so-called ``typical'' irreducible $\cM(p)$-modules are $C_1$-cofinite as $\cM(p)$-modules by \cite[Theorem 16]{CMR}, but they do not decompose as sums of $C_1$-cofinite $L(c_p,0)$-modules and thus are not objects of $\ind(\cC^1_{c_p})$.
\end{exam}

\begin{exam}
The $\mathcal B_p$ algebras introduced in \cite{CRW} are extensions of the tensor product of $\cM(p)$ with a rank-one Heisenberg vertex operator algebra. They are subregular $W$-algebras of type $A$ (when $p>2$) at certain boundary admissible level, and they are also chiral algebras of Argyres-Douglas theories in physics \cite{ACGY}. The first two examples are the $\beta\gamma$ vertex algebra ($p=2$) and the affine vertex algebra $V_{-4/3}(\mathfrak{sl}_2)$ ($p=3$) \cite{Ad2}. Aside from work of Allen and Wood on the $\beta\gamma$ vertex algebra \cite{AW}, establishing rigid vertex tensor category structure beyond ordinary modules is completely open for affine vertex algebras and $W$-algebras. But now, using the previous example, we can show that there is a full rigid braided tensor category of $\mathcal B_p$-modules that contains relaxed highest-weight modules. Most of the braided tensor category structure on relaxed highest-weight modules for $\mathcal{B}_p$ algebras has been conjectured \cite{ACKR}, based on expected relationships between the representation theories of singlet algebras and of unrolled unrestricted quantum groups of $\mathfrak{sl}_2$ \cite{CGP}. Now with the existence of rigid braided tensor categories for $\mathcal{B}_p$ algebras, one can partially prove these conjectures.
\end{exam}

\begin{exam}\label{ex:svir}
For a simple Lie algebra $\mathfrak{g}$, the category $KL_k(\mathfrak{g})$ of finitely-generated grading-restricted generalized modules for the universal affine vertex operator algebra $V^k(\mathfrak g)$ at level $k$ has vertex tensor category structure if $k+h^\vee \notin \mathbb Q_{\geq 0}$ \cite{KL1}-\cite{KL5}.
Let $V^k(\lambda)$ denote the generalized Verma module of level $k$ whose lowest conformal weight space is the irreducible highest-weight $\mathfrak{g}$-module with highest weight $\lambda $. There are many interesting vertex operator (super)algebra extensions.

 Take as simplest example $\mathfrak g = \mathfrak{sl}_2$
and set $k =  \frac{2-3t}{2t-1}$, $c_t = 13 -6t-6t^{-1}$, $s = 2k+3$ (so that $2t-1=s^{-1}$), $c_s= \frac{15}{2}-3s-3s^{-1}$, $\lambda_r= (r-1)\omega$ with $\omega$ the fundamental weight of $\mathfrak{sl}_2$. Also take the simple Virasoro vertex operator algebra $L(c_t,0)$ at central charge $c_t$, which has $C_1$-cofinite irreducible modules $L(c_t,h_{r,s})$ for $r,s\in\ZZ_+$ with lowest conformal weights
\begin{equation*}
 h_{r,s} =\frac{r^2-1}{4} t-\frac{rs-1}{2}+\frac{s^2-1}{4}t^{-1}.
\end{equation*}
We use the same notation for the simple Virasoro vertex operator algebra at central charge $c_{k+2}=13-6(k+2)-6(k+2)^{-1}$.

Corollary 2.6 together with Theorem 2.10  of \cite{CGL} say that for generic $k$ (that is, as vertex algebras over a localization of the polynomial ring in $k$),
\begin{equation}\label{eqn:coset_ext}
\begin{split}
V^k(\mathfrak{osp}_{1|2}) &\cong \bigoplus_{r=1}^\infty V^k(\lambda_r) \otimes L(c_t, h_{1, r})\\
S(c_s, 0) \otimes F &\cong \bigoplus_{r=1}^\infty L(c_{k+2}, h_{1, r})\otimes L(c_t, h_{1, r})
\end{split}
\end{equation}
with $V^k(\mathfrak{osp}_{1|2})$ the universal vertex operator superalgebra of $\mathfrak{osp}_{1|2}$ at level $k$, $S(c_s, 0)$ the $N=1$ super Virasoro algebra at central charge $c_s$, and $F$ the free fermion vertex superalgebra at central charge $\frac{1}{2}$.
Thus for generic $k$, these are commutative superalgebras in the completions $\ind(KL_k(\mathfrak{sl}_2) \boxtimes \cC^1_{c_t})$, $\ind(\cC^1_{c_{k+2}} \boxtimes \cC^1_{c_t})$ of the Deligne products of the underlying categories. Assuming $k\notin\mathbb{Q}$, both $KL_k(\mathfrak{sl}_2)$ and the Virasoro categories are semisimple, so the Deligne product categories are semisimple and Theorems 5.2 and 5.5 of \cite{CKM2} imply that they have vertex tensor category structure. Moreover, the direct limit completions are semisimple; this means that their Deligne product subcategories are the same as their subcategories of $C_1$-cofinite modules, and then it follows from Corollary \ref{cor:ImY_C1} that condition (5) of Theorem \ref{thm:main_thm} is also satisfied.

We now study $S(c_s,0)$-modules using the induction functor
\begin{equation*}
 \cF: \cC^1_{c_{k+2}}\tens\cC^1_{c_t}\rightarrow\rep S(c_s,0)\otimes F
\end{equation*}
from \cite{CKM-ext}. By \cite[Proposition 2.65]{CKM-ext}, an $L(c_{k+2},0)\otimes L(c_t,0)$-module induces to a local module, in $\rep^0 S(c_s,0)\otimes F$, if and only if its monodromy with $S(c_s,0)\otimes F$ is trivial. To calculate monodromies, we use the fusion rules
\[
L(c, h_{r, 1})  \boxtimes L(c, h_{1, s}) \cong L(c, h_{r, s})
\]
from \cite[Theorem 5.2.4]{CJORY}, valid for any $c\notin\mathbb{Q}$, and the balancing equation with twist $\theta=e^{2\pi i L(0)}$:
\begin{align*}
 \cR_{L(c,h_{1,s}),L(c,h_{r,1})}\circ\cR_{L(c,h_{r,1}),L(c,h_{1,s})} & =\theta_{L(c,h_{r,s})}\circ(\theta_{L(c,h_{r,1})}^{-1}\tens\theta_{L(c,h_{1,s})}^{-1})\nonumber\\
 & = e^{2\pi i(h_{r,s}-h_{r,1}-h_{1,s})}\Id_{L(c,h_{r,s})} = e^{\pi i(r+s-rs-1)}\Id_{L(c,h_{r,s})}.
\end{align*}
From this together with \eqref{eqn:coset_ext}, we see that $L(c_{k+2}, h_{n, 1})\otimes L(c_t, h_{m, 1})$ induces to a local $S(c_s,0)\otimes F$-module if and only if $m+n$ is even. As an $L(c_{k+2}, 0)\otimes L(c_t, 0)$-module, we have
\[
\mathcal F(L(c_{k+2}, h_{n, 1})\otimes L(c_t, h_{m, 1})) \cong \bigoplus_{r=1}^\infty L(c_{k+2}, h_{n, r})\otimes L(c_t, h_{m, r}).
\]
The minimum conformal weight of this module occurs in the summand with $r=\frac{n+m}{2}$ and takes the value
\[
\Delta_{n, m} = \frac{n^2-1}{8}s + \frac{m^2-1}{8}s^{-1} -\frac{mn-1}{4}.
\]
Define $S(c_s, \Delta_{n, m})$ by
$S(c_s, \Delta_{n, m}) \otimes F=\mathcal F(L(c_{k+2}, h_{n, 1})\otimes L(c_t, h_{m, 1}))$.

Since induction is a monoidal functor, the fusion rules for the modules $S(c_s,\Delta_{n,m})$ are obtained from those of the Virasoro algebra computed in \cite{FZ}:
\[
S(c_s, \Delta_{n, m}) \boxtimes S(c_s, \Delta_{n', m'}) \cong \bigoplus_{\substack{n'' = |n-n'|+1 \\ n+n'+n'' \ \text{odd}}}^{n+n'-1}
 \bigoplus_{\substack{m'' = |m-m'|+1 \\ m+m'+m'' \ \text{odd}}}^{m+m'-1}S(c_s, \Delta_{n'', m''}).
\]
The modules $S(c_s, \Delta_{n, m})$ are simple by \cite[Proposition 4.4]{CKM-ext}, whose proof is valid for superalgebras in a direct limit completion. They are also non-isomorphic since  by Frobenius reciprocity \cite{KO, CKM-ext} we have
\begin{equation}\nonumber
\begin{split}
\delta_{n, n'}\delta_{m, m'} & \mathbb C \cong \text{Hom}_{\cC^1_{c_{k+2}} \boxtimes \cC^1_{c_t}}\left(L(c_{k+2}, h_{n, 1})\otimes L(c_t, h_{m, 1}),   L(c_{k+2}, h_{n', 1})\otimes L(c_t, h_{m', 1}) \right)\\
&\cong  \text{Hom}_{\ind(\cC^1_{c_{k+2}} \boxtimes \cC^1_{c_t})}\left(L(c_{k+2}, h_{n, 1})\otimes L(c_t, h_{m, 1}),   L(c_{k+2}, h_{n', 1})\otimes L(c_t, h_{m', 1}) \right)\\
&\cong  \text{Hom}_{\ind(\cC^1_{c_{k+2}} \boxtimes \cC^1_{c_t})}\left(L(c_{k+2}, h_{n, 1})\otimes L(c_t, h_{m, 1}),  \bigoplus_{r=1}^\infty L(c_{k+2}, h_{n', r})\otimes L(c_t, h_{m', r}) \right)\\
&\cong  \text{Hom}_{\text{Rep}\,S(c_s, 0) \otimes F}\left(\mathcal F(L(c_{k+2}, h_{n, 1})\otimes L(c_t, h_{m, 1})) ,  \mathcal F(L(c_{k+2}, h_{n', 1})\otimes L(c_t, h_{m', 1})) \right)\\
&\cong \text{Hom}_{\text{Rep}\,S(c_s, 0) \otimes F}\left(S(c_s, \Delta_{n, m}) \otimes F,  S(c_s, \Delta_{n', m'}) \otimes F \right)\\
&\cong  \text{Hom}_{\text{Rep}\,S(c_s, 0)}\left(S(c_s, \Delta_{n, m}),  S(c_s, \Delta_{n', m'})\right).
\end{split}
\end{equation}
Let $\cC_c^{1, L}$ denote the full tensor subcategory of $\cC^1_c$ whose simple objects are the $L(c, h_{n, 1})$ for $n \in  \mathbb Z_{\geq 1}$, and
let $(\cC^{1, L}_{c_{k+2}} \boxtimes \cC^{1, L}_{c_t})_0$ be the subcategory of $\cC^{1. L}_{c_{k+2}} \boxtimes \cC^{1, L}_{c_t}$  whose simple objects are the $L(c_{k+2}, h_{n, 1}) \otimes L(c_t, h_{m, 1})$ with $n+m$ even.
 Then because the $S(c_s,\Delta_{n,m})$ are simple and distinct,
\[
\mathcal F: (\cC^{1, L}_{c_{k+2}} \boxtimes \cC^{1, L}_{c_t})_0 \rightarrow \text{Rep}^0\,S(c_s, 0) \otimes F
\]
is fully faithful and so its image is braided tensor equivalent to $(\cC^{1, L}_{c_{k+2}} \otimes \cC^{1, L}_{c_t})_0$ (see \cite[Theorem 2.67]{CKM-ext}). In particular, the image is rigid by \cite[Theorem 5.5.3]{CJORY}.

Next we verify that our braided tensor category of $S(c_s,0)$-modules is non-degenerate, that is, has trivial M\"{u}ger center. If $S_{n, m} = S(c_s,\Delta_{n,m})$ is transparent, then in particular it centralizes $S(c_s,\Delta_{2,2})$. The fusion product is
\begin{equation*}
\begin{split}
S_{2, 2} \boxtimes S_{n, m} &\cong S_{n-1, m-1} \oplus  S_{n-1, m+1} \oplus S_{n+1, m-1} \oplus  S_{n+1, m+1}
\end{split}
\end{equation*}
with the convention $S_{0, m} = 0 = S_{n, 0}$.
As monodromy is determined by conformal weight due to balancing, the monodromy restricted to the summand $S_{n+\epsilon, m+\epsilon'}$ acts by the scalar
$e^{2\pi i (\Delta_{n+\epsilon, m+\epsilon'} -\Delta_{n, m} -\Delta_{2, 2} )}$. Suppose that
\[
\Delta_{n+\epsilon, m+\epsilon'} -\Delta_{n, m} -\Delta_{2, 2}  = \frac{s}{4}(n\epsilon -1) + \frac{s^{-1}}{4}(m\epsilon' -1) - \frac{1}{4}(n\epsilon' +  m\epsilon -3 + \epsilon\epsilon')
\]
is an integer $N$ for some $(n,m)\neq(1,1)$ and  $\epsilon, \epsilon' \in \{ \pm 1\}$, assuming without loss of generality $n \neq 1$. Then
\[
\Delta_{n-\epsilon, m+\epsilon'} -\Delta_{n, m} -\Delta_{2, 2}  = N - \frac{s}{2}n\epsilon  + \frac{1}{2}(  m\epsilon +\epsilon\epsilon')
\]
is clearly not an integer, so $S_{n, m}$ is not transparent and we have verified non-degeneracy.

In conclusion, we have shown that the $S(c_s,0)$-modules $S(c_s,\Delta_{m,n})$ are the simple objects of a semisimple non-degenerate rigid braided tensor category that is braided tensor equivalent to $(\cC^{1, L}_{c_{k+2}} \boxtimes \cC^{1, L}_{c_t})_0$. Note that $e^{2\pi i L(0)}$ does not quite define a twist on this category since $S(c_s,0)$ is $\frac{1}{2}\ZZ$-graded. Instead, we get a twist by setting $\theta_X=P_X e^{2\pi i L(0)}$, where $P_X$ is the parity involution on an $S(c_s,0)$-module $X$. With this twist, $\cF$ restricted to $(\cC^{1, L}_{c_{k+2}} \boxtimes \cC^{1, L}_{c_t})_0$ becomes an equivalence of braided ribbon categories.

We turn to $V^k(\mathfrak{osp}_{1|2})$ next. The idea of studying affine $\mathfrak{osp}(1|2)$ via coset extensions was used at admissible level in
\cite{CFK, CKLR};  our results here are the generic-level analogue of those in \cite{CFK}.
Using $\mathcal G$ to denote the induction functor for $V^k(\mathfrak{osp}_{1|2})$, we define
\[
M^k(n) = \mathcal G(V^k(\mathfrak{sl}_2) \otimes L(c_t, h_{n, 1})) \cong \bigoplus_{r=1}^\infty V^k(\lambda_r) \otimes L(c_t, h_{n, r}).
\]
This is a local module, in $\rep^0 V^k(\mathfrak{osp}_{1\vert 2})$, if and only if $n$ is odd. The summands for $r = \frac{n \pm 1}{2}$ have lowest conformal weight, equal to $\frac{n^2-1}{8s}$; the highest $\mathfrak{sl}_2$-weight in these top spaces is $\frac{n-1}{2}\omega$. The $M^k(n)$ are simple by \cite[Proposition 4.4]{CKM-ext}, and since induction is monoidal the fusion rules are
\[
M^k(n) \boxtimes M^k(n') \cong  \bigoplus_{\substack{n'' = |n-n'|+1 \\ n+n'+n'' \ \text{odd}}}^{n+n'-1} M^k(n'').
\]
The subcategory $(\cC^{1, L}_{c_t})_0\subseteq\cC^{1, L}_{c_t}$  with simple objects $L(c_t, h_{n, 1})$, $n$ odd, embeds as a  braided tensor subcategory of $KL_k(\mathfrak{sl}_2) \boxtimes \cC^1_{c_t}$ via $L(c_t, h_{n, 1})\mapsto V^k(\mathfrak{sl}_2) \otimes L(c_t, h_{n, 1})$.
Then by Frobenius reciprocity as in the $N=1$ super Virasoro case, the restriction of $\mathcal G$ to $(\cC^{1, L}_{c_t})_0$ is fully faithful, so its image is braided tensor equivalent to $(\cC^{1, L}_{c_t})_0$ and is in particular rigid. Non-degeneracy is easily verified as before, and in this case $e^{2\pi i L(0)}$ does define a twist on our category of $V^k(\mathfrak{osp}_{1\vert 2})$-modules. We conclude that the $V^k(\mathfrak{osp}_{1\vert 2})$-modules $M^k(n)$ for $n$ odd are the simple objects of a semisimple non-degenerate braided ribbon category equivalent to $(\cC^{1, L}_{c_t})_0$ as a braided ribbon category.
\end{exam}

\end{document}